\documentclass[11pt]{amsart}

\usepackage{a4wide, amsmath, amsfonts, amssymb, mathrsfs, amsthm, breqn}
\usepackage[hyperfootnotes=false]{hyperref}

\allowdisplaybreaks[2]

\numberwithin{equation}{section}

\newtheorem{theorem}{Theorem}[section]
\newtheorem{lemma}[theorem]{Lemma}
\newtheorem{corollary}[theorem]{Corollary}
\newtheorem{remark}[theorem]{Remark}
\newtheorem{proposition}[theorem]{Proposition}
\newtheorem*{property}{Property}

\newcommand{\dd}{\,\mathrm{d}}
\renewcommand{\d}{\mathrm{d}}
\newcommand{\D}{\mathrm{D}}
\renewcommand{\epsilon}{\varepsilon}
\renewcommand{\phi}{\varphi}
\newcommand{\R}{\mathbb{R}}
\newcommand{\N}{\mathbb{N}}
\newcommand{\E}{\mathbb{E}}
\newcommand{\W}{\mathbb{W}}
\newcommand{\X}{\mathbb{X}}
\renewcommand{\P}{\mathbb{P}}
\newcommand{\1}{\mathbf{1}}
\newcommand{\bX}{\mathbf{X}}
\newcommand{\bW}{\mathbf{W}}
\newcommand{\cP}{\mathcal{P}}
\newcommand{\cB}{\mathcal{B}}
\newcommand{\cD}{\mathcal{D}}
\newcommand{\cF}{\mathcal{F}}
\newcommand{\cL}{\mathcal{L}}
\newcommand{\cM}{\mathcal{M}}
\newcommand{\cV}{\mathcal{V}}
\newcommand{\tY}{\widetilde{Y}}
\newcommand{\tA}{\widetilde{A}}
\newcommand{\tH}{\widetilde{H}}
\newcommand{\tX}{\widetilde{X}}
\newcommand{\tbbX}{\widetilde{\mathbb{X}}}
\newcommand{\tbX}{\widetilde{\mathbf{X}}}
\newcommand{\ty}{\widetilde{y}}
\newcommand{\tcP}{\widetilde{\mathcal{P}}}

\title[The Euler scheme for RDEs and SDEs]{Pathwise convergence of the Euler scheme for rough and stochastic differential equations}

\author[Allan]{Andrew L. Allan}
\address{Andrew L. Allan, Durham University, United Kingdom}
\email{andrew.l.allan@durham.ac.uk}

\author[Kwossek]{Anna P. Kwossek}
\address{Anna P. Kwossek, University of Vienna, Austria}
\email{anna.paula.kwossek@univie.ac.at}

\author[Liu]{Chong Liu}
\address{Chong Liu, ShanghaiTech University, China}
\email{liuchong@shanghaitech.edu.cn}

\author[Pr{\"o}mel]{David J. Pr{\"o}mel}
\address{David J. Pr{\"o}mel, University of Mannheim, Germany}
\email{proemel@uni-mannheim.de}

\date{\today}

\begin{document}

\begin{abstract}
  The convergence of the first order Euler scheme and an approximative variant thereof, along with convergence rates, are established for rough differential equations driven by c{\`a}dl{\`a}g paths satisfying a suitable criterion, namely the so-called Property (RIE), along time discretizations with vanishing mesh size. This property is then verified for almost all sample paths of Brownian motion, It{\^o} processes, L{\'e}vy processes and general c{\`a}dl{\`a}g semimartingales, as well as the driving signals of both mixed and rough stochastic differential equations, relative to various time discretizations. Consequently, we obtain pathwise convergence in $p$-variation of the Euler--Maruyama scheme for stochastic differential equations driven by these processes.
\end{abstract}

\maketitle




\section{Introduction}

Stochastic differential equations serve as models for dynamical systems which evolve randomly in time, and are fundamental mathematical objects, essential to numerous applications in finance, engineering, biology and beyond. In a fairly general form, a stochastic differential equation (SDE) is given by
\begin{equation}\label{eq:intro RDE Y}
  Y_t = y_0 + \int_0^t b(s,Y_s) \dd s + \int_0^t \sigma(s,Y_s) \dd X_s, \qquad t \in [0,T],
\end{equation}
where $y_0 \in \R^k$ is the initial condition, $b \colon [0,T] \times \R^k \to \R^k$ and $\sigma \colon [0,T] \times \R^k \to \R^{k \times d}$ are coefficients, and the driving signal $X = (X_t)_{t \in [0,T]}$ is a $d$-dimensional stochastic process which models the random noise affecting the system.

Assuming that $X$ is a c\`adl\`ag semimartingale, such as a Brownian motion or a L{\'e}vy process, and the coefficients $b, \sigma$ are suitably regular, it is well known that \eqref{eq:intro RDE Y} is well-posed as an It{\^o} SDE. That is, $\int_0^t \sigma(s,Y_s) \dd X_s$ can be defined as a stochastic It{\^o} integral, and the equation admits a unique adapted solution $Y = (Y_t)_{t \in [0,T]}$; see, e.g., \cite{Protter2005}. Unfortunately, such SDEs, including many of those which appear in practical applications, can rarely be solved explicitly, which has led to a vast literature on various numerical approximations of the solutions to SDEs; see, e.g., \cite{Kloeden1992}.

One of the most common approaches to numerically approximate the solution of an SDE is to rely on a time-discretized modification of the equation. This type of discretization is implemented in particular by the Euler scheme (also called the Euler--Maruyama scheme) and its higher order variants. For the SDE \eqref{eq:intro RDE Y}, the (first order) Euler approximation is defined by
\begin{equation}\label{eq:into Euler scheme}
  Y^n_t = y_0 + \sum_{i \hspace{1pt} : \hspace{1pt} t^n_{i+1} \leq t} b(t^n_i,Y^n_{t^n_i}) (t^n_{i+1} - t^n_i) + \sum_{i \hspace{1pt} : \hspace{1pt} t^n_{i+1} \leq t} \sigma(t^n_i,Y^n_{t^n_i}) (X_{t^n_{i+1}} - X_{t^n_i}),
\end{equation}
for $t \in [0,T]$, along a sequence of partitions $\cP^n = \{0 = t^n_0 < t^n_1 < \cdots < t^n_{N_n} = T\}$. Higher order Euler approximations, such as the Milstein scheme, introduce additional higher order correction terms in the approximation \eqref{eq:into Euler scheme}, which often involve iterated integrals of the driving signal $X$. In general, the numerical calculation of the approximation $Y^n$ is carried out path by path, motivating a pathwise convergence analysis of the Euler scheme and its higher order variants. Indeed, it is well known that, for SDEs driven by Brownian motion, the (higher-order) Euler approximations converge pathwise; see, e.g., \cite{Bichteler1981,Karandikar1991,Gyongy1998,Kloeden2007,Shardlow2016}.

A fully pathwise solution theory for SDEs like \eqref{eq:intro RDE Y} is provided by the theory of rough paths; see, e.g., \cite{Friz2020,Friz2010}. Loosely speaking, in our context, a rough path is pair $\bX = (X,\X)$, consisting of a deterministic c{\`a}dl{\`a}g $\R^d$-valued path $X$, and a two-parameter c{\`a}dl{\`a}g $\R^{d \times d}$-valued function $\X$, which satisfy certain analytic and algebraic conditions. We will work with c{\`a}dl{\`a}g rough paths with finite $p$-variation, in the regime with $p \in (2,3)$, which includes in particular almost any sample path of a general semimartingale $X$, in which case the corresponding rough path $\bX = (X,\X)$ is given by $\X_{s,t} = \int_s^t (X_{r-} - X_s) \otimes \d X_r$ via stochastic integration.

Replacing the stochastic driving signal $X$ in \eqref{eq:intro RDE Y} by a (deterministic) rough path $\bX = (X,\X)$, we obtain a so-called rough differential equation (RDE). Assuming sufficient regularity of the coefficients $b, \sigma$, the RDE \eqref{eq:intro RDE Y} driven by a given c{\`a}dl{\`a}g rough path $\bX = (X,\X)$ is well-posed, in the sense that $\int_0^t \sigma(s,Y_s) \dd \bX_s$ is defined as a rough integral, and the equation admits a unique solution $Y = (Y_t)_{t \in [0,T]}$; see \cite{Friz2018}. Moreover, if the rough path is, say, the It\^o lift of a semimartingale $X$, then the solution of the resulting random RDE is consistent with the solution of the corresponding SDE driven by $X$. Both interpretations of the equation are thus essentially equivalent. Furthermore, in contrast to classical SDE theory, rough path theory is not limited to the semimartingale setting, and it comes with powerful pathwise stability estimates.

Rough path theory is intrinsically linked to the numerical approximation of SDEs---see, e.g., \cite{Davie2008,Bailleul2015}---and provides a transparent explanation for the pathwise convergence of higher order Euler approximations and their modifications; see, e.g., \cite{Friz2008,Friz2010,Deya2012a,Friz2018,Liu2019}. More precisely, the existence of a rough path lift of the driving signal is a sufficient condition for the pathwise convergence of higher order Euler schemes for RDEs, thus implying pathwise convergence for the corresponding SDEs driven by, e.g., semimartingales. However, the pathwise convergence of the first order Euler scheme---the most prominent numerical scheme for differential equations---cannot be explained by the rough path lift of the driving signal. Moreover, in general, an Euler approximation cannot converge to the solution of an RDE driven by an arbitrary rough path, for at least two reasons: First, the Euler approximation for an SDE driven by a fractional Brownian motion with Hurst index $H < \frac{1}{2}$ fails to converge (see, e.g., \cite{Deya2012a}), and second, while the rough path lift $\bX = (X,\X)$ of a path $X$ is not unique, leading to potentially multiple solutions of the RDE, the Euler approximation $Y^n$ as defined in \eqref{eq:into Euler scheme} is independent of the choice of rough path, and can thus only converge to at most one such solution.

In the present paper we clarify the gap between rough and stochastic differential equations from the perspective of numerical approximation, by establishing the convergence of the first order Euler scheme for RDEs driven by It\^o-type rough path lifts. More precisely, in Theorem~\ref{thm: Euler scheme convergence} we obtain convergence in $p$-variation of the Euler scheme for RDEs driven by c{\`a}dl{\`a}g paths satisfying a suitable criterion---namely the so-called Property \textup{(RIE)}---relative to a sequence of partitions with vanishing mesh size.

Property \textup{(RIE)} was first introduced in \cite{Perkowski2016} and \cite{Allan2023b}, motivated by applications in mathematical finance under model uncertainty. While, strictly speaking, it is a condition on a c{\`a}dl{\`a}g path $X \colon [0,T] \to \R^d$, it always ensures the existence of an It\^o-type rough path lift $\bX = (X,\X)$, allowing one to treat \eqref{eq:intro RDE Y} as an RDE. Using this fact, we will show that Property \textup{(RIE)} is a sufficient condition on the sample paths of a stochastic driving signal to guarantee the convergence of the first order Euler scheme for the corresponding SDE. We note in particular that the Euler scheme converges surely on the set where the stochastic driving signal satisfies Property \textup{(RIE)}, which is a stronger statement compared to the earlier results in \cite{Bichteler1981,Karandikar1991,Gyongy1998,Kloeden2007,Shardlow2016}, in which the set on which the Euler scheme converges can depend on the coefficients $b, \sigma$. A criterion for H{\"o}lder continuous rough paths, related to Property \textup{(RIE)}, was previously introduced by Davie \cite{Davie2008}, which also allows one to obtain convergence of the Euler scheme for RDEs, and will be discussed in more detail in Remark~\ref{rem: Davie's criterion}.

Exploiting the continuity results of rough path theory, in Theorem~\ref{thm: Euler scheme convergence} we derive a precise error estimate in $p$-variation for the Euler approximation of RDEs with respect to the discretization error of the driving signal. The convergence rate is expressed transparently, in terms of the mesh size of the approximating partition, and the approximation error of the discretized signal and of its rough path lift. We also obtain an error estimate for the Euler approximation with respect to pathwise perturbations of the driving signal; see Proposition~\ref{prop:approximate Euler scheme}. This latter perturbation is motivated by so-called approximate Euler schemes for SDEs driven by jump processes; see, e.g., \cite{Jacod2005,Rubenthaler2003,Dereich2011}. For instance, approximate Euler schemes are used for L{\'e}vy-driven SDEs, since the increments of L{\'e}vy processes cannot always be simulated, and thus the increments of the driving L{\'e}vy process need to be approximated by random variables with known distributions.

To obtain pathwise convergence of the Euler scheme in $p$-variation for an SDE, it is then sufficient to verify that the associated stochastic driving signal of the equation satisfies Property \textup{(RIE)}, almost surely, relative to a sequence of partitions; see Sections~\ref{sec:SDE} and \ref{sec:RSDE}. Unsurprisingly, we find that the more regular the driving signal is, the more general the sequence of partitions may be chosen. Indeed, while the sample paths of a Brownian motion satisfy Property \textup{(RIE)}, almost surely, relative to sequences of partitions whose mesh size can converge to zero very slowly, the sample paths of more general It{\^o} processes satisfy Property \textup{(RIE)}, almost surely, relative to sequences of partitions whose mesh size is of order $2^{-n}$. For stochastic processes with jumps, such as L{\'e}vy processes or general c{\`a}dl{\`a}g semimartingales, one needs to ensure that the jump times are exhausted by the sequence of partitions, which is a necessary condition, for both the Euler scheme to converge pathwise, and for Property \textup{(RIE)} to be satisfied by the driving signal.

The presented pathwise analysis of the first order Euler approximation is not limited to SDEs in a semimartingale setting. As examples, we consider mixed SDEs driven by both Brownian motion and fractional Brownian motion with Hurst index $H > \frac{1}{2}$, as in, e.g., \cite{Zahle2001,Mishura2011}, as well as rough SDEs, which are differential equations driven by both a rough path and a Brownian motion; see \cite{Friz2021}. The latter equations are of interest, e.g., in the context of robust stochastic filtering; see \cite{Crisan2013,Diehl2015}.

\medskip

\noindent \textbf{Organization of the paper:} In Section~\ref{sec:RDE} we prove the convergence of the Euler scheme for RDEs assuming that the driving paths satisfy Property \textup{(RIE)}. In Sections~\ref{sec:SDE} and \ref{sec:RSDE} we provide various examples of stochastic processes which satisfy Property \textup{(RIE)} along suitable sequences of partitions, making the established convergence analysis applicable to the corresponding SDEs, and derive associated convergence rates.

\medskip

\noindent\textbf{Acknowledgments:} A.~P.~Kwossek and D.~J.~Pr{\"o}mel gratefully acknowledge financial support by the Baden-W{\"u}rttemberg Stiftung, and would like to thank A.~Neuenkirch for fruitful discussions which helped to improve the present work. A.~P.~Kwossek was affiliated with the University of Mannheim for the majority of this project's duration.

\section{The Euler scheme for rough differential equations}\label{sec:RDE}

In this section we study convergence of the (first order) Euler scheme for RDEs, which does not rely on the L{\'e}vy area of the path, and is known to converge pathwise for certain classes of SDEs. Before treating the Euler scheme, we will first recall some essentials from the theory of c{\`a}dl{\`a}g rough paths, as introduced in \cite{Friz2017,Friz2018}.

\subsection{Essentials on rough path theory}

A \emph{partition} $\mathcal{P}$ of an interval $[s,t]$ is a finite set of points between and including the points $s$ and $t$, i.e., $\mathcal{P} = \{s = u_0 < u_1 < \cdots < u_N = t\}$ for some $N \in \N$, and its mesh size is denoted by $|\mathcal{P}| := \max\{|u_{i+1} - u_i| \, : \, i = 0, \ldots, N-1\}$. A sequence $(\cP^n)_{n \in \N}$ of partitions is said to be \emph{nested}, if $\cP^n \subset \cP^{n+1}$ for all $n \in \N$.

Throughout, we let $T > 0$ be a fixed finite time horizon. We let $\Delta_T := \{(s,t) \in [0,T]^2 \, : \, s \leq t\}$ denote the standard $2$-simplex. A function $w \colon \Delta_T \to [0,\infty)$ is called a \emph{control function} if it is superadditive, in the sense that $w(s,u) + w(u,t) \leq w(s,t)$ for all $0 \leq s \leq u \leq t \leq T$. For two vectors $x = (x^1, \ldots, x^d), y = (y^1, \ldots, y^d) \in \R^d$ we use the usual tensor product
\begin{equation*}
  x \otimes y := (x^i y^j)_{i, j = 1, \ldots, d} \in \R^{d \times d}.
\end{equation*}
Whenever $(B,\|\cdot\|)$ is a normed space and $f, g \colon B \to \R$ are two functions on $B$, we shall write $f \lesssim g$ or $f \leq C g$ to mean that there exists a constant $C > 0$ such that $f(x) \leq C g(x)$ for all $x \in B$. The constant $C$ may depend on the normed space, e.g., through its dimension or regularity parameters.

\medskip

The space of linear maps from $\R^d \to \R^n$ is denoted by $\mathcal{L}(\R^d;\R^n)$, and we write, e.g., $C_b^k(\R^m;\mathcal{L}(\R^d;\R^n))$ for the space of $k$-times differentiable (in the Fr{\'e}chet sense) functions $f \colon \R^m \to \mathcal{L}(\R^d;\R^n)$ such that $f$ and all its derivatives up to order $k$ are continuous and bounded. We equip this space with the norm
\begin{equation*}
  \|f\|_{C^k_b} := \|f\|_\infty + \|\D f\|_\infty + \cdots + \|\D^k f\|_\infty,
\end{equation*}
where $\D^r f$ denotes the $r$-th order derivative of $f$, and $\|\hspace{0.5pt}\cdot\hspace{0.5pt}\|_{\infty}$ denotes the supremum norm on the corresponding spaces of operators.

\medskip

For a normed space $(E,|\cdot|)$, we let $D([0,T];E)$ denote the set of c{\`a}dl{\`a}g (right-continuous with left-limits) paths from $[0,T]$ to $E$. For $X \in D([0,T];E)$, the supremum norm of the path $X$ is given by
\begin{equation*}
\|X\|_{\infty} := \sup_{t \in [0,T]} |X_t|,
\end{equation*}
and, for $p \geq 1$, the $p$-variation of the path $X$ is given by
\begin{equation*}
  \|X\|_p := \|X\|_{p,[0,T]} \qquad \text{with} \qquad \|X\|_{p,[s,t]} := \bigg(\sup_{\mathcal{P}\subset[s,t]} \sum_{[u,v]\in \mathcal{P}} |X_v - X_u|^p \bigg)^{\frac{1}{p}}, \quad (s,t) \in \Delta_T,
\end{equation*}
where the supremum is taken over all possible partitions $\mathcal{P}$ of the interval $[s,t]$. We recall that, given a path $X$, we have that $\|X\|_p < \infty$ if and only if there exists a control function $w$ such that\footnote{Here and throughout, we adopt the convention that $\frac{0}{0} := 0$.}
\begin{equation*}
  \sup_{(u,v) \in \Delta_T} \frac{|X_v - X_u|^p}{w(u,v)} < \infty.
\end{equation*}
We write $D^p = D^p([0,T];E)$ for the space of paths $X \in D([0,T];E)$ which satisfy $\|X\|_p < \infty$. Moreover, for a path $X \in D([0,T];\R^d)$, we will often use the shorthand notation:
\begin{equation*}
  X_{s,t} := X_t - X_s
  \quad \text{and} \quad
  X_{t-}:= \lim_{u \nearrow t} X_u, \qquad \text{for} \quad (s,t) \in \Delta_T.
\end{equation*}

For $r \geq 1$ and a two-parameter function $\mathbb{X} \colon \Delta_T \to E$, we similarly define
\begin{equation*}
  \|\mathbb{X}\|_r := \|\mathbb{X}\|_{r,[0,T]} \qquad \text{with} \qquad \|\mathbb{X}\|_{r,[s,t]} := \bigg(\sup_{\mathcal{P} \subset [s,t]} \sum_{[u,v] \in \mathcal{P}} |\mathbb{X}_{u,v}|^r\bigg)^{\frac{1}{r}}, \quad (s,t) \in \Delta_T.
\end{equation*}
We write $D_2^r = D_2^r(\Delta_T;E)$ for the space of all functions $\mathbb{X} \colon \Delta_T \to E$ which satisfy $\|\X\|_r < \infty$, and are such that the maps $s \mapsto \mathbb{X}_{s,t}$ for fixed $t$, and $t \mapsto \mathbb{X}_{s,t}$ for fixed $s$, are both c{\`a}dl{\`a}g.

\medskip

For $p \in [2,3)$, a pair $\bX = (X,\X)$ is called a \emph{c{\`a}dl{\`a}g $p$-rough path} over $\R^d$ if
\begin{enumerate}
  \item[(i)] $X \in D^p([0,T];\R^d)$ and $\X \in D_2^{\frac{p}{2}}(\Delta_T;\R^{d \times d})$, and
  \item[(ii)] Chen's relation: $\mathbb{X}_{s,t} = \mathbb{X}_{s,u} + \mathbb{X}_{u,t} + X_{s,u} \otimes X_{u,t}$ holds for all $0 \leq s \leq u \leq t \leq T$.
\end{enumerate}
In component form, condition (ii) states that $\mathbb{X}^{ij}_{s,t} = \mathbb{X}^{ij}_{s,u} + \mathbb{X}^{ij}_{u,t} + X^i_{s,u} X^j_{u,t}$ for every $i$ and $j$. We will denote the space of c{\`a}dl{\`a}g $p$-rough paths by $\cD^p = \cD^p([0,T];\R^d)$. On the space $\cD^p([0,T];\R^d)$, we use the natural seminorm
\begin{equation*}
  \|\bX\|_{p} := \|\bX\|_{p,[0,T]} \qquad \text{with} \qquad \|\bX\|_{p,[s,t]} := \|X\|_{p,[s,t]} + \|\X\|_{\frac{p}{2},[s,t]}
\end{equation*}
for $(s,t) \in \Delta_T$, and the induced distance
\begin{equation}\label{eq:rough path distance} 
  \|\bX;\tbX\|_p :=  \|\bX;\tbX\|_{p,[0,T]} \qquad \text{with} \qquad \|\bX;\tbX\|_{p,[s,t]} := \|X - \tX\|_{p,[s,t]} + \|\X - \tbbX\|_{\frac{p}{2},[s,t]},
\end{equation}
whenever $\bX = (X,\X), \tbX = (\tX,\tbbX) \in \cD^p([0,T];\R^d)$.

\medskip

Let $p \in [2,3)$, $q \in [p,\infty)$ and $r \in [\frac{p}{2},2)$ such that $\frac{1}{p} + \frac{1}{r} > 1$ and $\frac{1}{p} + \frac{1}{q} = \frac{1}{r}$. Let $X \in D^p([0,T];\R^d)$. We say that a pair $(Y,Y')$ is a \emph{controlled path} (with respect to $X$), if
\begin{equation*}
  Y \in D^p([0,T];E), \quad Y' \in D^q([0,T];\cL(\R^d;E)), \quad \text{and} \quad R^Y \in D^r_2(\Delta_T;E),
\end{equation*}
where $R^Y$ is defined by
\begin{equation*}
  Y_{s,t} = Y'_s X_{s,t} + R^Y_{s,t} \qquad \text{for all} \quad (s,t) \in \Delta_T.
\end{equation*}
We write $\cV^{q,r}_X = \cV^{q,r}_X([0,T];E)$ for the space of $E$-valued controlled paths, which becomes a Banach space when equipped with the norm
\begin{equation*}
  (Y,Y') \mapsto |Y_0| + |Y'_0| + \|Y'\|_{q,[0,T]} + \|R^Y\|_{r,[0,T]}.
\end{equation*}

\begin{remark}
The definition of a controlled path adopted here is slightly more general than the classical definition in, e.g., \cite{Friz2018}, in which one takes $q = p$ and $r = \frac{p}{2}$. Allowing these regularity parameters to take larger values allows us to consider slightly more general integrands in rough integrals. In particular, this is convenient in Theorem~\ref{thm: RDE} below, as otherwise we would require further restrictions on the regularity of the paths $A$ and $H$ therein.
\end{remark}

For paths $A \in D^{q_1}$, $H \in D^{q_2}$ for $q_1, q_2 \in [1,2)$, and a rough path $\bX \in \cD^p$ for $p \in [2,3)$, we consider the rough differential equation (RDE):
\begin{equation}\label{eq:general RDE}
  Y_t = y_0 + \int_0^t b(H_s,Y_s) \dd A_s + \int_0^t \sigma(H_s,Y_s) \dd \bX_s, \qquad t \in [0,T].
\end{equation}
Provided that $\frac{1}{p} + \frac{1}{q_1} > 1$ and $\frac{1}{p} + \frac{1}{q_2} > 1$, the first integral in this equation can be defined as a Young integral, whilst the second integral is defined as a rough integral. For precise definitions, constructions and properties of these integrals, we refer to the comprehensive exposition in \cite{Friz2018}.

\begin{theorem}\label{thm: RDE}
Let $p \in [2,3)$ and $q_1, q_2 \in [1,2)$ such that $\frac{1}{p} + \frac{1}{q_1} > 1$ and $\frac{1}{p} + \frac{1}{q_2} > 1$. Let $b \in C^2_b(\R^{m+k};\cL(\R^n;\R^k))$, $\sigma \in C^3_b(\R^{m+k};\cL(\R^d;\R^k))$, $y_0 \in \R^k$, $A \in D^{q_1}([0,T];\R^n)$, $H \in D^{q_2}([0,T];\R^m)$ and $\bX = (X,\X) \in \cD^p([0,T];\R^d)$. Let $r \in [\frac{p}{2} \vee q_1 \vee q_2,2)$ such that $\frac{1}{p} + \frac{1}{r} > 1$, and let $q \in [p,\infty)$ such that $\frac{1}{p} + \frac{1}{q} = \frac{1}{r}$. Then there exists a unique path $Y \in D^p([0,T];\R^k)$ such that the controlled path $(Y,\sigma(H,Y)) \in \cV_X^{q,r}$ satisfies the RDE~\eqref{eq:general RDE}.

  Moreover, if $\ty_0 \in \R^k$, $\tA \in D^{q_1}$, $\tH \in D^{q_2}$ and $\tbX = (\tX,\tbbX) \in \cD^p$ with corresponding solution $\tY$, and if $\|A\|_{r}, \|\tA\|_{r}, \|H\|_{r}, \|\tH\|_{r}, \|\bX\|_{p}, \|\tbX\|_{p} \leq L$ for some $L > 0$, then
  \begin{equation}\label{eq:estimate for RDE}
    \begin{split}
    &\|Y - \tY\|_p + \|Y' - \tY'\|_q + \|R^Y - R^{\tY}\|_r\\
    &\quad\lesssim |y_0 - \ty_0| + |H_0 - \tH_0| + \|H - \tH\|_r + \|A - \tA\|_r + \|\bX;\tbX\|_p,
    \end{split}
  \end{equation}
  where the implicit multiplicative constant depends only on $p, q, r, \|b\|_{C_b^2}, \|\sigma\|_{C_b^3}$ and $L$.
\end{theorem}

The result of Theorem~\ref{thm: RDE} may be considered classical, and will be unsurprising to readers familiar with RDEs. However, to the best of our knowledge, a proof of the precise statement of the theorem does not appear in the existing literature. A sketch of the proof, based on the proof of \cite[Theorem~2.3]{Allan2021a}, is therefore given in Appendix~\ref{appendix proof of RDE thm}.

\subsection{Convergence of the Euler scheme}

Let us consider the RDE
\begin{equation}\label{eq:RDE Y}
  Y_t = y_0 + \int_0^t b(s,Y_s) \dd s + \int_0^t \sigma(s,Y_s) \dd \bX_s, \qquad t \in [0,T],
\end{equation}
where $y_0 \in \R^k$, $b \in C_b^2(\R^{k+1};\R^k)$, $\sigma \in C^3_b(\R^{k+1};\mathcal{L}(\R^d;\R^k))$ and $\bX = (X,\X) \in \cD^p([0,T];\R^d)$ is the driving c{\`a}dl{\`a}g $p$-rough path for $p \in [2,3)$. Given a sequence of partitions $\cP^n = \{0 = t^n_0 < t^n_1 < \cdots < t^n_{N_n} = T\}$, $n \in \N$, the Euler approximation $Y^n$ corresponding to the RDE \eqref{eq:RDE Y} along the partition $\cP^n$ is given by
\begin{equation}\label{eq:RDE Euler scheme}
  Y^n_t = y_0 + \sum_{i \hspace{1pt} : \hspace{1pt} t^n_{i+1} \leq t} b(t^n_i,Y^n_{t^n_i}) (t^n_{i+1} - t^n_i) + \sum_{i \hspace{1pt} : \hspace{1pt} t^n_{i+1} \leq t} \sigma(t^n_i,Y^n_{t^n_i}) (X_{t^n_{i+1}} - X_{t^n_i}),
\end{equation}
for $t \in [0,T]$.

\medskip

It is a classical result in the numerical analysis of SDEs that, if the driving signal is, e.g., a Brownian motion, then the Euler scheme (often also called the Euler--Maruyama scheme) converges pathwise; see, e.g., \cite{Kloeden2007}. On the other hand, it is known that in general the Euler scheme cannot converge if the driving signal is an arbitrary rough path, since the corresponding Euler scheme for SDEs driven by fractional Brownian motion fails to converge; see \cite{Deya2012a} for a more detailed discussion on this observation.

Moreover, since the extension of a path $X$ to a rough path $\bX = (X,\mathbb{X})$ is not unique, and the Euler approximation $Y^n$ defined in \eqref{eq:RDE Euler scheme} is independent of $\mathbb{X}$, the sequence $(Y^n)_{n \in \N}$ cannot converge to the solution of a general RDE. Thus, in order to ensure the convergence of the Euler scheme, it is necessary to identify the ``correct'' rough path lift $\bX$ as the driving signal for the RDE~\eqref{eq:RDE Y}. A suitable resolution to this is provided by the so-called Property \textup{(RIE)}, as introduced in \cite{Perkowski2016} and \cite{Allan2023b}.

\begin{property}[\textbf{RIE}]
  Let $p \in (2,3)$ and let $\cP^n = \{0 = t^n_0 < t^n_1 < \cdots < t^n_{N_n} = T\}$, $n \in \N$, be a sequence of partitions of the interval $[0,T]$ such that $|\mathcal{P}^n| \to 0$ as $n \to \infty$. For $X \in D([0,T];\R^d)$, and each $n \in \N$, we define $X^n \colon [0,T] \to \R^d$ by
  \begin{equation*}
    X^n_t = X_T \1_{\{T\}}(t) + \sum_{k=0}^{N_n - 1} X_{t^n_k} \1_{[t^n_k,t^n_{k+1})}(t), \qquad t \in [0,T].
  \end{equation*}
  We assume that:
  \begin{enumerate}
    \item[(i)] the sequence of paths $(X^n)_{n \in \N}$ converges uniformly to $X$ as $n \to \infty$,
    \item[(ii)] the Riemann sums
    \[\int_0^t X^n_u \otimes \d X_u := \sum_{k=0}^{N_n-1} X_{t^n_k} \otimes X_{t^n_k \wedge t,t^n_{k+1} \wedge t}\]
    converge uniformly as $n \to \infty$ to a limit, which we denote by $\int_0^t X_u \otimes \d X_u$, $t \in [0,T]$,
    \item[(iii)] and there exists a control function $w$ such that
    \begin{equation}\label{eq:RIE inequality}
      \sup_{(s,t) \in \Delta_T} \frac{|X_{s,t}|^p}{w(s,t)} + \sup_{n \in \N} \, \sup_{0 \leq k < \ell \leq N_n} \frac{|\int_{t^n_k}^{t^n_\ell} X^n_u \otimes \d X_u - X_{t^n_k} \otimes X_{t^n_k,t^n_\ell}|^{\frac{p}{2}}}{w(t^n_k,t^n_\ell)} \leq 1.
    \end{equation}
  \end{enumerate}
\end{property}

We say that a path $X \in D([0,T];\R^d)$ satisfies Property \textup{(RIE)} relative to $p$ and $(\mathcal{P}^n)_{n \in \N}$, if $p$, $(\mathcal{P}^n)_{n \in \N}$ and $X$ together satisfy Property \textup{(RIE)}.

\medskip

It is known that, if a path $X \in D([0,T];\R^d)$ satisfies Property \textup{(RIE)}, then $X$ extends canonically to a rough path $\mathbf{X} = (X,\mathbb{X}) \in \mathcal{D}^p([0,T];\R^d)$, where the lift $\X$ is defined by
\begin{equation}\label{eq:RIE rough path}
  \mathbb{X}_{s,t} := \int_s^t X_u \otimes \d X_u - X_s \otimes (X_t - X_s), \qquad (s,t) \in \Delta_T,
\end{equation}
with $\int_s^t X_u \otimes \d X_u := \int_0^t X_u \otimes \d X_u - \int_0^s X_u \otimes \d X_u$, and the existence of the integral $\int_0^t X_u \otimes \d X_u$ is ensured by condition (ii) of Property \textup{(RIE)}; see \cite[Lemma~2.13]{Allan2023b}. When assuming Property \textup{(RIE)} for a path $X$, we will always work with the rough path $\mathbf{X} = (X,\mathbb{X})$ defined via \eqref{eq:RIE rough path}, and note that $\mathbf{X} = (X,\mathbb{X})$ corresponds to the It{\^o} rough path lift of a stochastic process, since the ``iterated integral'' $\mathbb{X}$ is given as a limit of left-point Riemann sums, analogously to the stochastic It{\^o} integral.

\medskip

Postulating Property \textup{(RIE)} for the driving signal of an RDE ensures that the (first order) Euler approximation converges to the solution of the equation, as stated precisely in the next theorem.

\begin{theorem}\label{thm: Euler scheme convergence}
  Suppose that $X \colon [0,T] \to \R^d$ satisfies Property \textup{(RIE)} relative to some $p \in (2,3)$ and a sequence of partitions $(\cP^n)_{n \in \N}$, and let $\bX$ be the canonical rough path lift of $X$, as defined in \eqref{eq:RIE rough path}. Let $Y$ be the solution to the RDE \eqref{eq:RDE Y} driven by $\bX$, and let $Y^n$ be the Euler approximation defined in \eqref{eq:RDE Euler scheme}. Then,
  \begin{equation*}
    \|Y^n - Y\|_{p'} \, \longrightarrow \, 0 \qquad \text{as} \quad n \, \longrightarrow \, \infty,
  \end{equation*}
  for any $p' \in (p,3)$, and the rate of convergence is determined by the estimate
  \begin{equation}\label{eq:est rate of convergence}
    \|Y^n - Y\|_{p'} \lesssim |\cP^n|^{1 - \frac{1}{q}} + \|X^n - X\|_{\infty}^{1 - \frac{p}{p'}} + \bigg\|\int_0^{\cdot} X^n_{u} \otimes \d X_u - \int_0^{\cdot} X_{u} \otimes \d X_u\bigg\|_{\infty}^{1 - \frac{p}{p'}},
  \end{equation}
 which holds for any $q \in (1,2)$ such that $\frac{1}{p'} + \frac{1}{q} > 1$, where the implicit multiplicative constant depends only on $p, p', q, \|b\|_{C_b^2}, \|\sigma\|_{C_b^3}, T, |X_0|$ and $w(0,T)$, where $w$ is the control function for which \eqref{eq:RIE inequality} holds.
\end{theorem}

Note that Property \textup{(RIE)} implies that each of the terms on the right-hand side of \eqref{eq:est rate of convergence} tends to zero as $n \to \infty$.

\begin{remark}\label{rem: Davie's criterion}
  In \cite{Davie2008}, A.~M.~Davie observed that, under suitable conditions, the first order Euler scheme along equidistant partitions converges to the solution of a given RDE. More precisely, for $p \in (2,3)$ and $\alpha := \frac{1}{p}$, let $\bX = (X,\X)$ be an $\alpha$-H{\"o}lder continuous rough path, so that $|X_{s,t}| \lesssim |t-s|^\alpha$ and $|\X_{s,t}| \lesssim |t-s|^{2\alpha}$ for $(s,t) \in \Delta_T$, such that, for some $\beta \in (1 - \alpha,2\alpha)$, there exists a constant $C > 0$ such that
  \begin{equation*}
    \bigg|\sum_{j=k}^{\ell-1} \X_{jh,(j+1)h}\bigg| \leq C (\ell - k)^\beta h^{2\alpha}
  \end{equation*}
  whenever $h > 0$ and $0 \leq k < \ell$ are integers such that $\ell h \leq T$. Under this condition on the driving signal $\bX$, \cite[Theorem~7.1]{Davie2008} states that the Euler approximations $Y^n$, as defined in \eqref{eq:RDE Euler scheme}, converge uniformly to the solution $Y$ of the RDE \eqref{eq:RDE Y} along the equidistant partitions $(\cP^n_{\text{U}})_{n \in \N}$, where $\cP^n_U = \{\frac{iT}{n} : i = 0, 1, \ldots, n\}$. Note that Davie's condition implies Property \textup{(RIE)}---see \cite[Appendix~B]{Perkowski2016}---and is thus less general, even in the case of H{\"o}lder continuous rough paths.
\end{remark}

\begin{remark}
  Since the ``iterated integrals'' appearing in the definition of a rough path (and in, e.g., higher order Euler schemes) are often numerically difficult to simulate, various approaches have been developed to avoid the direct involvement of iterated integrals in the approximation of stochastic and rough differential equations. For instance, \cite{Deya2012a} introduced a simplified Milstein scheme for SDEs driven by fractional Brownian motion, where the iterated integrals are replaced by products of the increments of the driving process. Using this idea, simplified Runge--Kutta methods for differential equations driven by general (continuous) rough paths were investigated in \cite{Redmann2022}; see also \cite{Hong2018}.
\end{remark}

The rest of this subsection is devoted to the proof of Theorem~\ref{thm: Euler scheme convergence}, which first requires us to establish some auxiliary results.

In the following, we will always assume that $X \colon [0,T] \to \R^d$ satisfies Property \textup{(RIE)} relative to some $p \in (2,3)$ and a sequence of partitions $(\cP^n)_{n\in \N}$. As the piecewise constant approximation $X^n$ (as defined in Property \textup{(RIE)}) has finite $1$-variation, it possesses a canonical rough path lift $\bX^n = (X^n,\X^n) \in \mathcal{D}^p([0,T];\R^d)$, with $\X^n$ given by
\begin{equation}\label{eq:defn bbX^n}
  \X^n_{s,t} := \int_s^t X^n_{s,u} \otimes \d X^n_u, \qquad (s,t) \in \Delta_T,
\end{equation}
where the integral is defined as a classical limit of left-point Riemann sums. Note that, while \cite[Section~5.3]{Friz2018} discretizes the rough path $\mathbf{X} = (X,\mathbb{X})$ in a piecewise constant manner, here we instead discretize the path $X$ and then extend it to a rough path $\bX^n = (X^n,\X^n)$ via \eqref{eq:defn bbX^n}.

As a first step towards the proof of Theorem~\ref{thm: Euler scheme convergence}, we establish the convergence of the rough paths $(\bX^n)_{n\in\N}$ to the rough path $\bX$ in a suitable rough path distance. For this purpose, we need two auxiliary lemmas.

\begin{lemma}\label{lemma: bbX^n to bbX uniformly}
  Suppose that $X \colon [0,T] \to \R^d$ satisfies Property \textup{(RIE)} relative to some $p \in (2,3)$ and a sequence of partitions $(\cP^n)_{n\in \N}$. Then, we have the estimate
  \begin{equation*}
    \sup_{(s,t) \in \Delta_{T}} |\X^n_{s,t} - \X_{s,t}| \leq 2 \|X\|_{\infty} \|X^n - X\|_{\infty} + \sup_{(s,t) \in \Delta_{T}} \bigg|\int_s^t X^n_{s,u} \otimes \d X_u - \X_{s,t}\bigg|,
  \end{equation*}
  where $\X^n$ and $\X$ were defined in \eqref{eq:defn bbX^n} and \eqref{eq:RIE rough path}, respectively. In particular, we have that
  \begin{equation*}
    \X^n \, \longrightarrow \, \X \quad \text{uniformly as} \quad n \, \longrightarrow \, \infty.
  \end{equation*}
\end{lemma}

\begin{proof}
  Since
  \begin{equation*}
    |\X^n_{s,t} - \X_{s,t}| \leq \bigg|\X^n_{s,t} - \int_s^t X^n_{s,u} \otimes \d X_u \bigg| + \bigg|\int_s^t X^n_{s,u} \otimes \d X_u - \X_{s,t}\bigg|,
  \end{equation*}
  and the limit in condition (ii) of Property \textup{(RIE)} holds uniformly, it is enough to prove that the function given by
  \begin{equation*}
    \Lambda^n_{s,t} := \X^n_{s,t} - \int_s^t X^n_{s,u} \otimes \d X_u = \int_s^t X^n_{s,u} \otimes \d (X^n - X)_u
  \end{equation*}
  satisfies
  \begin{equation}\label{eq:sup Lambda^n bound}
    \sup_{(s,t) \in \Delta_{T}} |\Lambda^n_{s,t}| \leq 2 \|X\|_{\infty} \|X^n - X\|_{\infty}.
  \end{equation}

  If $t^n_k \leq s < t \leq t^n_{k+1}$ for some $k$, then $X^n_{s,u} = X_{t^n_k,t^n_k} = 0$ for every $u \in [s,t)$, so that $\Lambda^n_{s,t} = 0$. Otherwise, let $k_0$ be the smallest $k$ such that $t^n_k \in (s,t)$, and let $k_1$ be the largest such $k$. It is straightforward to see that the triplet $(X^n - X,X^n,\Lambda^n)$ satisfies Chen's relation:
  \begin{equation*}
    \Lambda^n_{s,t} = \Lambda^n_{s,u} + \Lambda^n_{u,t} + X^n_{s,u} \otimes (X^n - X)_{u,t}
  \end{equation*}
  for all $s \leq u \leq t$, from which it follows that
  \begin{equation*}
    \Lambda^n_{s,t} = \Lambda^n_{s,t^n_{k_0}} + \Lambda^n_{t^n_{k_0},t^n_{k_1}} + \Lambda^n_{t^n_{k_1},t} + X^n_{s,t^n_{k_0}} \otimes (X^n - X)_{t^n_{k_0},t^n_{k_1}} + X^n_{s,t^n_{k_1}} \otimes (X^n - X)_{t^n_{k_1},t}.
  \end{equation*}
  As we already observed, we have that $\Lambda^n_{s,t^n_{k_0}} = \Lambda^n_{t^n_{k_1},t} = 0$. In fact, we also have that
  \begin{equation}\label{eq:Lambda^n01 = 0}
    \begin{split}
    \Lambda^n_{t^n_{k_0},t^n_{k_1}} &= \int_{t^n_{k_0}}^{t^n_{k_1}} X^n_{t^n_{k_0},u} \otimes \d (X^n - X)_u = \sum_{i=k_0}^{k_1-1} \int_{t^n_i}^{t^n_{i+1}} X^n_{t^n_{k_0},u} \otimes \d (X^n - X)_u\\
    &= \sum_{i=k_0}^{k_1-1} \int_{t^n_i}^{t^n_{i+1}} X_{t^n_{k_0},t^n_i} \otimes \d (X^n - X)_u = \sum_{i=k_0}^{k_1-1} X_{t^n_{k_0},t^n_i} \otimes (X^n - X)_{t^n_i,t^n_{i+1}} = 0.
    \end{split}
  \end{equation}
  Since $(X^n - X)_{t^n_{k_0}} = (X^n - X)_{t^n_{k_1}} = 0$, we simply obtain $\Lambda^n_{s,t} = X^n_{s,t^n_{k_1}} \otimes (X^n_t - X_t)$, from which \eqref{eq:sup Lambda^n bound} follows.
\end{proof}

\begin{lemma}\label{lemma: bbX^n p/2var estimate}
  Suppose that $X \colon [0,T] \to \R^d$ satisfies Property \textup{(RIE)} relative to some $p \in (2,3)$ and a sequence of partitions $(\cP^n)_{n\in \N}$. Let $w$ be the control function with respect to which~$X$ satisfies the inequality~\eqref{eq:RIE inequality}. Then, there exists a constant $C$, which depends only on $p$, such that
  \begin{equation}\label{eq:bbx^n p/2var bound}
    \|\X^n\|_{\frac{p}{2}} \leq C w(0,T)^{\frac{2}{p}}
  \end{equation}
  for every $n \in \N$, where $\X^n$ was defined in \eqref{eq:defn bbX^n}.
\end{lemma}

\begin{proof}
  Let $n \in \N$, and let $(s,t) \in \Delta_{T}$. If $t^n_k \leq s < t \leq t^n_{k+1}$ for some $k$, then $X^n_{s,u} = X_{t^n_k,t^n_k} = 0$ for every $u \in [s,t)$, so that $\X^n_{s,t} = 0$. Otherwise, let $k_0$ be the smallest $k$ such that $t^n_k \in (s,t)$, and let $k_1$ be the largest such $k$. It is straightforward to see that $(X^n,\X^n)$ satisfies Chen's relation:
  \begin{equation*}
    \X^n_{s,t} = \X^n_{s,u} + \X^n_{u,t} + X^n_{s,u} \otimes X^n_{u,t}
  \end{equation*}
  for all $s \leq u \leq t$, from which it follows that
  \begin{equation*}
    \X^n_{s,t} = \X^n_{s,t^n_{k_0}} + \X^n_{t^n_{k_0},t^n_{k_1}} + \X^n_{t^n_{k_1},t} + X^n_{s,t^n_{k_0}} \otimes X^n_{t^n_{k_0},t^n_{k_1}} + X^n_{s,t^n_{k_1}} \otimes X^n_{t^n_{k_1},t}.
  \end{equation*}
  As we have already seen, we have that $\X^n_{s,t^n_{k_0}} = \X^n_{t^n_{k_1},t} = 0$. Recalling the calculation in \eqref{eq:Lambda^n01 = 0}, we note that
  \begin{equation*}
    \X^n_{t^n_{k_0},t^n_{k_1}} = \int_{t^n_{k_0}}^{t^n_{k_1}} X^n_{t^n_{k_0},u} \otimes \d X^n_u = \int_{t^n_{k_0}}^{t^n_{k_1}} X^n_{t^n_{k_0},u} \otimes \d X_u,
  \end{equation*}
  and hence, by the inequality in \eqref{eq:RIE inequality}, that
  \begin{equation*}
    |\X^n_{t^n_{k_0},t^n_{k_1}}|^{\frac{p}{2}} = \bigg|\int_{t^n_{k_0}}^{t^n_{k_1}} X^n_{t^n_{k_0},u} \otimes \d X_u\bigg|^{\frac{p}{2}} \leq w(t^n_{k_0},t^n_{k_1}) \leq w(t^n_{k_0-1},t^n_{k_1+1}).
  \end{equation*}
  We estimate the remaining terms as
  \begin{align*}
    &|X^n_{s,t^n_{k_0}} \otimes X^n_{t^n_{k_0},t^n_{k_1}}|^{\frac{p}{2}} + |X^n_{s,t^n_{k_1}} \otimes X^n_{t^n_{k_1},t}|^{\frac{p}{2}} \lesssim |X^n_{s,t^n_{k_0}}|^p + |X^n_{t^n_{k_0},t^n_{k_1}}|^p + |X^n_{s,t^n_{k_1}}|^p + |X^n_{t^n_{k_1},t}|^p\\
    &\leq |X_{t^n_{k_0-1},t^n_{k_0}}|^p + |X_{t^n_{k_0},t^n_{k_1}}|^p + |X_{t^n_{k_0-1},t^n_{k_1}}|^p + |X_{t^n_{k_1},t^n_{k_1+1}}|^p\\
    &\leq w(t^n_{k_0-1},t^n_{k_0}) + w(t^n_{k_0},t^n_{k_1}) + w(t^n_{k_0-1},t^n_{k_1}) + w(t^n_{k_1},t^n_{k_1+1})\\
    &\leq 2 w(t^n_{k_0-1},t^n_{k_1+1}).
  \end{align*}
  Putting this together, we have that
  \begin{equation*}
    |\X^n_{s,t}|^{\frac{p}{2}} \leq \widetilde{C} w(t^n_{k_0-1},t^n_{k_1+1})
  \end{equation*}
  for some constant $\widetilde{C}$. It follows that, for an arbitrary partition $\cP$ of the interval $[0,T]$, we have the bound
  \begin{equation*}
    \sum_{[s,t] \in \cP} |\X^n_{s,t}|^{\frac{p}{2}} \leq 3 \widetilde{C} w(0,T),
  \end{equation*}
  and hence that \eqref{eq:bbx^n p/2var bound} holds with $C = (3 \widetilde{C})^{\frac{2}{p}}$.
\end{proof}

Using the previous two lemmas, we can now infer the convergence of the rough paths $(\bX^n)_{n \in \N}$ to the rough path $\bX$.

\begin{lemma}\label{lem:rough path approximation}
  Suppose that $X \colon [0,T] \to \R^d$ satisfies Property \textup{(RIE)} relative to some $p \in (2,3)$ and a sequence of partitions $(\cP^n)_{n \in \N}$. Let $\bX = (X,\X)$ and $\bX^n = (X^n,\X^n)$ be the c{\`a}dl{\`a}g rough paths defined via \eqref{eq:RIE rough path} and \eqref{eq:defn bbX^n}, respectively. Then, for any $p' > p$, we have that
  \begin{equation}\label{eq:approx rough paths converge}
    \|\bX^n;\bX\|_{p'} \, \longrightarrow \, 0 \qquad \text{as} \quad n \, \longrightarrow \, \infty,
  \end{equation}
  with a rate of convergence given by
  \begin{equation}\label{eq:rate of convergence for approx rough paths}
    \|\bX^n;\bX\|_{p'} \lesssim \|X^n - X\|_{\infty}^{1 - \frac{p}{p'}} + \sup_{(s,t) \in \Delta_{T}} \bigg|\int_s^t X^n_{s,u} \otimes \d X_u - \X_{s,t}\bigg|^{1 - \frac{p}{p'}},
  \end{equation}
  where the implicit multiplicative constant depends only on $p, p', |X_0|$ and $w(0,T)$, where $w$ is the control function for which \eqref{eq:RIE inequality} holds.
\end{lemma}

\begin{proof}
  By a standard interpolation estimate (e.g., \cite[Proposition~5.5]{Friz2010}), it follows, for any $p' > p$, that
  \begin{equation*}
    \|X^n - X\|_{p'} \leq \|X^n - X\|_p^{\frac{p}{p'}} \|X^n - X\|_{\infty}^{1 - \frac{p}{p'}}.
  \end{equation*}
  We similarly have that
  \begin{equation*}
    \|\X^n - \X\|_{\frac{p'}{2}} \leq \|\X^n - \X\|_{\frac{p}{2}}^{\frac{p}{p'}} \sup_{(s,t) \in \Delta_T} |\X^n_{s,t} - \X_{s,t}|^{1 - \frac{p}{p'}}.
  \end{equation*}
  We recall from Lemma~\ref{lemma: bbX^n to bbX uniformly} that
  \begin{equation*}
    \sup_{(s,t) \in \Delta_T} |\X^n_{s,t} - \X_{s,t}| \leq 2 \|X\|_{\infty} \|X^n - X\|_{\infty} + \sup_{(s,t) \in \Delta_T} \bigg|\int_s^t X^n_{s,u} \otimes \d X_u - \X_{s,t}\bigg|.
  \end{equation*}
We have that $\sup_{n \in \N} \|X^n\|_{p} \leq \|X\|_{p}$ and $\|X\|_\infty \leq |X_0| + \|X\|_p \leq |X_0| + w(0,T)^{\frac{1}{p}}$, and, by the lower semi-continuity of the $\frac{p}{2}$-variation norm and Lemma~\ref{lemma: bbX^n p/2var estimate}, $\|\X\|_{\frac{p}{2}} \leq \liminf_{n \to \infty} \|\X^n\|_{\frac{p}{2}} \leq \sup_{n \in \N} \|\X^n\|_{\frac{p}{2}} \leq C w(0,T)^{\frac{2}{p}}$. Putting this together, we conclude that \eqref{eq:rate of convergence for approx rough paths} holds. By conditions (i) and (ii) in Property \textup{(RIE)}, the convergence in \eqref{eq:approx rough paths converge} then also follows.
\end{proof}

As a next step towards the proof of Theorem~\ref{thm: Euler scheme convergence}, we introduce a discretized version of the RDE~\eqref{eq:RDE Y}. For this purpose, we define a time discretization path along $\cP^n$ by
\begin{equation}\label{eq:time discretization operator}
  \gamma^n_t := T \1_{\{T\}}(t) + \sum_{k=0}^{N_n - 1} t^n_k \1_{[t^n_k,t^n_{k+1})}(t), \qquad t \in [0,T],
\end{equation}
and consider the RDE
\begin{equation}\label{eq:RDE tilde Y^n}
  \tY^n_t = y_0 + \int_0^t b(\gamma^n_s,\tY^n_s) \dd \gamma^n_s + \int_0^t \sigma(\gamma^n_s,\tY^n_s) \dd \bX^n_s, \qquad t \in [0,T].
\end{equation}

Thanks to Lemma~\ref{lem:rough path approximation} and the local Lipschitz continuity of the It{\^o}--Lyons map, we obtain the following proposition.

\begin{proposition}\label{prop: discretized RDE}
  Suppose that $X \colon [0,T] \to \R^d$ satisfies Property \textup{(RIE)} relative to some $p \in (2,3)$ and a sequence of partitions $(\cP^n)_{n \in \N}$. Let $Y$ be the solution of the RDE~\eqref{eq:RDE Y}, and let $\tY^n$ be the solution of the RDE~\eqref{eq:RDE tilde Y^n}. Then,
  \begin{equation}\label{eq:tY^n converge to Y}
    \|\tY^n - Y\|_{p'} \, \longrightarrow \, 0 \qquad \text{as} \quad n \, \longrightarrow \, \infty, 
  \end{equation}
  for any $p' \in (p,3)$, with a rate of convergence given by
  \begin{equation*}
    \|\tY^n - Y\|_{p'} \lesssim |\cP^n|^{1 - \frac{1}{q}} + \|X^n - X\|_{\infty}^{1 - \frac{p}{p'}} + \bigg\|\int_0^{\cdot} X^n_{u} \otimes \d X_u - \int_0^{\cdot} X_{u} \otimes \d X_u\bigg\|_{\infty}^{1 - \frac{p}{p'}},
  \end{equation*}
for any $q \in (1,2)$ such that $\frac{1}{p'} + \frac{1}{q}>1$, where the implicit multiplicative constant depends only on $p, p', q, \|b\|_{C_b^2}, \|\sigma\|_{C_b^3}, T, |X_0|$ and $w(0,T)$, where $w$ is the control function for which \eqref{eq:RIE inequality} holds.
\end{proposition}

\begin{proof}
  Setting $\gamma_t := t$ for $t \in [0,T]$, the RDE~\eqref{eq:RDE Y} may be rewritten as
  \begin{equation*}
    Y_t = y_0 + \int_0^t b(\gamma_s,Y_s) \dd \gamma_s + \int_0^t \sigma(\gamma_s,Y_s) \dd \bX_s, \qquad t \in [0,T].
  \end{equation*}
  Hence, by Theorem~\ref{thm: RDE}, we know that
  \begin{equation}\label{eq:proof of Milstein}
    \|\tY^n - Y\|_{p'} \lesssim \|\gamma^n - \gamma\|_{q} + \|\bX^n;\bX\|_{p'}
  \end{equation}
  for any $p' \in (p,3)$ and any $q \in [1,2)$ such that $\frac{1}{p'} + \frac{1}{q} > 1$.

  Note that $\gamma^n$ and $\gamma$ have finite $1$-variation, with $\|\gamma^n\|_1 = \|\gamma\|_1 = T$, and $\|\gamma^n - \gamma\|_{1} = 2T$. Although $\gamma^n$ does not converge to $\gamma$ in $1$-variation, it is straightforward to see by interpolation that
  \begin{equation*}
    \|\gamma^n - \gamma\|_q \leq \|\gamma^n - \gamma\|_1^{\frac{1}{q}} \|\gamma^n - \gamma\|_\infty^{1 - \frac{1}{q}} = (2T)^{\frac{1}{q}} |\cP^n|^{1 - \frac{1}{q}}
  \end{equation*}
  for any $q > 1$. Combining this with the estimate in \eqref{eq:proof of Milstein} and the result of Lemma~\ref{lem:rough path approximation}, we infer the convergence in \eqref{eq:tY^n converge to Y}, and the estimate
  \begin{align*}
    \|\tY^n - Y\|_{p'} &\lesssim \|\gamma^n - \gamma\|_q + \|X^n - X\|_{\infty}^{1 - \frac{p}{p'}} + \sup_{(s,t) \in \Delta_T} \bigg|\int_s^t X^n_{s,u} \otimes \d X_u - \X_{s,t}\bigg|^{1 - \frac{p}{p'}}\\
    &\lesssim |\cP^n|^{1 - \frac{1}{q}} + \|X^n - X\|_{\infty}^{1 - \frac{p}{p'}} + \bigg\|\int_0^{\cdot} X^n_{u} \otimes \d X_u - \int_0^{\cdot} X_{u} \otimes \d X_u\bigg\|_{\infty}^{1 - \frac{p}{p'}}.
  \end{align*}
\end{proof}

\begin{remark}
  For a path $A \in D^1([0,T];\R^d)$ of finite $1$-variation, let us consider the controlled ordinary differential equation (ODE)
  \begin{equation}\label{eq: RDE example}
    Z_t = z_0 + \int_0^t \sigma(Z_s) \dd A_s, \qquad t \in [0,T],
  \end{equation}
  where the integral is interpreted in the Riemann--Stieltjes sense. It is a classical result that, provided $\sigma$ is sufficiently regular, the ODE in \eqref{eq: RDE example} is well-posed, and that the solution map $\Phi \colon A \mapsto Z$ is continuous with respect to the $1$-variation norm $\|\cdot\|_1$. A major insight of the theory of rough paths is that the solution map $\Phi$ can be extended from the space of smooth paths to the space $\mathscr{C}^{0,p\textup{-var}}([0,T];\R^d)$ of continuous geometric rough paths for $p \in (2,3)$; see, e.g., \cite{Friz2010}. Of course, the closure of a set containing only continuous paths with respect to $p$-variation norms will again only contain continuous paths.

  In the current framework of c{\`a}dl{\`a}g rough paths, Lemma~\ref{lem:rough path approximation} and Proposition~\ref{prop: discretized RDE} motivate us to consider instead the closure of c{\`a}dl{\`a}g paths of finite $1$-variation. For $p \in (2,3)$, let $\cD^{0,p}([0,T];\R^d)$ denote the closure of the set
  \begin{equation*}
    \bigg\{\mathbf{A} = (A,\mathbb{A}) \, : \, A \in D^1([0,T];\R^d) \, \text{ and } \, \mathbb{A}_{s,t} := \int_s^t A_{s,u} \otimes \d A_u \, \text{ for all } \, (s,t) \in \Delta_T\bigg\}
  \end{equation*}
  with respect to the rough path distance $\|\,\cdot\,;\,\cdot\,\|_{p}$ (as defined in \eqref{eq:rough path distance}), where $\int_s^t A_{s,u} \otimes \d A_u$ is defined as a left-point Riemann--Stieltjes integral. Then, the solution map $\Phi \colon A \mapsto Z$ extends continuously to the space $\cD^{0,p}([0,T];\R^d)$ by Theorem~\ref{thm: RDE}, and every path satisfying Property \textup{(RIE)} is in $\cD^{0,p'}([0,T];\R^d)$ for $p'\in (p,3)$ by Lemma~\ref{lem:rough path approximation}.
\end{remark}

Next, we shall verify that the piecewise constant approximation $X^n$ of $X$, as defined in Property \textup{(RIE)}, itself satisfies Property \textup{(RIE)} relative to any sequence of partitions $(\tcP^m)_{m \in \N}$ which are coarser than $\cP^n$ and have vanishing mesh size.

\begin{lemma}\label{lemma: X^n satisfies RIE}
  Suppose that a path $X$ satisfies Property \textup{(RIE)} relative to $p \in (2,3)$ and a sequence of partitions $(\cP^n)_{n \in \N}$, and let $X^n$ be the usual piecewise constant approximation of $X$ along $\cP^n$. Then the path $X^n$ satisfies Property \textup{(RIE)} relative to $p$ and any sequence of partitions $(\tcP^m)_{m \in \N}$ such that $\cP^n \subseteq \tcP^m$ for every $m \in \N$, and $|\tcP^m| \to 0$ as $m \to \infty$.
\end{lemma}

\begin{proof}
  We need to verify each of the conditions (i)--(iii) of Property \textup{(RIE)} along the sequence of partitions $(\tcP^m)_{m \in \N}$. Since $\cP^n \subseteq \tcP^m$ for every $m \in \N$, the piecewise constant approximation of $X^n$ along the partition $\tcP^m$ is simply the path $X^n$ itself. Conditions (i) and (ii) thus hold trivially.

  Let $w_{1,n}$ be the control function given by $w_{1,n}(s,t) := \|X^n\|_{p,[s,t]}^p$, so that $|X^n_{s,t}|^p \leq w_{1,n}(s,t)$ for all $(s,t) \in \Delta_{T}$, and similarly let $w_{2,n}$ be the control function given by $w_{2,n}(s,t) := \|X^n\|_{\frac{p}{2},[s,t]}^{\frac{p}{2}}$. Let us also write $\tcP^m = \{0 = r^m_0 < r^m_1 < \cdots < r^m_{\widetilde{N}_m} = T\}$ for each $m \in \N$. Then, for any $m \in \N$ and any $0 \leq k < \ell \leq \widetilde{N}_m$, using the standard estimate for Young integration (see, e.g., \cite[Proposition~2.4]{Friz2018}) we have that
  \begin{align*}
    &\bigg|\int_{r^m_k}^{r^m_\ell} X^n_u \otimes \d X^n_u - X^n_{r^m_k} \otimes X^n_{r^m_k,r^m_\ell}\bigg|^{\frac{p}{2}} \lesssim \|X^n\|_{p,[r^m_k,r^m_\ell]}^{\frac{p}{2}} \|X^n\|_{\frac{p}{2},[r^m_k,r^m_\ell]}^{\frac{p}{2}}\\
    &\quad \leq \|X^n\|_p^{\frac{p}{2}} \|X^n\|_{\frac{p}{2},[r^m_k,r^m_\ell]}^{\frac{p}{2}}
    \leq \|X\|_p^{\frac{p}{2}} w_{2,n}(r^m_k,r^m_\ell).
  \end{align*}
  Thus, condition (iii) holds for $X^n$ with the control function $w_{3,n}$, given by
  \begin{equation*}   
    w_{3,n}(s,t) := w_{1,n}(s,t) + \|X\|_p^{\frac{p}{2}} w_{2,n}(s,t), \qquad (s,t) \in \Delta_{T}.
  \end{equation*}
\end{proof}

We are now in a position to complete the proof of Theorem~\ref{thm: Euler scheme convergence}. For this, we will apply in particular the result of Theorem~\ref{thm:rough int under RIE}, which states that, under Property \textup{(RIE)}, the rough integral can be obtained as a limit of classical left-point Riemann sums.

\begin{proof}[Proof of Theorem~\ref{thm: Euler scheme convergence}]
  Note that the Euler scheme in \eqref{eq:RDE Euler scheme} may be expressed as the solution of the controlled ODE
  \begin{equation}\label{eq:ODE Y^n}
    Y^n_t = y_0 + \int_0^t b(\gamma^n_s,Y^n_s) \dd \gamma^n_s + \int_0^t \sigma(\gamma^n_s,Y^n_s) \dd X^n_s, \qquad t \in [0,T],
  \end{equation}
  where $\gamma^n$ denotes the time discretization path along $\cP^n$ defined in \eqref{eq:time discretization operator}, and the integrals are defined as limits of left-point Riemann sums. Recall that $\tY^n$ denotes the solution of the RDE in \eqref{eq:RDE tilde Y^n}, that is
  \begin{equation}\label{eq:tY^n RDE in proof}
    \tY^n_t = y_0 + \int_0^t b(\gamma^n_s,\tY^n_s) \dd \gamma^n_s + \int_0^t \sigma(\gamma^n_s,\tY^n_s) \dd \bX^n_s, \qquad t \in [0,T],
  \end{equation}
  where $\bX^n$ is the canonical rough path lift of $X^n$, as constructed in \eqref{eq:defn bbX^n}.

  Since $X^n$ is piecewise constant, it is clear from the definition of $\X^n$ that $\X^n_{s,t} = 0$ for any times $s \leq t$ which lie in the same subinterval $[t^n_k,t^n_{k+1})$ of the partition $\cP^n$. Since $\gamma^n$ is also constant on each such subinterval, it follows from the definitions of Young and rough integrals that the solution $\tY^n$ of \eqref{eq:tY^n RDE in proof} is itself also piecewise constant along the partition $\cP^n$.

  Let $\tcP^m = \{0 = r^m_0 < r^m_1 < \cdots < r^m_{\widetilde{N}_m} = T\}$, $m \in \N$, be any sequence of partitions with mesh size converging to $0$, such that $\cP^n \subseteq \tcP^m$ for every $m \in \N$. By Lemma~\ref{lemma: X^n satisfies RIE}, we have that the path $X^n$ satisfies Property \textup{(RIE)} relative to $p$ and the sequence $(\tcP^m)_{m \in \N}$. Since $\gamma^n$ and $\tY^n$ are piecewise constant along the partition $\cP^n$, it is clear that the jump times of the integrand $s \mapsto \sigma(\gamma^n_s,\tY^n_s)$ all belong to $\cP^n$, and thus also belong to the set $\liminf_{m \to \infty} \tcP^m$. It thus follows from Theorem~\ref{thm:rough int under RIE} that the rough integral $\int_0^t \sigma(\gamma^n_s,\tY^n_s) \dd \bX^n_s$ is equal to a limit of left-point Riemann sums along the sequence $(\tcP^m)_{m \in \N}$. That is, for any $t \in [0,T]$, we have that
  \begin{align*}
    \int_0^t \sigma(\gamma^n_s,\tY^n_s) \dd \bX^n_s &= \lim_{m \to \infty} \sum_{k=0}^{\widetilde{N}_m-1} \sigma(\gamma^n_{r^m_k},\tY^n_{r^m_k}) X^n_{r^m_k \wedge t,r^m_{k+1} \wedge t}\\  
    &= \sum_{k=0}^{N_n-1} \sigma(\gamma^n_{t^n_k},\tY^n_{t^n_k}) X^n_{t^n_k \wedge t,t^n_{k+1} \wedge t} = \int_0^t \sigma(\gamma^n_s,\tY^n_s) \dd X^n_s.
  \end{align*}
  Since these integrals are equal, it follows that the ODE in \eqref{eq:ODE Y^n} and the RDE in \eqref{eq:tY^n RDE in proof} are consistent, so that $Y^n = \tY^n$. The result then follows from Proposition~\ref{prop: discretized RDE}.
\end{proof}

\subsection{Error bound for an approximate Euler scheme}

In general, the Euler scheme~\eqref{eq:RDE Euler scheme} is not applicable to numerically approximate the solution of an SDE driven by a general L{\'e}vy process---as we will consider in Section~\ref{section: Levy processes} below---since the increments of L{\'e}vy processes cannot always be simulated. Therefore, to obtain a numerical approximation of the solution of such a L{\'e}vy-driven SDE, one needs to consider approximate Euler schemes---see, e.g., \cite{Jacod2005,Rubenthaler2003,Dereich2011}---where the increments of the driving L{\'e}vy process are approximated by random variables with known distributions.

\medskip

As a pathwise counterpart, we introduce the approximate Euler scheme $\widehat{Y}^n$ of the RDE \eqref{eq:RDE Y} along the partition $\cP^n$, given by
\begin{equation}\label{eq:approximate Euler scheme for RDE}
  \widehat{Y}^n_t = y_0 + \sum_{i \hspace{1pt} : \hspace{1pt} t^n_{i+1} \leq t} b(t^n_i,\widehat{Y}^n_{t^n_i}) (t^n_{i+1} - t^n_i) +  \sum_{i \hspace{1pt} : \hspace{1pt} t^n_{i+1} \leq t} \sigma(t^n_i,\widehat{Y}^n_{t^n_i}) (\widehat{X}_{t^n_{i+1}} - \widehat{X}_{t^n_i}),
\end{equation}
for $t \in [0,T]$, with the modified driving signal
\begin{equation*}
  \widehat{X} := X + \phi,
\end{equation*}
where $\phi \in D^q([0,T];\R^d)$, for some $q \in [1,2)$ such that $\frac{1}{p} + \frac{1}{q} > 1$, and, as usual, we write $\mathcal{P}^n = \{0 = t^n_0 < t^n_1 < \cdots < t_{N_n}^n = T\}$.

\medskip

While the approximation error of the Euler scheme~\eqref{eq:RDE Euler scheme} was only caused by discretizing the time interval $[0,T]$, the approximate Euler scheme~\eqref{eq:approximate Euler scheme for RDE} produces an additional approximation error due to taking the modified driving signal $\widehat{X}$ as an input, instead of the actual driving signal $X$ of the RDE~\eqref{eq:RDE Y}.

To ensure the convergence of the approximate Euler scheme, we first need to verify that, if the actual driving signal satisfies Property \textup{(RIE)}, then the same is true for the modified driving signal.

\begin{proposition}\label{prop: stability of RIE}
  Suppose that $X \in D([0,T];\mathbb{R}^d)$ satisfies Property \textup{(RIE)} relative to some $p \in (2,3)$ and a sequence of partitions $\cP^n = \{0 = t^n_0 < t^n_1 < \cdots < t^n_{N_n} = T\}$, $n \in \N$. Let $\phi \in D^q([0,T];\R^d)$ for some $q \in [1,2)$ such that $\frac{1}{p} + \frac{1}{q} > 1$. For each $n \in \N$, we define $\phi^n \colon [0,T] \to \R^d$ by
  \begin{equation}\label{eq: defn phi^n approx}
    \phi_t^n = \phi_T \1_{\{T\}}(t) + \sum_{k=0}^{N_n-1} \phi_{t_k^n} \1_{[t_k^n,t_{k+1}^n)}(t), \qquad t \in [0,T],
  \end{equation}
  as the discretization of $\phi$ along $\cP^n$. Suppose that $\|\phi^n - \phi\|_q \to 0$ as $n \to \infty$. Then the path $\widehat{X} = X + \phi$ satisfies Property \textup{(RIE)} relative to $p$ and $(\cP^n)_{n \in \N}$.
\end{proposition}

\begin{proof}
  We need to verify the conditions (i)--(iii) of Property \textup{(RIE)}.

  \emph{(i):} Letting $\widehat{X}^n$ denote the piecewise constant approximation of $\widehat{X}$ along the partition $\cP^n$, it is clear that $\widehat{X}^n = X^n + \phi^n$ for each $n \in \N$. Since $X^n$ converges uniformly to $X$ by Property \textup{(RIE)}, and $\|\phi^n - \phi\|_q \to 0$ by assumption, it is clear that $\widehat{X}^n$ converges uniformly to $\widehat{X}$ as $n \to \infty$.

  \emph{(ii):} We need to verify that the integral
  \begin{equation*}
    \int_0^t \widehat{X}_u^n \otimes \d \widehat{X}_u = \int_0^t X_u^n \otimes \d X_u + \int_0^t X_u^n \otimes \d \phi_u + \int_0^t \phi_u^n \otimes \d X_u + \int_0^t \phi_u^n \otimes \d \phi_u,
  \end{equation*} 
  converges as $n \to \infty$ to the limit
  \begin{equation*}
    \int_0^t \widehat{X}_u \otimes \d \widehat{X}_u := \int_0^t X_u \otimes \d X_u + \int_0^t X_u \otimes \d \phi_u + \int_0^t \phi_u \otimes \d X_u + \int_0^t \phi_u \otimes \d \phi_u,
  \end{equation*}
  uniformly in $t \in [0,T]$, where the latter three integrals are defined as Young integrals.

  Since $X$ satisfies Property \textup{(RIE)}, we have that
  \begin{equation*}
    \bigg\|\int_0^\cdot X_u^n \otimes \d X_u - \int_0^\cdot X_u \otimes \d X_u \bigg \|_\infty \, \longrightarrow \, 0 \qquad \text{as} \quad n \, \longrightarrow \, \infty.
  \end{equation*}
  Let $p' > p$ such that $\frac{1}{p'} + \frac{1}{q} > 1$. By the standard estimate for Young integrals---see, e.g., \cite[Proposition~2.4]{Friz2018}---we have, for all $t \in [0,T]$, that
  \begin{equation*}
    \bigg|\int_0^t X_u^n \otimes \d \phi_u - \int_0^t X_u \otimes \d \phi_u\bigg| \lesssim \|X^n - X\|_{p'} \|\phi\|_q.
  \end{equation*}
  It follows by interpolation (see, e.g., \cite[Proposition~5.5]{Friz2010}) that
  \begin{equation*}
    \|X^n - X\|_{p'} \leq \|X^n - X\|_\infty^{1 - \frac{p}{p'}} \|X^n - X\|_p^{\frac{p}{p'}}.
  \end{equation*}
  Since $X^n$ converges uniformly to $X$ as $n \to \infty$, and $\sup_{n \in \N} \|X^n\|_p \leq \|X\|_p < \infty$, we deduce that
  \begin{equation*}
     \bigg\|\int_0^\cdot X_u^n \otimes \d \phi_u - \int_0^\cdot X_u \otimes \d \phi_u \bigg \|_\infty \, \longrightarrow \, 0 \qquad \text{as} \quad n \, \longrightarrow \, \infty.
  \end{equation*}
  Similarly, for each $t \in [0,T]$, it holds that
  \begin{equation*}
    \bigg|\int_0^t \phi_u^n \otimes \d X_u - \int_0^t \phi_u \otimes \d X_u \bigg| \lesssim \|\phi^n - \phi\|_q \|X\|_p,
  \end{equation*}
  and
  \begin{equation*}
    \bigg|\int_0^t \phi_u^n \otimes \d \phi_u - \int_0^t \phi_u \otimes \d \phi_u \bigg| \lesssim \|\phi^n - \phi\|_q \|\phi\|_q,
  \end{equation*}
  and, since $\|\phi^n - \phi\|_q \to 0$ as $n \to \infty$, we infer the required convergence.

  \emph{(iii):} We aim to find a control function $w$ such that
  \begin{equation}\label{eq: condition (iii)}
    \sup_{(s,t) \in \Delta_T} \frac{|\widehat{X}_{s,t}|^p}{w(s,t)} + \sup_{n \in \N} \, \sup_{0 \leq k < \ell \leq N_n} \frac{|\int_{t^n_k}^{t^n_\ell} \widehat{X}^n_{t^n_k,u} \otimes \d \widehat{X}_u|^{\frac{p}{2}}}{w(t_k^n,t_\ell^n)} \leq 1,
  \end{equation}
  where
  \begin{equation*}
    \int_{t^n_k}^{t^n_\ell} \widehat{X}_{t^n_k,u}^n \otimes \d \widehat{X}_u = \int_{t^n_k}^{t^n_\ell} X_{t^n_k,u}^n \otimes \d X_u + \int_{t^n_k}^{t^n_\ell} X_{t^n_k,u}^n \otimes \d \phi_u + \int_{t^n_k}^{t^n_\ell} \phi_{t^n_k,u}^n \otimes \d X_u + \int_{t^n_k}^{t^n_\ell} \phi_{t^n_k,u}^n \otimes \d \phi_u.
  \end{equation*}

  Let $w_X$ be the control function with respect to which $X$ satisfies Property \textup{(RIE)}, and define moreover the control function $w_\phi$, given by $w_{\phi}(s,t) = \|\phi\|_{q,[s,t]}^q$ for $(s,t) \in \Delta_T$.

  We have from Property \textup{(RIE)} that
  \begin{equation*}
    \sup_{(s,t) \in \Delta_T} \frac{|\widehat{X}_{s,t}|^p}{w_X(s,t) + w_\phi(s,t)} \lesssim \sup_{(s,t) \in \Delta_T} \frac{|X_{s,t}|^p}{w_X(s,t)} + \sup_{(s,t) \in \Delta_T} \frac{|\phi_{s,t}|^p}{w_\phi(s,t)} \leq 2,
  \end{equation*}
  and that
  \begin{equation*}  
    \sup_{n \in \N} \, \sup_{0 \leq k < \ell \leq N_n} \frac{|\int_{t_k^n}^{t_\ell^n} X_{t_k^n,u}^n \otimes \d X_u|^{\frac{p}{2}}}{w_X(t_k^n,t_\ell^n)} \leq 1.
  \end{equation*}
  By the standard estimate for Young integrals (see, e.g., \cite[Proposition~2.4]{Friz2018}), for every $n \in \N$ and $0 \leq k < \ell \leq N_n$, we have
  \begin{align*}
    \bigg|\int_{t_k^n}^{t_\ell^n} X_{t_k^n,u}^n \otimes \d \phi_u \bigg|^{\frac{p}{2}} &\lesssim \|X^n\|_{p,[t_k^n,t_\ell^n]}^{\frac{p}{2}} \|\phi\|_{q,[t_k^n,t_\ell^n]}^{\frac{p}{2}}\\
    &\leq \|X\|_{p,[t_k^n,t_\ell^n]}^{\frac{p}{2}} \|\phi\|_{q,[t_k^n,t_\ell^n]}^{\frac{p}{2}} \leq w_X(t_k^n,t_\ell^n)^{\frac{1}{2}} w_\phi(t_k^n,t_\ell^n)^{\frac{p}{2q}},
  \end{align*}
  and we can similarly obtain
  \begin{equation*}
    \bigg|\int_{t_k^n}^{t_\ell^n} \phi^n_{t_k^n,u} \otimes \d X_u \bigg|^{\frac{p}{2}} \lesssim w_X(t_k^n,t_\ell^n)^{\frac{1}{2}} w_\phi(t_k^n,t_\ell^n)^{\frac{p}{2q}}
  \end{equation*}
  and
  \begin{equation*}
    \bigg|\int_{t_k^n}^{t_\ell^n} \phi_{t_k^n,u}^n \otimes \d \phi_u\bigg|^{\frac{p}{2}} \lesssim w_\phi(t_k^n,t_\ell^n)^{\frac{p}{q}}.
  \end{equation*}
  Since $p \in (2,3)$ and $q \in [1,2)$, we have that $\frac{1}{2} + \frac{p}{2q} > 1$ and $\frac{p}{q} > 1$, and it follows that the maps $(s,t) \mapsto w_X(s,t)^{\frac{1}{2}} w_\phi(s,t)^{\frac{p}{2q}}$ and $(s,t) \mapsto w_\phi(s,t)^{\frac{p}{q}}$ are superadditive and thus control functions. We deduce that \eqref{eq: condition (iii)} holds with a control function $w$ of the form
  \begin{equation*}  
    w(s,t) = C\Big(w_X(s,t) + w_\phi(s,t) + w_X(s,t)^{\frac{1}{2}} w_\phi(s,t)^{\frac{p}{2q}} + w_\phi(s,t)^{\frac{p}{q}}\Big), \qquad (s,t) \in \Delta_T,
  \end{equation*}
  where $C > 0$ is a suitable constant which depends only on $p$ and $q$.
\end{proof}

By Proposition~\ref{prop: stability of RIE}, the modified driving signal $\widehat{X}$ satisfies Property \textup{(RIE)}, and can thus be canonically lifted to a rough path $\widehat{\mathbf{X}} = (\widehat{X},\widehat{\mathbb{X}}) \in \cD^p([0,T];\R^d)$ via \eqref{eq:RIE rough path}. By Theorem~\ref{thm: RDE}, the RDE \eqref{eq:RDE Y} driven by $\widehat{\mathbf{X}}$ has a unique solution $\widehat{Y}$, and the approximate Euler scheme $\widehat{Y}^n$ in \eqref{eq:approximate Euler scheme for RDE} converges to $\widehat{Y}$ by Theorem~\ref{thm: Euler scheme convergence}. We will see an application of this to SDEs driven by L\'evy processes in Section~\ref{section: Levy processes}.

The next proposition provides an error and convergence analysis for the approximate Euler scheme \eqref{eq:approximate Euler scheme for RDE} with respect to the solution $Y$ of the RDE~\eqref{eq:RDE Y} driven by the rough path $\bX = (X,\mathbb{X})$ under Property \textup{(RIE)}.

\begin{proposition}\label{prop:approximate Euler scheme}
  Suppose that $X \in D([0,T];\R^d)$ satisfies Property \textup{(RIE)} relative to $p \in (2,3)$ and a sequence of partitions $(\cP^n)_{n \in \N}$, and let $\bX$ be its canonical rough path lift. Let $\phi \in D^q([0,T];\R^d)$ for some $q \in (1,2)$ such that $\frac{1}{p} + \frac{1}{q} > 1$, let $\phi^n$ be the piecewise constant approximation of $\phi$, as defined in \eqref{eq: defn phi^n approx}, and assume that $\|\phi^n - \phi\|_q \to 0$ as $n \to \infty$. Let $Y$ be the solution of the RDE \eqref{eq:RDE Y} driven by $\bX$, and let $\widehat{Y}^n$ be the approximate Euler scheme defined in \eqref{eq:approximate Euler scheme for RDE}. We have the error estimate
  \begin{align*}
    \|\widehat{Y}^n - Y\|_{p'} &\lesssim (1 + \|X\|_p + \|\phi\|_q) \|\phi\|_q + |\cP^n|^{1 - \frac{1}{q}} + (\|X^n - X\|_\infty + \|\phi^n - \phi\|_\infty)^{1 - \frac{p}{p'}}\\
    &\quad + \bigg(\bigg\|\int_0^\cdot X_u^n \otimes \d X_u - \int_0^\cdot X_u \otimes \d X_u\bigg\|_\infty + \|X^n - X\|_{p'} + \|\phi^n - \phi\|_q\bigg)^{\hspace{-2pt}1 - \frac{p}{p'}}
  \end{align*}
  for any $p' \in (p,3)$ such that $\frac{1}{p'} + \frac{1}{q} > 1$, where the implicit multiplicative constant depends on $p, p', q, \|b\|_{C_b^2}, \|\sigma\|_{C_b^3}, T, \|X\|_\infty, \|\bX\|_p, \|\phi\|_\infty, \|\phi\|_q$ and $w(0,T)$, where $w$ is the control function for which \eqref{eq:RIE inequality} holds. In particular, we have that
  \begin{equation}\label{eq:bound on limsup approx Euler}
    \limsup_{n \to \infty} \, \|\widehat{Y}^n - Y\|_{p'} \lesssim (1 + \|X\|_p + \|\phi\|_q) \|\phi\|_q.
  \end{equation}
\end{proposition}

\begin{proof}
  By Proposition~\ref{prop: stability of RIE}, we know that the path $\widehat{X} = X + \phi$ satisfies Property \textup{(RIE)} relative to $p$ and $(\cP^n)_{n \in \N}$. Let $\widehat{\bX}$ be the canonical rough path lift of $\widehat{X}$, and let $Y$ and $\widehat{Y}$ be the solutions of the RDE~\eqref{eq:RDE Y} driven by $\bX$ and $\widehat{\mathbf{X}}$ respectively. It is clear that
  \begin{equation*}
    \|\widehat{Y}^n - Y\|_{p'} \leq \|\widehat{Y}^n - \widehat{Y}\|_{p'} + \|\widehat{Y} - Y\|_{p'}.
  \end{equation*} 
  By Theorem~\ref{thm: RDE}, we have the estimate
  \begin{equation*}
    \|\widehat{Y} - Y\|_{p'} \lesssim \|\widehat{\bX};\bX\|_{p'},
  \end{equation*} 
  and, by Theorem~\ref{thm: Euler scheme convergence}, we have that
  \begin{equation*}
    \|\widehat{Y}^n - \widehat{Y}\|_{p'} \lesssim |\cP^n|^{1 - \frac{1}{q}} + \|\widehat{X}^n - \widehat{X}\|_{\infty}^{1 - \frac{p}{p'}} + \bigg\|\int_0^\cdot \widehat{X}_u^n \otimes \d \widehat{X}_u - \int_0^\cdot \widehat{X}_u \otimes \d \widehat{X}_u \bigg\|_{\infty}^{1 - \frac{p}{p'}},
  \end{equation*}
  where $\widehat{X}^n$ is the piecewise constant approximation of $\widehat{X}$ along $\cP^n$. Since $\widehat{X}^n = X^n + \phi^n$, we can bound
  \begin{equation*}
    \|\widehat{X}^n - \widehat{X}\|_\infty \leq \|X^n - X\|_\infty + \|\phi^n - \phi\|_\infty.
  \end{equation*}
  As shown in the proof of Proposition~\ref{prop: stability of RIE},
  \begin{align*}
    &\bigg\|\int_0^\cdot \widehat{X}_u^n \otimes \d \widehat{X}_u - \int_0^\cdot \widehat{X}_u \otimes \d \widehat{X}_u\bigg\|_\infty\\
    &\quad\lesssim \bigg\|\int_0^\cdot X_u^n \otimes \d X_u - \int_0^\cdot X_u \otimes \d X_u \bigg\|_\infty + \|X^n - X\|_{p'} \|\phi\|_q + \|\phi^n - \phi\|_q (\|X\|_p + \|\phi\|_q).
  \end{align*}

  We also note that
  \begin{equation*}
    \widehat{\X}_{s,t} - \X_{s,t} = \int_s^t X_{s,u} \otimes \d \phi_u + \int_s^t \phi_{s,u} \otimes \d X_u + \int_s^t \phi_{s,u} \otimes \d \phi_u
  \end{equation*}
  for $(s,t) \in \Delta_T$, so that, by the standard estimate for Young integrals (see, e.g., \cite[Proposition~2.4]{Friz2018}), we obtain
  \begin{equation*}
    |\widehat{\X}_{s,t} - \X_{s,t}| \lesssim \|X\|_{p,[s,t]} \|\phi\|_{q,[s,t]} + \|\phi\|_{q,[s,t]}^2.
  \end{equation*}
  This implies that, for any partition $\cP$ of the interval $[0,T]$,
  \begin{align*} 
    &\sum_{[s,t] \in \cP} |\widehat{\X}_{s,t} - \X_{s,t}|^{\frac{p}{2}} \lesssim \sum_{[s,t] \in \cP} (\|X\|_{p,[s,t]}^{\frac{p}{2}} \|\phi\|_{q,[s,t]}^{\frac{p}{2}} + \|\phi\|_{q,[s,t]}^p)\\
    &\quad \leq \bigg(\sum_{[s,t] \in \cP} \|X\|_{p,[s,t]}^p\bigg)^{\hspace{-2pt}\frac{1}{2}} \bigg(\sum_{[s,t] \in \cP} \|\phi\|_{q,[s,t]}^p\bigg)^{\hspace{-2pt}\frac{1}{2}} + \sum_{[s,t] \in \cP} \|\phi\|_{q,[s,t]}^p\\
    &\quad \leq \bigg(\sum_{[s,t] \in \cP} \|X\|_{p,[s,t]}^p\bigg)^{\hspace{-2pt}\frac{1}{2}} \bigg(\sum_{[s,t] \in \cP} \|\phi\|_{q,[s,t]}^q\bigg)^{\hspace{-2pt}\frac{p}{2q}} + \bigg(\sum_{[s,t] \in \cP} \|\phi\|_{q,[s,t]}^q\bigg)^{\hspace{-2pt}\frac{p}{q}} \leq \|X\|_p^{\frac{p}{2}} \|\phi\|_q^{\frac{p}{2}} + \|\phi\|_q^p,
  \end{align*}
  so that $\|\widehat{\X} - \X\|_{\frac{p}{2}} \lesssim \|X\|_p \|\phi\|_q + \|\phi\|_q^2$. We thus deduce that
  \begin{equation*}  
    \|\widehat{\bX};\bX\|_{p'} \leq \|\widehat{X} - X\|_p + \|\widehat{\X} - \X\|_{\frac{p}{2}} \lesssim (1 + \|X\|_p + \|\phi\|_q) \|\phi\|_q,
  \end{equation*}
  and combining the estimates above, we obtain the desired error estimate.
\end{proof}

As an immediate consequence of Proposition~\ref{prop:approximate Euler scheme}, if the modified driving signal $\widehat{X}$ converges to the driving signal $X$, then the approximate Euler scheme converges to the solution $Y$ of the RDE~\eqref{eq:RDE Y}. This is made precise in the following corollary, which follows from \eqref{eq:bound on limsup approx Euler}.

\begin{corollary}
Recall the setting of Proposition~\ref{prop:approximate Euler scheme}, and now let $\check{Y}^n$ be the approximate Euler scheme of the RDE \eqref{eq:RDE Y} along the partition $\cP^n$, given by
  \begin{equation*}
    \check{Y}^n_t = y_0 + \sum_{i \hspace{1pt} : \hspace{1pt} t^n_{i+1} \leq t} b(t^n_i,\check{Y}^n_{t^n_i}) (t^n_{i+1} - t^n_i) + \sum_{i \hspace{1pt} : \hspace{1pt} t^n_{i+1} \leq t} \sigma(t^n_i,\check{Y}^n_{t^n_i}) (\check{X}^n_{t^n_{i+1}} - \check{X}^n_{t^n_i})
  \end{equation*}
  for $t \in [0,T]$, with the modified driving signal
  \begin{equation*}
    \check{X}^n := X + \psi^n,
  \end{equation*}
  where $\psi^n \in D^q([0,T];\R^d)$ for some $q \in (1,2)$ such that $\frac{1}{p} + \frac{1}{q} > 1$. If $\|\psi^n\|_q \to 0$ as $n \to \infty$, then
  \begin{equation*}
    \|\check{Y}^n - Y\|_{p'} \, \longrightarrow \, 0 \qquad \text{as} \quad n \, \longrightarrow \, \infty
  \end{equation*}
  for any $p' \in (p,3)$ such that $\frac{1}{p'} + \frac{1}{q} > 1$.
\end{corollary}

\begin{remark}\label{remark: absorb young part into drift instead}
  In this section we handled the modified driving signal $X + \phi$ by considering the rough path lift $\widehat{\bX}$ of $\widehat{X} = X + \phi$, and considering the solution $\widehat{Y}$ of the RDE \eqref{eq:RDE Y} driven by $\widehat{\bX}$. An alternative, equally valid approach would be to instead absorb $\phi$ into the drift of the RDE. The resulting equation would not strictly speaking be of the form in \eqref{eq:RDE Y}, but it would still fall into the regime of the more general RDE in \eqref{eq:general RDE}, and an error estimate could be obtained using the stability estimate in Theorem~\ref{thm: RDE}.
\end{remark}

\section{Applications to stochastic differential equations}\label{sec:SDE}

In this section we apply the deterministic theory developed in Section~\ref{sec:RDE}, regarding the Euler scheme for RDEs, to stochastic differential equations (SDEs). For this purpose, we now let $X$ be a $d$-dimensional c{\`a}dl{\`a}g semimartingale, defined on a probability space $(\Omega,\mathcal{F},\P)$ with a filtration $(\mathcal{F}_t)_{t \in [0,T]}$ satisfying the usual conditions, i.e., completeness and right-continuity. We consider the SDE
\begin{equation}\label{eq:SDE}
  Y_t = y_0 + \int_0^t b(s,Y_{s-}) \dd s + \int_0^t \sigma(s,Y_{s-}) \dd X_s, \qquad t \in [0,T],
\end{equation}
where $y_0 \in \R^k$, $b \in C_b^2(\R^{k+1};\R^k)$ and $\sigma \in C^3_b(\R^{k+1};\mathcal{L}(\R^d;\R^k))$, and $\int_0^t \sigma(s,Y_{s-}) \dd X_s$ is defined as an It{\^o} integral. For a comprehensive introduction to stochastic It{\^o} integration and SDEs we refer, e.g., to the textbook \cite{Protter2005}. It is well known that the SDE \eqref{eq:SDE} possesses a unique (strong) solution (see, e.g., \cite[Chapter~V, Theorem~6]{Protter2005}), and that the semimartingale $X$ can be lifted to a random rough path via It{\^o} integration, by defining $\bX = (X,\mathbb{X}) \in \mathcal{D}^{p}([0,T];\R^d)$, $\P$-a.s., for any $p \in (2,3)$, where
\begin{equation}\label{eq:semimartingale lift}
  \mathbb{X}_{s,t} := \int_s^t (X_{r-} - X_s) \otimes \dd X_r = \int_s^t X_{r-} \otimes \dd X_r - X_s \otimes X_{s,t}, \qquad (s,t) \in \Delta_T;
\end{equation}
see \cite[Proposition~3.4]{Liu2018} or \cite[Theorem~6.5]{Friz2018}. It turns out that, if the semimartingale $X$ satisfies Property \textup{(RIE)} relative to $p \in (2,3)$ and a suitable sequence of partitions $(\cP^n)_{n \in \N}$, then the solutions to the SDE \eqref{eq:SDE} and to the RDE \eqref{eq:RDE Y} driven by the random rough path $\bX = (X,\mathbb{X})$ coincide $\P$-almost surely.

\begin{lemma}\label{lem:RDE and SDE agree}
  Let $p \in (2,3)$ and let $\cP^n = \{\tau^n_k\}$, $n \in \N$, be a sequence of adapted partitions (so that each $\tau^n_k$ is a stopping time), such that, for almost every $\omega \in \Omega$, $(\cP^n(\omega))_{n \in \N}$ is a sequence of (finite) partitions of $[0,T]$ with vanishing mesh size. Let $X$ be a c{\`a}dl{\`a}g semimartingale, and suppose that, for almost every $\omega \in \Omega$, the sample path $X(\omega)$ satisfies Property \textup{(RIE)} relative to $p$ and $(\cP^n(\omega))_{n \in \N}$.
  \begin{itemize}
    \item[(i)] The random rough paths $\bX = (X,\X)$, with $\X$ defined pathwise via \eqref{eq:RIE rough path}, and with $\X$ defined by stochastic integration as in \eqref{eq:semimartingale lift}, coincide $\P$-almost surely.
    \item[(ii)] The solution of the SDE \eqref{eq:SDE} driven by $X$, and the solution of the RDE~\eqref{eq:RDE Y} driven by the random rough path $\bX = (X,\X)$, coincide $\P$-almost surely. 
  \end{itemize}
\end{lemma}

\begin{proof}
  \emph{(i):} 
  By construction, the pathwise rough integral $\int_0^t X_u(\omega) \otimes \d X_u(\omega)$ constructed via Property \textup{(RIE)} is given by the limit as $n \to \infty$ of left-point Riemann sums:
  \begin{equation}\label{eq:Riemann sums for lift of X}
    \sum_{k=0}^{N_n-1} X_{\tau^n_k(\omega)}(\omega) \otimes X_{\tau^n_k(\omega) \wedge t,\tau^n_{k+1}(\omega) \wedge t}(\omega).
  \end{equation}
It is known that these Riemann sums also converge uniformly in probability to the It{\^o} integral $\int_0^t X_{u-} \otimes \d X_u$ (see, e.g., \cite[Chapter~II, Theorem~21]{Protter2005}), and the result thus follows from the (almost sure) uniqueness of limits.

  \emph{(ii):}
  In the following, we adopt the notation $J_F$ for the set of jump times of a path $F$, and we write $\liminf_{n \to \infty} \cP^n := \bigcup_{m \in \N} \, \bigcap_{n \geq m} \cP^n$.

Let $Y$ be the solution to the RDE \eqref{eq:RDE Y} driven by the random rough path $\bX = (X,\X)$. By the definition of $\X$ in \eqref{eq:RIE rough path}, it is straightforward to see that $\X_{t-,t} = 0$ for every $t \in (0,T]$. It then follows from the definition of rough integration that the integral $t \mapsto \int_0^t \sigma(s,Y_s) \dd \bX_s$ can only have a jump at the jump times of $X$, and it follows that the same is true of the solution $Y$ to the RDE \eqref{eq:RDE Y}, i.e., $J_Y \subseteq J_X$.

  Since the piecewise constant approximation $X^n$ of $X$ along $\cP^n$ converges uniformly to $X$ (by condition~(i) of Property \textup{(RIE)}), we have from Proposition~\ref{prop: uniform convergence} that $J_X \subseteq \liminf_{n \to \infty} \cP^n$. Since $J_Y \subseteq J_X$, we have that $J_Y \subseteq \liminf_{n \to \infty} \cP^n$. It then follows from Theorem~\ref{thm:rough int under RIE} that
  \begin{equation*}
    \int_0^t \sigma(s,Y_s) \dd \bX_s = \lim_{n \to \infty} \sum_{k=0}^{N_n-1} \sigma(\tau^n_k,Y_{\tau^n_k}) X_{\tau^n_k \wedge t,\tau^n_{k+1} \wedge t}.
  \end{equation*}
  Since these Riemann sums also converge in probability to the It{\^o} integral $\int_0^t \sigma(s,Y_{s-}) \dd X_s$ (see, e.g., \cite[Chapter~II, Theorem~21]{Protter2005}), these integrals coincide almost surely. We infer that $Y$ is also a solution of the SDE \eqref{eq:SDE}, which has a unique solution (by, e.g., \cite[Chapter~V, Theorem~6]{Protter2005}).
\end{proof}

As a consequence of Theorem~\ref{thm: Euler scheme convergence} and Lemma~\ref{lem:RDE and SDE agree}, for semimartingales which satisfy Property \textup{(RIE)} relative to a sequence of adapted partitions, the Euler scheme \eqref{eq:RDE Euler scheme} converges pathwise to the solution of the SDE \eqref{eq:SDE}. In the following subsections we verify Property \textup{(RIE)} for various semimartingales relative to suitable sequences of partitions, and derive the pathwise convergence rate of the associated Euler scheme with respect to the $p$-variation norm.

\subsection{Brownian motion}

We start with the most prominent example of a semimartingale, by taking $X = W$ to be a $d$-dimensional Brownian motion $W = (W_t)_{t \in [0,T]}$ with respect to the underlying filtration $(\mathcal{F}_t)_{t \in [0,T]}$.

\begin{proposition}\label{prop: BM satisfies RIE}
  Let $p \in (2,3)$ and let $\cP^n = \{0 = t_0^n < t_1^n < \cdots < t_{N_n}^n = T\}$, $n \in \N$, be a sequence of equidistant partitions of the interval $[0,T]$, so that, for each $n \in \N$, there exists some $\pi_n > 0$ such that $t^n_{i+1} - t^n_i = \pi_n$ for each $0 \leq i < N_n$. If $\pi_n^{2 - \frac{4}{p}} \log(n) \to 0$ as $n \to \infty$, then, for almost every $\omega \in \Omega$, the sample path $W(\omega)$ satisfies Property \textup{(RIE)} relative to $p$ and $(\cP^n)_{n \in \N}$.
\end{proposition}

\begin{proof}
  As stated in Remark~\ref{rem: Davie's criterion}, Davie's condition implies Property \textup{(RIE)}. While \cite[Appendix~B]{Perkowski2016} shows this for the sequence of partitions $(\cP_U^n)_{n\in\N}$, where $\cP^n_U = \{\frac{iT}{n} : i = 0, 1, \ldots, n\}$, i.e.~$\pi_n = \frac{T}{n}$, their proof actually holds for any sequence of equidistant partitions of the interval $[0,T]$. We therefore show the necessary condition proposed in \cite{Davie2008}, under the assumption that $\pi_n^{2-\frac{4}{p}} \log(n) \to 0$ as $n \to \infty$.

  More precisely, let $\mathbf{W} = (W,\mathbb{W})$ be the It\^o Brownian rough path lift of $W$. Write $\alpha := \frac{1}{p}$ and let $\beta \in (1 - \alpha,2\alpha)$. We show that, almost surely, there exists a constant $C > 0$ such that
  \begin{equation*}
    \Big|\sum_{m=k}^{\ell - 1} \mathbb{W}_{{t_m^n},t_{m+1}^n}^{ij}\Big| \leq C (\ell - k)^\beta \pi_n^{2\alpha},
  \end{equation*}
  for every $i, j = 1, \ldots, d$ and $n \in \N$, whenever $0 < k < \ell$ are integers such that $\ell \pi_n \leq T$.

  \emph{Step 1.} We recall that a (zero mean) random variable $Z$ is said to be \emph{sub-Gaussian} if its sub-Gaussian norm $\|Z\|_{\psi_2} := \inf\{z>0: \E[\exp(Z^2/z^2)] \leq 2\}$ is finite. It is well known that the sub-Gaussian property admits an equivalent formulation; namely, $Z$ is sub-Gaussian if and only if $\E[\exp(\lambda^2 Z^2)] \leq \exp(\lambda^2 K^2)$ holds for all $\lambda$ such that $|\lambda| \leq \frac{1}{K}$, for some $K > 0$. In this case we have $\|Z\|_{\psi_2} = K$ up to a multiplicative constant.

  We will prove that $\mathbb{W}_{t_m^n,{t_{m+1}^n}}^{ij}$, $m = k, \ldots, \ell-1$, are independent sub-Gaussian random variables with sub-Gaussian norm $\|\mathbb{W}_{t_m^n,t_{m+1}^n}^{ij}\|_{\psi_2} = C \pi_n$ for some $C > 0$.

  First, we note that, by \cite[Proposition~13.4]{Friz2010}, for all $m \in \N$, the random variables
  \[\frac{\mathbb{W}^{ij}_{t^n_m,t^n_{m+1}}}{t^n_{m+1} - t^n_m}\]
  are independent and identically distributed, with the same distribution as $\mathbb{W}^{ij}_{0,1}$, and that the latter satisfies $\E[\exp(\eta \mathbb{W}^{ij}_{0,1})] < \infty$ for some sufficiently small $\eta > 0$, which is equivalent to the Gaussian tail property, i.e., that $\|\mathbb{W}^{ij}_{0,1}\|_{L^q} \leq c \sqrt{q}$ for all $q \geq 1$, where the constant $c$ is independent of $q$; see \cite[Lemma~A.17]{Friz2010}. As a consequence, using the fact that $t^n_{m+1} - t^n_m = \pi_n$ for all $m$, and setting $q = 2\nu$, we deduce that
  \begin{equation}\label{eq: moment bound Brownian rough lift}
    \E[|\mathbb{W}_{t_m^n,t_{m+1}^n}^{ij}|^{2\nu}] \leq c^\nu \nu^\nu \pi_n^{2 \nu}, \qquad \nu \in \N,
  \end{equation}
  for a new constant $c > 0$ which does not depend on $\nu$.

  We now aim to show that there exists a constant $C > 0$ such that
  \begin{equation}\label{eq: moment generating function}
    \E[\exp(\lambda^2 (\mathbb{W}_{t_m^n,t_{m+1}^n}^{ij})^2)] \leq \exp(C^2 \pi_n^2 \lambda^2),
  \end{equation}
  for all $\lambda$ such that $|\lambda| \leq \frac{1}{C \pi_n}$, which then implies that $\mathbb{W}_{t_m^n,t_{m+1}^n}^{ij}$ is sub-Gaussian with norm $\|\mathbb{W}_{t_m^n,t_{m+1}^n}^{ij}\|_{\psi_2} = C \pi_n$, up to a multiplicative constant which we may then absorb into $C$. Using the Taylor expansion for the exponential function, we get, for $\lambda \in \R$, that
  \begin{equation*}
    \E[\exp(\lambda^2 (\mathbb{W}_{t_m^n,t_{m+1}^n}^{ij})^2)] = \E \bigg[1 + \sum_{\nu=1}^\infty \frac{\lambda^{2\nu} (\mathbb{W}_{t_m^n,t_{m+1}^n}^{ij})^{2\nu}}{\nu!} \bigg] = 1 + \sum_{\nu=1}^\infty \frac{\lambda^{2\nu} \E[(\mathbb{W}_{t_m^n,t_{m+1}^n}^{ij})^{2\nu}]}{\nu!}.
  \end{equation*}
  By the bound in \eqref{eq: moment bound Brownian rough lift} and Stirling's approximation (which implies in particular that $\nu! \geq (\frac{\nu}{e})^\nu$ for all $\nu \geq 1$), we obtain
  \begin{equation*}
    \E[\exp(\lambda^2 (\mathbb{W}_{t_m^n,t_{m+1}^n}^{ij})^2)] \leq 1 + \sum_{\nu=1}^\infty (e c\lambda^2 \pi_n^2)^\nu = \frac{1}{1 - ec\lambda^2 \pi_n^2} \leq \exp (2ec \lambda^2 \pi_n^2),
  \end{equation*}
  which is valid provided that
  \begin{equation}\label{eq:lambda bound}
    ec \lambda^2 \pi_n^2 \leq \frac{1}{2},
  \end{equation}
  since $\frac{1}{1-x} \leq \exp(2x)$ for $x \in [0,\frac{1}{2}]$. We then obtain \eqref{eq: moment generating function} by choosing $C = \sqrt{2ec}$, and note that then \eqref{eq:lambda bound} does indeed hold when $|\lambda| \leq \frac{1}{C \pi_n}$.

  \emph{Step 2.} Let $C > 0$ be the constant found above, so that $\|\mathbb{W}_{t_m^n,t_{m+1}^n}^{ij}\|_{\psi_2} = C\pi_n$. Then Hoeffding's inequality (see, e.g., \cite[Theorem~2.6.2]{Vershynin2018}) gives
  \begin{align*}
    \P \bigg(\bigg|\sum_{m=k}^{\ell-1} \mathbb{W}_{t_m^n,t_{m+1}^n}^{ij}\bigg| \geq C (\ell - k)^\beta \pi_n^{2\alpha}\bigg) &\leq 2 \exp \bigg(-\frac{C^2 (\ell - k)^{2\beta} \pi_n^{4\alpha}}{\sum_{m=k}^{\ell-1} \|\mathbb{W}_{t_m^n,t_{m+1}^n}^{ij}\|_{\psi_2}^2}\bigg)\\
    &= 2 \exp \bigg(-\frac{(\ell - k)^{2\beta - 1}}{{\pi_n}^{2 - 4\alpha}}\bigg).
  \end{align*}
  Since $\beta > 1 - \alpha > \frac{1}{2}$, we can bound this further by
  \begin{equation*}
    \P \bigg(\bigg|\sum_{m=k}^{\ell-1} \mathbb{W}_{t_m^n,t_{m+1}^n}^{ij}\bigg| \geq C (\ell - k)^\beta \pi_n^{2\alpha}\bigg) \leq 2 \exp \Big(-\frac{1}{{\pi_n}^{2-4\alpha}}\Big) = 2 n^{-\frac{1}{\gamma_n}},
  \end{equation*}
  where we denote $\gamma_n = {\pi_n}^{2-4\alpha} \log(n)$. Since, by assumption, $\gamma_n \to 0$ as $n \to \infty$, we have that $\frac{1}{\gamma_n} > 1$ for all sufficiently large $n \in \N$, and hence that the series $\sum_{n \in \N} n^{-\frac{1}{\gamma_n}}$ is absolutely convergent. The desired statement then follows from the Borel--Cantelli lemma.
\end{proof}

\begin{remark}
  Proposition~\ref{prop: BM satisfies RIE} can be generalized to any sequence of partitions $(\cP^n)_{n \in \N}$, which possibly consists of non-equidistant partitions, such that $|\cP^n|^{2-\frac{4}{p}} \log(n) \to 0$ as $n \to \infty$, provided that there exists a positive number $\eta > 0$ such that
  \begin{equation*}
    \frac{|\cP^n|}{\min_{0 \leq k < N_n} |t^n_{k+1} - t^n_k|} \leq \eta
  \end{equation*}
  for every $n \in \N$. This additional condition requires that the sequence $(\cP^n)_{n \in \N}$ is a ``balanced partition sequence'' in the sense of \cite{Cont2023}.
\end{remark}

\begin{remark}
  Combining Proposition~\ref{prop: BM satisfies RIE} with Lemma~\ref{lem:rough path approximation}, we infer that the piecewise constant approximations of a Brownian motion along equidistant partitions converge to its It\^o rough path lift, which, as far as we are aware, is a novel construction of this lift. Existing approximations of Brownian rough path are all continuous approximations, such as piecewise linear or mollifier approximations---cf.~\cite{Friz2010}---which play a crucial role, e.g., in the rough path based proofs of Wong--Zakai results, support theorems and large deviation principles.
\end{remark}

\begin{corollary}\label{cor: convergence rate for Brownian motion}
  Let $p \in (2,3)$ and let $\cP_U^n = \{0 = t_0^n < t_1^n < \cdots < t_n^n = T\}$, $n \in \N$, with $t_i^n = \frac{iT}{n}$, be the sequence of equidistant partitions with width $\frac{T}{n}$ of the interval $[0,T]$. Let $Y$ be the solution of the SDE \eqref{eq:SDE} driven by a Brownian motion $W$, and let $Y^n$ be the corresponding Euler approximation along $\cP_U^n$, as defined in \eqref{eq:RDE Euler scheme}. For any $p' \in (p,3)$, $q \in (1,2)$ and $\beta \in (1 - \frac{1}{p},\frac{2}{p})$ such that $\frac{1}{p'} + \frac{1}{q} > 1$, there exists a random variable $C$, which does not depend on $n$, such that
  \begin{equation}\label{eq:Brownian rate convergence}
    \|Y^n - Y\|_{p'} \leq C (n^{-(1-\frac{1}{q})} + n^{-(\frac{2}{p}-\beta)(1-\frac{p}{p'})}), \qquad n \in \N.
  \end{equation}
\end{corollary}

\begin{proof}
  Since $|\cP_U^n| = \frac{T}{n}$, we have that $|\cP_U^n|^{2-\frac{4}{p}} \log(n) \to 0$ as $n \to \infty$. Thus, by Proposition~\ref{prop: BM satisfies RIE}, for almost every $\omega \in \Omega$, the sample path $W(\omega)$ satisfies Property \textup{(RIE)} relative to $p$ and $(\cP_U^n)_{n \in \N}$, which allows us to apply the result of Theorem~\ref{thm: Euler scheme convergence}.

  Since the sample paths of $W$ are almost surely $\frac{1}{p}$-H\"older continuous, it is easy to see that
  \begin{equation*}
    \|W^n - W\|_\infty \lesssim n^{-\frac{1}{p}}, \qquad n \in \N,
  \end{equation*}
  where the implicit multiplicative constant is a random variable which does not depend on $n$. Moreover, by \cite[Appendix~B]{Perkowski2016}, the left-point Riemann sums along $(\cP_U^n)_{n \in \N}$ converge uniformly as $n \to \infty$, with rate $n^{-(\frac{2}{p}-\beta)}$ for $\beta \in (1 - \frac{1}{p},\frac{2}{p})$, i.e.,
  \begin{equation*}
    \bigg\|\int_0^\cdot W_u^n \otimes \d W_u - \int_0^\cdot W_u \otimes \d W_u\bigg\|_\infty \lesssim n^{-(\frac{2}{p}-\beta)}, \qquad n \in \N.
  \end{equation*}
  Hence, by Theorem~\ref{thm: Euler scheme convergence}, we get that
  \begin{equation*}
    \|Y^n - Y\|_{p'} \lesssim n^{-(1-\frac{1}{q})} + n^{-\frac{1}{p}(1-\frac{p}{p'})} + n^{-(\frac{2}{p}-\beta)(1-\frac{p}{p'})}.
  \end{equation*}
  Since $\frac{1}{p} < 1 - \frac{1}{p} < \beta$ for $p \in (2,3)$, this gives the rate of convergence in \eqref{eq:Brownian rate convergence}.
\end{proof}

\subsection{It{\^o} processes}

In this subsection we let $X$ be an It{\^o} process. More precisely, we suppose that
\begin{equation}\label{eq: Ito process}
  X_t = x_0 + \int_0^t b_r \dd r + \int_0^t H_r \dd W_r, \qquad t \in [0,T],
\end{equation}
for some $x_0 \in \R^d$, and some locally bounded predictable integrands $b \colon \Omega \times [0,T] \to \R^d$ and $H \colon \Omega \times [0,T] \to \cL(\R^m;\R^d)$, where $W$ is an $\R^m$-valued Brownian motion. We consider the sequence of dyadic partitions $(\mathcal{P}^n_D)_{n \in \N}$ of $[0,T]$, given by
\begin{equation}\label{eq: dyadic partitions}
  \cP^n_D := \{0 = t_0^n < t_1^n < \cdots < t_{2^n}^n = T\} \qquad \text{with} \quad t_k^n := k 2^{-n} T \quad \text{for} \quad k = 0, 1, \ldots, 2^n.
\end{equation}

In the next proposition we will show that $X$ satisfies Property \textup{(RIE)} along the sequence of partitions $(\cP^n_D)_{n \in \N}$, and establish the rate of convergence of the associated Euler scheme. Note that, in contrast to the proof of Proposition~\ref{prop: BM satisfies RIE}, for general It{\^o} processes we cannot rely on the concentration of measure inequality for sub-Gaussian distributions.

\begin{proposition}\label{prop: Ito diffusion has RIE}
  Let $p \in (2,3)$ and let $X$ be an It{\^o} process of the form in \eqref{eq: Ito process}. Let $Y$ be the solution of the SDE \eqref{eq:SDE} driven by $X$, and let $Y^n$ denote the corresponding Euler approximation, as defined in \eqref{eq:RDE Euler scheme}, based on $X$ and the sequence of dyadic partitions $(\cP_D^n)_{n \in \N}$.
  \begin{itemize}
    \item[(i)] For almost every $\omega \in \Omega$, the sample path $X(\omega)$ satisfies Property \textup{(RIE)} relative to $p$ and $(\cP_D^n)_{n \in \N}$.
    \item[(ii)] For any $p' \in (p,3)$ and $q \in (1,2)$ such that $\frac{1}{p'} + \frac{1}{q} > 1$, and any $\epsilon \in (0,1)$, there exists a random variable $C$, which does not depend on $n$, such that
    \begin{equation}\label{eq: Ito proc p' bound}
      \|Y^n - Y\|_{p'} \leq C (2^{-n(1-\frac{1}{q})} + 2^{-n(\frac{1}{p}-\frac{1}{p'})} + 2^{-\frac{n}{2}(1-\epsilon)(1-\frac{p}{p'})}), \qquad n \in \N,
    \end{equation}
    and
    \begin{equation}\label{eq: Ito proc 3 bound}
      \|Y^n - Y\|_3 \leq C 2^{-n(\frac{1}{6}-\epsilon)}, \qquad n \in \N.
    \end{equation}
  \end{itemize}
\end{proposition}

\begin{proof}
  \emph{(i):}
  By a localization argument, we may assume that $b$ and $H$ are globally bounded. Let
  \begin{equation*}
    A_t := \int_0^t b_r \dd r \qquad \text{and} \qquad M_t := \int_0^t H_r \dd W_r
  \end{equation*}
  for $t \in [0,T]$, so that $X = x_0 + A + M$, and recall that we denote the piecewise constant approximation of $X$ along $\cP_D^n$ by
  \begin{equation*}
    X^n_t = X_{T} \1_{T}(t) + \sum_{k=0}^{2^n-1} X_{t^n_k} \1_{[t^n_k,t^n_{k+1})}(t), \qquad t \in [0,T],
  \end{equation*}
  with $t^n_k = k 2^{-n} T$ for each $k = 0, 1, \ldots, 2^n$ and $n \in \N$. Note that, by the uniform continuity of the sample paths of $X$, it is clear that $X^n$ converges uniformly to $X$ almost surely as $n \to \infty$.

  \emph{Step 1.} In this step we verify that the sample paths of $X$ are almost surely $\frac{1}{p}$-H{\"o}lder continuous. This is a standard application of the Burkholder--Davis--Gundy inequality. Indeed, for any $q \geq 1$, using the boundedness of $H$, and writing $[\cdot]$ for quadratic variation, we have that
  \begin{equation*}
    \E[|M_t - M_s|^q] = \E \bigg[\bigg|\int_s^t H_u \dd W_u\bigg|^q\bigg] \lesssim \E \bigg[\bigg[\int_0^\cdot H_u \dd W_u\bigg]_{s,t}^{\frac{q}{2}}\bigg] \lesssim |t - s|^{\frac{q}{2}},
  \end{equation*}
  so that $\|M_t - M_s\|_{L^q} \lesssim |t - s|^{\frac{1}{2}}$. By the Kolmogorov continuity theorem (see, e.g., \cite[Theorem~A.10]{Friz2010}), it follows that $\E[\|M\|_{\gamma\text{-H\"ol}}] < \infty$, where $\|\cdot\|_{\gamma\text{-H\"ol}}$ denotes the $\gamma$-H\"older norm, for any $\gamma \in [0,\frac{1}{2}-\frac{1}{q})$, which, taking $q$ sufficiently large, implies that the sample paths of $M$ are almost surely $\frac{1}{p}$-H{\"o}lder continuous. Since $A = \int_0^\cdot b_r \dd r$ with the bounded integrand $b$, the sample paths of $A$ are Lipschitz continuous, and thus also $\frac{1}{p}$-H\"older continuous.

  \emph{Step 2.} In this step we show that, almost surely, $\int_0^\cdot X_u^n \otimes \d X_u$ converges uniformly to the It{\^o} integral $\int_0^\cdot X_u \otimes \d X_u$ as $n \to \infty$. For this purpose, we write $X^n = x_0 + A^n + M^n$, where
  \begin{equation*}
    A^n_t := A_T \1_{\{T\}}(t) + \sum_{k=0}^{2^n-1} A_{t^n_k} \1_{[t^n_k,t^n_{k+1})}(t) \quad \text{and} \quad M^n_t := M_T \1_{\{T\}}(t) + \sum_{k=0}^{2^n-1} M_{t^n_k} \1_{[t^n_k,t^n_{k+1})}(t),
  \end{equation*}
  for $t \in [0,T]$. Since $X = x_0 + A + M$, we obtain
  \begin{equation}\label{eq:L2 unif conv int Ito}
    \begin{split}
    &\E\bigg[\bigg\|\int_0^\cdot X_u^n \otimes \d X_u - \int_0^\cdot X_u \otimes \d X_u\bigg\|_\infty^2\bigg]\\
    &\quad\lesssim \E\bigg[\bigg\|\int_0^\cdot (A_u^n - A_u) \otimes \d A_u\bigg\|_\infty^2\bigg] + \E\bigg[\bigg\|\int_0^\cdot (M_u^n - M_u) \otimes \d A_u\bigg\|_\infty^2\bigg]\\
    &\quad\quad + \E\bigg[\bigg\|\int_0^\cdot (A_u^n -  A_u) \otimes \d M_u\bigg\|_\infty^2\bigg] + \E\bigg[\bigg\|\int_0^\cdot (M_u^n - M_u) \otimes \d M_u\bigg\|_\infty^2\bigg].
    \end{split}
  \end{equation}
  Using the Burkholder--Davis--Gundy inequality, the fact that $[M] = [\int_0^\cdot H_t \dd W_t] = \int_0^\cdot |H_t|^2 \dd t$, and the boundedness of $H$, we can bound
  \begin{align*}
    &\E\bigg[\bigg\|\int_0^\cdot (M_u^n - M_u) \otimes \d M_u\bigg\|_\infty^2\bigg] \lesssim \E\bigg[\int_0^T |M^n_t - M_t|^2 \dd [M]_t\bigg]\\
    &\quad\lesssim \int_0^T \E[|M^n_t - M_t|^2] \dd t = \sum_{k=0}^{2^n-1} \int_{t^n_k}^{t^n_{k+1}} \E[|M_{t^n_k} - M_t|^2] \dd t \lesssim \sum_{k=0}^{2^n-1} \int_{t^n_k}^{t^n_{k+1}} \E[|[M]_{t^n_k,t}|] \dd t\\
    &\quad= \sum_{k=0}^{2^n-1} \int_{t^n_k}^{t^n_{k+1}} \E\bigg[\int_{t^n_k}^t |H_r|^2 \dd r\bigg] \dd t \lesssim \sum_{k=0}^{2^n-1} \int_{t^n_k}^{t^n_{k+1}} (t - t^n_k) \dd t \leq \sum_{k=0}^{2^n-1} (t^n_{k+1} - t^n_k)^2 = 2^{-n}.
  \end{align*}
  The other terms on the right-hand side of \eqref{eq:L2 unif conv int Ito} can be bounded similarly by $2^{-n}$, up to a constant which does not depend on $n$, and we thus have that
  \begin{equation*}
    \E \bigg[\bigg\|\int_0^\cdot X_u^n \otimes \d X_u - \int_0^\cdot X_u \otimes \d X_u\bigg\|_\infty^2\bigg] \lesssim 2^{-n},
  \end{equation*}
  for every $n \in \N$. By Markov's inequality, for any $\epsilon \in (0,1)$, we then have that
  \begin{align*}
    &\P\bigg(\bigg\|\int_0^\cdot X_u^n \otimes \d X_u - \int_0^\cdot X_u \otimes \d X_u\bigg\|_\infty \geq 2^{-\frac{n}{2}(1-\epsilon)}\bigg)\\
    &\quad\leq 2^{n(1-\epsilon)} \E \bigg[\bigg\|\int_0^\cdot X_u^n \otimes \d X_u - \int_0^\cdot X_u \otimes \d X_u\bigg\|_\infty^2\bigg] 
    \lesssim 2^{n(1-\epsilon)} 2^{-n} = 2^{-n\epsilon}.
  \end{align*}  
  It then follows from the Borel--Cantelli lemma that, almost surely,
  \begin{equation}\label{eq: convergence rate of discrete integral to Ito integral}
    \bigg\|\int_0^\cdot X_u^n \otimes \d X_u - \int_0^\cdot X_u \otimes \d X_u\bigg\|_\infty < 2^{-\frac{n}{2}(1-\epsilon)}
  \end{equation}
  for all sufficiently large $n$, which implies the desired convergence.

  \emph{Step 3.} Let $\epsilon \in (0,1)$ and $\rho = 2 + \frac{(1-\epsilon)(p-2)}{4} \in (2,3)$. We infer from Step~1 above that the sample paths of $X$ are almost surely $\frac{1}{\rho}$-H\"older continuous, from which it follows that
  \begin{equation*}
    |X_{s,t}| \lesssim |t - s|^{\frac{1}{\rho}},
  \end{equation*}
  where the implicit multiplicative constant is a random variable which does not depend on $s$ or $t$. Proceeding as in the proof of \cite[Lemma~3.2]{Liu2018}, we can show, for any $0 \leq k < \ell \leq 2^{n}$, and writing $N = \ell - k = 2^n |t^n_\ell - t^n_k| T^{-1}$, that
  \begin{equation*}
    \bigg|\int_{t^n_k}^{t^n_\ell} X^n_u \otimes \d X_u - X_{t^n_k} \otimes X_{t^n_k,t^n_\ell}\bigg| \lesssim N^{1-\frac{2}{\rho}} |t^n_\ell - t^n_k|^{\frac{2}{\rho}} \lesssim 2^{n(1 - \frac{2}{\rho})} |t^n_\ell - t^n_k| \leq 2^{n(\rho-2)} |t^n_\ell - t^n_k|.
  \end{equation*}
  If $2^{-n} \geq |t^n_\ell - t^n_k|^{\frac{4}{p(1-\epsilon)}}$, then it follows that
  \begin{equation*}
    \bigg|\int_{t^n_k}^{t^n_\ell} X^n_u \otimes \d X_u - X_{t^n_k} \otimes X_{t^n_k,t^n_\ell}\bigg| \lesssim |t^n_\ell - t^n_k|^{1 - \frac{4}{p(1-\epsilon)}(\rho-2)} = |t^n_\ell - t^n_k|^{\frac{2}{p}}.
  \end{equation*}
  We will now aim to obtain the same estimate in the case that $2^{-n} < |t^n_\ell - t^n_k|^{\frac{4}{p(1-\epsilon)}}$. To this end, let $\X$ denote the second level component of the It\^o rough path lift of $X$, as defined in \eqref{eq:semimartingale lift}. It follows from the Kolmogorov criterion for rough paths (see \cite[Theorem~3.1]{Friz2020}) that
  \begin{equation}\label{eq:bound Ito pro iter int}
    |\X_{s,t}| \lesssim |t - s|^{\frac{2}{p}},
  \end{equation}
  where the implicit multiplicative constant is a random variable which does not depend on $s$ or $t$. Using the bounds in \eqref{eq: convergence rate of discrete integral to Ito integral} and \eqref{eq:bound Ito pro iter int}, we then have, for all sufficiently large $n$, that
  \begin{align*}
    &\bigg|\int_{t^n_k}^{t^n_\ell} X^n_u \otimes \d X_u - X_{t^n_k} \otimes X_{t^n_k,t^n_\ell}\bigg|\\  
    &\quad= \bigg|\int_{t^n_k}^{t^n_\ell} X^n_u \otimes \d X_u - \int_{t^n_k}^{t^n_\ell} X_u \otimes \d X_u + \int_{t^n_k}^{t^n_\ell} X_u \otimes \d X_u - X_{t^n_k} \otimes X_{t^n_k,t^n_\ell}\bigg|\\
    &\quad\leq 2 \bigg\|\int_0^\cdot X^n_u \otimes \d X_u - \int_0^\cdot X_u \otimes \d X_u\bigg\|_\infty + |\X_{t^n_k,t^n_\ell}|\\
    &\quad\lesssim 2^{-\frac{n}{2}(1-\epsilon)} + |t^n_\ell - t^n_k|^{\frac{2}{p}}\\
    &\quad\lesssim |t^n_\ell - t^n_k|^{\frac{2}{p}}.
  \end{align*}

  We have thus established that
  \begin{equation*}
    \bigg|\int_{t^n_k}^{t^n_\ell} X^n_u \otimes \d X_u - X_{t^n_k} \otimes X_{t^n_k,t^n_\ell}\bigg|^{\frac{p}{2}} \lesssim |t^n_\ell - t^n_k|
  \end{equation*}
  holds for all $0 \leq k < \ell \leq 2^n$ and all sufficiently large $n$. It follows that there exists a random control function $w(s,t) := c |t-s|$, for some random variable $c$, such that
  \begin{equation*}
    \sup_{(s,t) \in \Delta_T} \frac{|X_{s,t}|^p}{w(s,t)} + \sup_{n \in \N} \, \sup_{0 \leq k < \ell \leq 2^n} \frac{|\int_{t^n_k}^{t^n_\ell} X^n_u \otimes \d X_u - X_{t^n_k} \otimes X_{t^n_k,t^n_\ell}|^{\frac{p}{2}}}{w(t^n_k,t^n_\ell)} \leq 1
  \end{equation*}
  holds almost surely. This means that, for almost every $\omega \in \Omega$, the sample path $X(\omega)$ satisfies Property \textup{(RIE)} relative to any $p \in (2,3)$ and the sequence of dyadic partitions $(\cP_D^n)_{n \in \N}$.

  \emph{(ii):}
  Since the sample paths of $X$ are almost surely $\frac{1}{p}$-H\"older continuous (by Step~1 above), it is straightforward to see that
  \begin{equation*}
    \|X^n - X\|_\infty \lesssim 2^{-\frac{n}{p}}, \qquad n \in \N, 
  \end{equation*}
  and, recalling \eqref{eq: convergence rate of discrete integral to Ito integral}, we have that
  \begin{equation*}
    \bigg\|\int_0^\cdot X_u^n \otimes \d X_u - \int_0^\cdot X_u \otimes \d X_u\bigg\|_\infty \lesssim 2^{-\frac{n}{2}(1-\epsilon)}, \qquad n \in \N.
  \end{equation*}
  Hence, by Theorem~\ref{thm: Euler scheme convergence}, we deduce that
  \begin{equation*}
    \|Y^n - Y\|_{3} \leq \|Y^n - Y\|_{p'} \lesssim 2^{-n(1 - \frac{1}{q})} + 2^{-\frac{n}{p}(1 - \frac{p}{p'})} + 2^{-\frac{n}{2} (1-\epsilon) (1 - \frac{p}{p'})},
  \end{equation*}
  for any $p' \in (p,3)$ and $q \in (1,2)$ such that $\frac{1}{p'} + \frac{1}{q} > 1$, which leads to \eqref{eq: Ito proc p' bound}. Choosing $p$ sufficiently close to $2$, $p'$ to $3$, and $q$ to $\frac{3}{2}$, and replacing $\epsilon$ by $6\epsilon$, then reveals \eqref{eq: Ito proc 3 bound}.
\end{proof}

\subsection{L{\'e}vy processes}\label{section: Levy processes}

Let $L = (L_t)_{t \in [0,T]}$ be a $d$-dimensional L\'evy process with characteristics $(\lambda, \Sigma, \nu)$. In this section, we shall work under the assumption that $\int_{|x| < 1} |x|^q \hspace{1pt} \nu(\d x) < \infty$ for some $q \in [1,2)$.

By the L\'evy--It\^o decomposition (see, e.g., \cite[Theorem~2.4.16]{Applebaum2009}), there exists a Brownian motion $W$ with covariance matrix $\Sigma$, and an independent Poisson random measure $\mu$ on $[0,T] \times (\R^d \setminus \{0\})$ with compensator $\nu$, such that $L = W + \phi$, where
\begin{equation}\label{eq: defn phi Levy}
\phi_t = \lambda t + \int_{|x| \geq 1} x \hspace{1pt} \mu(t,\d x) + \int_{|x| < 1} x \hspace{1pt} (\mu(t,\d x) - t \hspace{1pt} \nu(\d x)), \qquad t \in [0,T].
\end{equation}
Since $\int_{|x| < 1} |x|^q \hspace{1pt} \nu(\d x) < \infty$, we have that $\phi(\omega) \in D^q([0,T];\R^d)$ for almost every $\omega \in \Omega$; see \cite[Theorem~2.4.25]{Applebaum2009} and \cite[Th\'eor\`eme~IIIb]{Bretagnolle1972}.

Let $(\cP_D^n)_{n \in \N}$ be the dyadic partitions of $[0,T]$, as defined in \eqref{eq: dyadic partitions}. For each $n \in \N$, we also let $J^n = \{t \in (0,T] \, : \, |\Delta \phi_t| \geq 2^{-n}\}$, where $\Delta \phi_t = \phi_t - \phi_{t-}$ denotes the jump of $\phi$ at time $t$, and we let
\begin{equation}\label{eq: Levy partitions}
  \cP_L^n = \cP_D^n \cup J^n.
\end{equation}
We will consider $(\cP_L^n)_{n \in \N}$ as our sequence of adapted partitions, noting in particular that, for almost every $\omega \in \Omega$, $(\cP_L^n(\omega))_{n \in \N}$ is a nested sequence of (finite) partitions with vanishing mesh size, and that $\{t \in (0,T] : L_{t-}(\omega) \neq L_t(\omega)\} \subseteq \cup_{n \in \N} \cP_L^n(\omega)$.

\begin{remark}\label{remark: must include jump times}
  In order to obtain pointwise convergence of an Euler scheme, it is necessary that the jump times of the driving signal belong to the partitions used to construct the discretization, a fact which follows immediately from Proposition~\ref{prop: uniform convergence}, necessitating the inclusion of the jump times $(J^n)_{n \in \N}$ above.
\end{remark}

\begin{proposition}\label{prop: Levy process has RIE}
  Let $L$ be a $d$-dimensional L\'evy process with characteristics $(\lambda, \Sigma, \nu)$, and assume that $\int_{|x| < 1} |x|^q \hspace{1pt} \nu(\d x) < \infty$ for some $q \in [1,2)$. Let $p \in (2,3)$ such that $\frac{1}{p} + \frac{1}{q} > 1$. Let $Y$ be the solution to the SDE \eqref{eq:SDE} driven by $L$, and let $Y^n$ be the corresponding Euler approximation along $\cP_L^n$, as defined in \eqref{eq:RDE Euler scheme}.
  \begin{itemize}
    \item[(i)] For almost every $\omega \in \Omega$, the sample path $L(\omega)$ satisfies Property \textup{(RIE)} relative to $p$ and $(\cP_L^n(\omega))_{n \in \N}$.
    \item[(ii)] For any $p' \in (p,3)$ and $q' \in (q,2)$ such that $\frac{1}{p'} + \frac{1}{q'} > 1$, any $\gamma \in (0,\frac{1}{p})$, and any $\delta \in (0,1-\frac{q}{2})$, there exists a random variable $C$, which does not depend on $n$, such that
    \begin{equation*}
      \|Y^n - Y\|_{p'} \leq C \Big(2^{-n(1-\frac{1}{q'})} + (2^{-n(\frac{1}{p} - \gamma)} + 2^{-n(\frac{1}{p} - \frac{1}{p'})} + 2^{-n\delta(1-\frac{q}{q'})})^{1-\frac{p}{p'}}\Big), \qquad n \in \N.
    \end{equation*}
  \end{itemize}
\end{proposition}

To prove this statement, we need the following lemma.

\begin{lemma}\label{lemma: BM satisfies RIE wrt jump times of Levy process}
  Let $p \in (2,3)$, let $W$ be a $d$-dimensional Brownian motion with covariance matrix $\Sigma$, and let $(\cP_L^n)_{n \in \N}$ be the sequence of adapted partitions defined in \eqref{eq: Levy partitions}. For almost every $\omega \in \Omega$, the sample path $W(\omega)$ satisfies Property \textup{(RIE)} relative to $p$ and $(\cP_L^n(\omega))_{n \in \N}$.
\end{lemma}

\begin{proof}
  We need to verify each of the conditions (i)--(iii) in Property \textup{(RIE)}.

  \emph{(i):}
  Since the sample paths of $W$ are uniformly continuous on the compact interval $[0,T]$, it is straightforward to see that $W^n(\omega) \to W(\omega)$ uniformly as $n \to \infty$ for almost every $\omega \in \Omega$, where $W^n$ denotes the piecewise constant approximation of $W$ along $\cP_L^n$.

  \emph{(ii):}
  It follows from the Kolmogorov continuity criterion that the sample paths of Brownian motion are almost surely $\frac{1}{p}$-H\"older continuous, and that the H\"older constant $\|W\|_{\frac{1}{p}\textup{-H\"ol}}$ has finite moments of all orders (see, e.g., \cite[Theorem~A.1]{Bartl2019b}). Applying the Burkholder--Davis--Gundy inequality, we then have that
  \begin{align*}
    \E \bigg[\bigg\|\int_0^\cdot W_u^n \otimes \d W_u - \int_0^\cdot W_u \otimes \d W_u\bigg\|_\infty^2\bigg] &\lesssim \E \bigg[\int_0^T |W_t^n - W_t|^2 \dd t\bigg]\\
    &\leq \E \bigg[\|W\|_{\frac{1}{p}\textup{-H\"ol}}^2 \int_0^T |\cP_L^n|^{\frac{2}{p}} \dd t\bigg] \lesssim \E [\|W\|_{\frac{1}{p}\textup{-H\"ol}}^2] 2^{-\frac{2n}{p}}.
  \end{align*}
  Let $\gamma \in (0,\frac{1}{p})$ and $\epsilon = 1 - \frac{2}{p} + 2\gamma \in (1 - \frac{2}{p},1)$. By Markov's inequality, we infer that
  \begin{equation*}
     \P \bigg(\bigg\|\int_0^\cdot W^n_u \otimes \d W_u - \int_0^\cdot W_u \otimes \d W_u\bigg\|_\infty \geq 2^{-\frac{n}{2}(1-\epsilon)}\bigg) \lesssim 2^{-\frac{2n}{p} + n(1 - \epsilon)} = 2^{-2n\gamma}.
  \end{equation*}
  By the Borel--Cantelli lemma, we then have that, almost surely,
  \begin{equation}\label{eq: Borel Cantelli bound for BM part}
    \bigg\|\int_0^\cdot W^n_u \otimes \d W_u - \int_0^\cdot W_u \otimes \d W_u\bigg\|_\infty < 2^{-\frac{n}{2}(1-\epsilon)}
  \end{equation}
  for all sufficiently large $n$. It follows that $(\int_0^\cdot W^n_u \otimes \d W_u)(\omega)$ converges uniformly to $(\int_0^\cdot W_u \otimes \d W_u)(\omega)$ as $n \to \infty$ for almost every $\omega \in \Omega$.

  \emph{(iii):}
  Let $\rho = 2 + \frac{(1-\epsilon)(p-2)}{4} \in (2,3)$. Since the sample paths of $W$ are almost surely $\frac{1}{\rho}$-H\"older continuous, it follows that
  \begin{equation*}
    |W_{s,t}|^{\rho} \lesssim |t-s|,
  \end{equation*}
  where the implicit multiplicative constant is a random variable which does not depend on $s$ or $t$. Proceeding as in the proof of \cite[Lemma~3.2]{Liu2018}, we can show, for any $0 \leq k < \ell$, and writing $N = \ell - k$, we can show that
  \begin{equation*}
    \bigg|\int_{t_k^n}^{t_\ell^n} W_u^n \otimes \d W_u - W_{t_k^n} \otimes W_{t_k^n,t_\ell^n}\bigg| \lesssim N^{1-\frac{2}{\rho}} |t_\ell^n - t_k^n|^{\frac{2}{\rho}},
  \end{equation*}
  where $\{0 = t^n_0 < t^n_1 < \cdots\}$ are the partition points of $\cP_L^n(\omega)$ for some (here fixed) $\omega \in \Omega$. Using $|\cdot|$ here to denote the cardinality of a set, we note that the number $N$ can be bounded by
  \begin{align*}
    N &\leq |\cP_D^n(\omega) \cap (t_k^n,t_\ell^n]| + |J^n(\omega) \cap (t_k^n,t_\ell^n]| \leq 2^n T^{-1} |t_\ell^n - t_k^n| + 2^{nq} \sum_{t \in J^n(\omega) \cap (t_k^n,t_\ell^n]} |\Delta \phi_t(\omega)|^q\\
    &\lesssim 2^n |t_\ell^n - t_k^n| + 2^{nq} \|\phi(\omega)\|_{q,[t^n_k,t^n_\ell]}^q \leq 2^{n\rho} c(t_k^n,t_\ell^n),
  \end{align*}
  where $c$ is the control function defined by $c(s,t) := |t - s| + \|\phi(\omega)\|_{q,[s,t]}^q$ for $(s,t) \in \Delta_T$. If $2^{-n} \geq c(t_k^n,t_\ell^n)^{\frac{4}{p(1-\epsilon)}}$, this implies that
  \begin{equation*}
    \bigg|\int_{t_k^n}^{t_\ell^n} W_u^n \otimes \d W_u - W_{t_k^n} \otimes W_{t_k^n,t_\ell^n}\bigg| \lesssim 2^{n(\rho - 2)} c(t_k^n,t_\ell^n) \leq c(t_k^n,t_\ell^n)^{1 - \frac{4}{p(1-\epsilon)}(\rho-2)} = c(t_k^n,t_\ell^n)^{\frac{2}{p}}.
  \end{equation*}
  In the case that $2^{-n} < c(t_k^n,t_\ell^n)^{\frac{4}{p(1-\epsilon)}}$, we can follow the same argument as in Step~3 of the proof of part~(i) of Proposition~\ref{prop: Ito diffusion has RIE} (using in particular the bound in \eqref{eq: Borel Cantelli bound for BM part}) to obtain again that
  \begin{equation*}
    \bigg|\int_{t_k^n}^{t_\ell^n} W_u^n \otimes \d W_u - W_{t_k^n} \otimes W_{t_k^n,t_\ell^n}\bigg| \lesssim c(t_k^n,t_\ell^n)^{\frac{2}{p}},
  \end{equation*}
  where, as usual, the implicit multiplicative constant depends on $\omega$, but not on $n$.

  It follows that there exists a random control function $w$ such that
  \begin{equation*}
    \sup_{(s,t) \in \Delta_T} \frac{|W_{s,t}|^p}{w(s,t)} + \sup_{n \in \N} \, \sup_{0 \leq k < \ell} \frac{|\int_{t_k^n}^{t_\ell^n} W_u^n \otimes \d W_u - W_{t_k^n} \otimes W_{t_k^n,t_\ell^n}|^{\frac{p}{2}}}{w(s,t)} \leq 1
  \end{equation*}
  holds almost surely.
\end{proof}

\begin{proof}[Proof of Proposition~\ref{prop: Levy process has RIE}]
  Let $W$ be a Brownian motion with covariance matrix $\Sigma$, and let $\phi$ be the process defined in \eqref{eq: defn phi Levy}, so that $L = W + \phi$. As usual, we let $L^n$, $W^n$ and $\phi^n$ denote the piecewise constant approximations of $L$, $W$ and $\phi$ respectively, along the adapted partition $\cP_L^n$.

  Recalling \eqref{eq: defn phi Levy}, we see that we can write $\phi = \eta + \xi$, where
  \begin{equation}\label{eq: eta expression}
    \eta_t := \lambda t + \int_{|x| \geq 2^{-n}} x \hspace{1pt} \mu(t,\d x) - t \int_{2^{-n} \leq |x| < 1} x \hspace{1pt} \nu(\d x)
  \end{equation}
  and
  \begin{equation*}
    \xi_t := \int_{|x| < 2^{-n}} x \hspace{1pt} (\mu(t,\d x) - t \hspace{1pt} \nu(\d x)).
  \end{equation*}
  Let $\eta^n$ and $\xi^n$ denote the piecewise constant approximations of $\eta$ and $\xi$ along $\cP_L^n$. Recalling how the adapted partition $\cP_L^n$ was defined in \eqref{eq: Levy partitions}, we note that, when estimating the difference $\eta^n - \eta$, we may ignore all jumps of size greater than $2^{-n}$, and may thus ignore the first integral on the right-hand side of \eqref{eq: eta expression}. We then have that
  \begin{equation}\label{eq: eta^n - eta bound}
    \begin{split}
    \|\eta^n - \eta\|_\infty &\leq 2^{-n} T |\lambda| + 2^{-n} T \int_{2^{-n} \leq |x| < 1} |x| \hspace{1pt} \nu(\d x)\\
    &\leq 2^{-n} T |\lambda| + 2^{-n(2-q)} T \int_{2^{-n} \leq |x| < 1} |x|^q \hspace{1pt} \nu(\d x) \lesssim 2^{-n(2-q)}.
    \end{split}
  \end{equation}

  Writing $\langle \cdot \rangle$ for the predictable quadratic variation, we have (see, e.g., \cite[Chapter~2, Theorem~1.33]{Jacod2003}) that
  \begin{equation*}
    \E [\langle \xi \rangle_T] \leq T \int_{|x| < 2^{-n}} |x|^2 \hspace{1pt} \nu(\d x) \leq 2^{-n(2-q)} T \int_{|x| < 2^{-n}} |x|^q \hspace{1pt} \nu(\d x).
  \end{equation*}
  Since this quantity is finite, the process $\xi$ is a square integrable martingale, and in particular $\E [[\xi]_T] = \E [\langle \xi \rangle_T]$, where $[\cdot]$ denotes the usual quadratic variation. By the Burkholder--Davis--Gundy inequality, we then have that
  \begin{equation}\label{eq: E xi^2 bound}
    \E [\|\xi\|_\infty^2] \lesssim \E [[\xi]_T] = \E [\langle \xi \rangle_T] \lesssim 2^{-n(2-q)}.
  \end{equation}

  Note that, for any $a > 0$, if $\|\xi\|_\infty < \frac{a}{2}$, then $\|\xi^n - \xi\|_\infty < a$. It follows that, for any $\delta \in (0,1-\frac{q}{2})$,
  \begin{equation*}
    \P(\|\xi^n - \xi\|_\infty \geq 2^{-n\delta}) \leq \P(\|\xi\|_\infty \geq 2^{-1 - n\delta}).
  \end{equation*}
  By Markov's inequality and the bound in \eqref{eq: E xi^2 bound}, we see that
  \begin{equation*}
    \P(\|\xi^n - \xi\|_\infty \geq 2^{-n\delta}) \lesssim 2^{2 - n(2-q-2\delta)},
  \end{equation*}
  and the Borel--Cantelli lemma then implies that, almost surely,
  \begin{equation}\label{eq: xi^n - xi bound}
    \|\xi^n - \xi\|_\infty \lesssim 2^{-n\delta},
  \end{equation}
  where the implicit multiplicative constant is a random variable which does not depend on $n$. It follows from \eqref{eq: eta^n - eta bound} and \eqref{eq: xi^n - xi bound} that
  \begin{equation}\label{eq: phi^n - phi bound}
    \|\phi^n - \phi\|_\infty \lesssim 2^{-n\delta}.
  \end{equation}

  Let $p' \in (p,3)$ and $q' \in (q,2)$ such that $\frac{1}{p'} + \frac{1}{q'} > 1$. Using interpolation, the fact that $\sup_{n \in \N} \|\phi^n\|_q \leq \|\phi\|_q$, and the bound in \eqref{eq: phi^n - phi bound}, we have that, almost surely,
  \begin{equation}\label{eq: phi^n - phi q' bound}
    \|\phi^n - \phi\|_{q'} \leq \|\phi^n - \phi\|_\infty^{1-\frac{q}{q'}} \|\phi^n - \phi\|_q^{\frac{q}{q'}} \lesssim \|\phi^n - \phi\|_\infty^{1-\frac{q}{q'}} \lesssim 2^{-n\delta(1-\frac{q}{q'})}.
  \end{equation}
  We also have from Lemma~\ref{lemma: BM satisfies RIE wrt jump times of Levy process} that, for almost every $\omega \in \Omega$, the sample path $W(\omega)$ satisfies Property \textup{(RIE)} relative to $p$ and $(\cP_L^n(\omega))_{n \in \N}$. Thus, by Proposition~\ref{prop: stability of RIE}, for almost every $\omega \in \Omega$, the sample path $L(\omega)$ satisfies Property \textup{(RIE)} relative to $p$ and $(\cP^n(\omega))_{n \in \N}$, which establishes part~(i).

  Since the sample paths of $W$ are almost surely $\frac{1}{p}$-H{\"o}lder continuous, it is straightforward to see that
  \begin{equation*}
    \|W^n - W\|_\infty \lesssim 2^{-\frac{n}{p}},
  \end{equation*}
  where the implicit multiplicative constant depends on the (random) H\"older constant of the path. Since $L = W + \phi$, we have that
  \begin{equation*}
    \|L^n - L\|_\infty \leq \|W^n - W\|_\infty + \|\phi^n - \phi\|_\infty \lesssim 2^{-\frac{n}{p}} + 2^{-n\delta}.
  \end{equation*}

  We recall from \eqref{eq: Borel Cantelli bound for BM part} that
  \begin{equation*}
    \bigg\|\int_0^\cdot W^n_u \otimes \d W_u - \int_0^\cdot W_u \otimes \d W_u\bigg\|_\infty \lesssim 2^{-\frac{n}{2}(1-\epsilon)} = 2^{-n(\frac{1}{p} - \gamma)}
  \end{equation*}
  for any $\gamma \in (0,\frac{1}{p})$. We obtained a bound for $\|\phi^n - \phi\|_{q'}$ in \eqref{eq: phi^n - phi q' bound}, and an analogous argument also shows that
  \begin{equation*}
    \|W^n - W\|_{p'} \leq \|W^n - W\|_\infty^{1-\frac{p}{p'}} \|W^n - W\|_p^{\frac{p}{p'}} \lesssim \|W^n - W\|_\infty^{1-\frac{p}{p'}} \lesssim 2^{-n(\frac{1}{p} - \frac{1}{p'})}.
  \end{equation*}
  Using the standard estimate for Young integrals (see, e.g., \cite[Proposition~2.4]{Friz2018}), similarly to the proof of Proposition~\ref{prop: stability of RIE}, we then obtain
  \begin{align*}
    &\bigg\|\int_0^\cdot L_u^n \otimes \d L_u - \int_0^\cdot L_u \otimes \d L_u\bigg\|_\infty\\
    &\lesssim \bigg\|\int_0^\cdot W_u^n \otimes \d W_u - \int_0^\cdot W_u \otimes \d W_u\bigg\|_\infty + \|W^n - W\|_{p'} \|\phi\|_q + \|\phi^n - \phi\|_{q'} (\|W\|_p + \|\phi\|_q)\\
    &\lesssim 2^{-n(\frac{1}{p} - \gamma)} + 2^{-n(\frac{1}{p} - \frac{1}{p'})} + 2^{-n\delta(1-\frac{q}{q'})}.
  \end{align*}
  Hence, by Theorem~\ref{thm: Euler scheme convergence}, we establish the estimate in part~(ii).
\end{proof}

In the following remark, we briefly discuss $\alpha$-stable L{\'e}vy processes.

\begin{remark}
Suppose now that $L$ were an $\alpha$-stable L{\'e}vy process for some $\alpha \in (0,2]$. That is, for all $a > 0$, there exists $c \in \R^d$ such that
\begin{equation*}
(L_{at})_{t \in [0,T]} \overset{d}{=} (a^{\frac{1}{\alpha}} L_t + ct)_{t \in [0,T]},
\end{equation*}
where we write $X \overset{d}{=} Y$ to mean that $X$ and $Y$ have the same distribution; see, e.g., \cite[Proposition~3.15]{Cont2004}. We now distinguish two cases:

In the case when $\alpha = 2$, $L$ is $\alpha$-stable if and only if it is Gaussian, that is, its characteristics are given by $(\lambda, \Sigma, 0)$; see, e.g., \cite[Proposition~3.15]{Cont2004}. It can thus be decomposed into the sum of a Brownian motion $W$ with covariance matrix $\Sigma$, and a linear drift term: $L_t = W_t + \lambda t$, for $t \in [0,T]$. In this case the SDE \eqref{eq:SDE} driven by $L$ can therefore be reformulated as an SDE driven by $W$ by simply absorbing the linear drift term $\lambda t$ into the drift of the SDE, and the resulting equation can then be treated as in Corollary~\ref{cor: convergence rate for Brownian motion}.

In the case when $\alpha \in (0,2)$, $L$ is $\alpha$-stable if and only if its characteristics are given by $(\lambda,0,\nu)$ (i.e., $L = \phi$ for some $\phi$ of the form in \eqref{eq: defn phi Levy}), and there exists a finite measure $\rho$ on $S$, a unit sphere on $\R^d$, such that
\begin{equation*}
\nu(B) = \int_S \int_0^\infty \1_B(r\xi) \hspace{1pt} \frac{\d r}{r^{1 + \alpha}} \hspace{1pt} \rho(\d \xi)
\end{equation*}
for all Borel sets $B$ on $\R^d$; see, e.g., \cite[Proposition~3.15]{Cont2004}.

We then have that $\int_{|x|< 1} |x|^q \hspace{1pt} \nu(\d x) < \infty$ for $q > \alpha$, and in particular that almost all sample paths of $L$ are of finite $q$-variation for $q \in (\alpha,2)$ if $\alpha \in [1,2)$, and are of finite $1$-variation if $\alpha < 1$. This then fits into the setting of Proposition~\ref{prop: Levy process has RIE}, and, since there is no Gaussian term, the resulting error estimate for the associated Euler scheme reduces to
\begin{equation*}
\|Y^n - Y\|_{p'} \leq C (2^{-n(1-\frac{1}{q'})} + 2^{-n\delta(1-\frac{q}{q'})(1-\frac{p}{p'})}), \qquad n \in \N,
\end{equation*}
Of course, in this case it is not necessary to utilize the rough path framework, since the integral $\int_0^t \sigma(s,Y_{s-}) \dd L_s$ in \eqref{eq:SDE} can be defined as a pathwise Young integral, and by discretizing this integral one could derive pathwise results using stability estimates for Young integrals.
\end{remark}

\subsection{C{\`a}dl{\`a}g semimartingales}

In this section, we consider the case when $X$ is a general c{\`a}dl{\`a}g semimartingale. As noted in Remark~\ref{remark: must include jump times}, to hope for pointwise convergence of the Euler scheme, we need to ensure that the sequence of partitions exhausts all the jump times of $X$. With this in mind, for each $n \in \N$, we introduce the stopping times $(\tau^n_k)_{k \in \N \cup \{0\}}$, such that $\tau^n_0 = 0$, and
\begin{equation}\label{eq: defn stopping times semimart}
  \tau^n_k = \inf \{t > \tau^n_{k-1} : \, |t - \tau^n_{k-1}| + |X_t - X_{\tau^n_{k-1}}| \geq 2^{-n}\} \wedge T, \qquad k \in \N.
\end{equation}
We then define a sequence of adapted partitions $(\cP_X^n)_{n \in \N}$ by
\begin{equation*}
  \cP_X^n = \{\tau^n_k : k \in \N \cup \{0\}\}.
\end{equation*}
Note that, for almost every $\omega \in \Omega$, $(\cP_X^n(\omega))_{n \in \N}$ is a sequence of (finite) partitions with vanishing mesh size. The next result verifies that $X$ satisfies Property \textup{(RIE)} relative to any $p \in (2,3)$ and $(\cP^n_X)_{n \in \N}$, and establishes the rate of convergence of the associated Euler scheme.

\begin{proposition}
  Let $p \in (2,3)$, and let $X$ be a $d$-dimensional c{\`a}dl{\`a}g semimartingale. Let $Y$ be the solution of the SDE \eqref{eq:SDE} driven by $X$, and let $Y^n$ be the corresponding Euler approximation along $\cP_X^n$, as defined in \eqref{eq:RDE Euler scheme}.
  \begin{itemize}
    \item[(i)] For almost every $\omega \in \Omega$, the sample path $X(\omega)$ satisfies Property \textup{(RIE)} relative to $p$ and $(\cP_X^n(\omega))_{n \in \N}$.
    \item[(ii)] For any $p' \in (p,3)$ and $q \in (1,2)$ such that $\frac{1}{p'} + \frac{1}{q} > 1$, and any $\epsilon \in (0,1)$, there exists a random variable $C$, which does not depend on $n$, such that
    \begin{equation}\label{eq: semimart p' est}
      \|Y^n - Y\|_{p'} \leq C (2^{-n (1 - \frac{1}{q})} + 2^{-n (1-\epsilon) (1 - \frac{p}{p'})}), \qquad n \in \N,
    \end{equation}
    and
    \begin{equation}\label{eq: semimart 3 est}
      \|Y^n - Y\|_3 \leq C 2^{-n (\frac{1}{3}-\epsilon)}, \qquad n \in \N.
    \end{equation}
  \end{itemize}
\end{proposition}

\begin{proof}
  \emph{(i):}
  The proof is just a slight modification of the proof of \cite[Proposition~4.1]{Allan2023b}, and is therefore omitted here for brevity. It is actually slightly easier, as here we do not require the sequence of partitions to be nested, and the sequence of stopping times in \eqref{eq: defn stopping times semimart} is constructed to ensure that the mesh size vanishes, even if $X$ exhibits intervals of constancy.

  \emph{(ii):} 
  By the definition of the partition $\cP_X^n$, it is clear that
  \begin{equation*}
    \|X^n - X\|_\infty \leq 2^{-n}.
  \end{equation*} 
  By an application of the Burkholder--Davis--Gundy inequality and the Borel--Cantelli lemma, as in the proof of \cite[Proposition~3.4]{Liu2018}, one can show that
  \begin{equation*}
    \bigg\|\int_0^{\cdot} X^n_{u-} \otimes \d X_u - \int_0^{\cdot} X_{u-} \otimes \d X_u\bigg\|_\infty \lesssim 2^{-n(1-\epsilon)}, \qquad n \in \N,
  \end{equation*}
  where the implicit multiplicative constant is a random variable which does not depend on $n$.

  It thus follows from Theorem~\ref{thm: Euler scheme convergence} that
  \begin{equation*}
    \|Y^n - Y\|_3 \leq \|Y^n - Y\|_{p'} \lesssim  2^{-n(1 - \frac{1}{q})} + 2^{-n(1 - \frac{p}{p'})} + 2^{-n(1 - \epsilon)(1 - \frac{p}{p'})},
  \end{equation*}  
  which leads to \eqref{eq: semimart p' est}. Choosing $p$ sufficiently close to $2$, $p'$ to $3$, and $q$ to $\frac{3}{2}$, and replacing $\epsilon$ by $3\epsilon$, then reveals \eqref{eq: semimart 3 est}.
\end{proof}

\section{Applications to differential equations driven by non-semimartingales}\label{sec:RSDE}

While in the previous section we considered SDEs driven by various classes of semimartingales, like the general theory of rough paths, the deterministic theory developed in Section~\ref{sec:RDE} is not limited to the semimartingale framework. In this section we investigate Property \textup{(RIE)} in the context of ``mixed'' and ``rough'' SDEs. The main insight is again that the random driving signals of these equations do, indeed, satisfy Property \textup{(RIE)} and, thus, the pathwise convergence results regarding the Euler scheme, as presented in Theorem~\ref{thm: Euler scheme convergence} and Proposition~\ref{prop:approximate Euler scheme}, are applicable.

Further examples of stochastic processes which fulfill Property \textup{(RIE)} almost surely include \emph{$p$-semimartingales} (also known as \emph{Young semimartingales}) in the sense of Norvai\v{s}a~\cite{Norvaivsa2003}, as well as \emph{typical price paths} in the sense of Vovk, relative to suitable sequences of adapted partitions. The pathwise convergence of the Euler scheme is thus immediately applicable to differential equations driven by such $p$-semimartingales \cite{Kubilius2002} and typical price paths \cite{Bartl2019a}.

\subsection{Mixed stochastic differential equations}

Differential equations driven by both a Brownian motion as well as a fractional Brownian motion with Hurst parameter $H > \frac{1}{2}$ are classical objects in stochastic analysis; see, e.g., \cite{Zahle2001,Mishura2011}. More precisely, a ``mixed'' stochastic differential equation (mixed SDE) is given by
\begin{equation}\label{eq: mixed SDE}
  Y_t = y_0 + \int_0^t b(s,Y_s) \dd s + \int_0^t \sigma_1(s,Y_s) \dd W_s + \int_0^t \sigma_2(s,Y_s) \dd W^H_s, \qquad t \in [0,T],
\end{equation}
where $b \in C^2_b(\R^{k+1};\R^k)$, $\sigma_1 \in C^3_b(\R^{k+1};\mathcal{L}(\R^{d_1};\R^k))$, $\sigma_2 \in C^3_b(\R^{k+1};\mathcal{L}(\R^{d_2};\R^k))$ and $y_0 \in \R^k$. Here, $W$ is a $d_1$-dimensional standard Brownian motion, and $W^H$ is a $d_2$-dimensional fractional Brownian motion with Hurst parameter $H > \frac{1}{2}$, which are independent and both defined on a filtered probability space $(\Omega,\cF,(\cF_t)_{t \in [0,T]},\P)$ satisfying the usual conditions.

The mixed SDE \eqref{eq: mixed SDE} lies outside the semimartingale framework, but there are various ways to provide a rigorous meaning to its solution. Here we consider the mixed SDE \eqref{eq: mixed SDE} as a random RDE, driven by the It\^o rough path lift of $(W,W^H)$, the existence of which follows from Lemma~\ref{lemma: W, W^H satisfies RIE} below. In particular, it then follows from Theorem~\ref{thm: RDE} that there exists a unique solution $Y$ to \eqref{eq: mixed SDE}.

\begin{lemma}\label{lemma: W, W^H satisfies RIE}
  Let $W$ be a standard Brownian motion, and let $W^H$ be a fractional Brownian motion with Hurst parameter $H \in (\frac{1}{2},1)$. Let $p \in (2,3)$ such that $\frac{1}{p} + H > 1$, and let $\cP^n = \{0 = t_0^n < t_1^n < \cdots < t_{N_n}^n = T\}$, $n \in \N$, be a sequence of equidistant partitions of the interval $[0,T]$, so that, for each $n \in \N$, there exists some $\pi_n > 0$ such that $t^n_{i+1} - t^n_i = \pi_n$ for each $0 \leq i < N_n$. If $\pi_n^{2 - \frac{4}{p}} \log(n) \to 0$ as $n \to \infty$, then, for almost every $\omega \in \Omega$, the sample path $(W(\omega),W^H(\omega))$ satisfies Property \textup{(RIE)} relative to $p$ and $(\cP^n)_{n \in \N}$.
\end{lemma}

\begin{proof}
  We first note that the process $(W,0)$ satisfies the hypotheses of Theorem~\ref{prop: BM satisfies RIE}, and thus that almost all of its sample paths satisfy Property \textup{(RIE)} relative to $p$ and $(\cP^n)_{n \in \N}$. Let $\frac{1}{H} < q < q' < 2$ such that $\frac{1}{p} + \frac{1}{q'} > 1$. Since $\frac{1}{q} < H$, it is well known that the sample paths of $(0,W^H)$ are almost surely $\frac{1}{q}$-H{\"o}lder continuous, and hence that $\|W^H\|_q < \infty$. Writing $W^{H,n}$ for the usual piecewise constant approximation of $W^H$ along $\cP^n$, we have by interpolation that
  \begin{equation*}
    \|W^{H,n} - W^H\|_{q'} \leq \|W^{H,n} - W^H\|_\infty^{1 - \frac{q}{q'}} \|W^{H,n} - W^H\|_q^{\frac{q}{q'}} \lesssim \|W^{H,n} - W^H\|_\infty^{1 - \frac{q}{q'}} \, \longrightarrow \, 0
  \end{equation*}
  as $n \to \infty$. The result then follows by applying Proposition~\ref{prop: stability of RIE} to $(W,0) + (0,W^H)$.
\end{proof}

Of course, since here we consider Hurst parameters $H > \frac{1}{2}$, the trajectories of $W^H$ have in particular finite $q$-variation for any $q \in (\frac{1}{H},2)$, so we could alternatively define the integral $\int_0^t \sigma_2(s,Y_s) \dd W^H_s$ in \eqref{eq: mixed SDE} as a pathwise Young integral, and by discretizing this integral one could in principle derive analogous pathwise convergence results; cf.~Remark~\ref{remark: absorb young part into drift instead}.

\subsection{Rough stochastic differential equations}

Rough stochastic differential equations (rough SDEs) are differential equations driven by both a rough path and a semimartingale. These equations first appeared in the context of robust stochastic filtering---see \cite{Crisan2013,Diehl2015}---and were recently studied in a general form in \cite{Friz2021}. In this section we will adapt the setting of \cite{Diehl2015}, which allows to treat H\"older continuous rough paths and Brownian motion as driving signals.

\medskip

We let $\eta \colon [0,T] \to \R^d$ be a deterministic path which is $\frac{1}{p}$-H{\"o}lder continuous for some $p \in (2,3)$, and which satisfies Property \textup{(RIE)} relative to $p$ and the dyadic partitions $(\cP_D^n)_{n \in \N}$, as defined in \eqref{eq: dyadic partitions}. We write $\boldsymbol{\eta} = (\boldsymbol{\eta}^1,\boldsymbol{\eta}^2)$ for the canonical rough path lift of $\eta$, with $\boldsymbol{\eta}^2$ defined as in \eqref{eq:RIE rough path}, so that $\boldsymbol{\eta}^2_{s,t} = \int_s^t \eta_{s,u} \otimes \d \eta_u$ for each $(s,t) \in \Delta_T$. We also let $W$ be an $\R^e$-valued Brownian motion. For vector fields $a \in C^2_b(\R^k;\R^k)$, $b \in C^3_b(\R^k;\cL(\R^d;\R^k))$ and $c \in C^3_b(\R^k;\cL(\R^e;\R^k))$, and an initial value $y_0 \in \R^k$, we then consider the rough SDE
\begin{equation}\label{eq: rough SDE}
  Y_t = y_0 + \int_0^t a(Y_s) \dd s + \int_0^t b(Y_s) \dd \boldsymbol{\eta}_s + \int_0^t c(Y_s) \dd W_s, \qquad t \in [0,T].
\end{equation}

To give a rigorous meaning to the rough SDE \eqref{eq: rough SDE}, following the method introduced in \cite{Diehl2015}, we need to construct a suitable joint rough path lift $\boldsymbol{\Lambda}(\omega)$ above the $\R^{d+e}$-valued path $(\eta,W(\omega))$ for almost every $\omega \in \Omega$. Indeed, the (pathwise) unique solution to the random RDE
\begin{equation*}
  Y_t = y_0 + \int_0^t a(Y_s) \dd s + \int_0^t (b,c)(Y_s) \dd \boldsymbol{\Lambda}_s, \qquad t \in [0,T],
\end{equation*}
is then defined to be the solution to the rough SDE \eqref{eq: rough SDE}.

\medskip

To construct the It\^o rough path lift of $(\eta,W)$, we need the existence of the quadratic covariation of $\eta$ and $W$ along the dyadic partitions. More precisely, writing $\cP_D^n = \{0 = t^n_0 < t^n_1 < \cdots < t^n_{2^n} = T\}$ with $t^n_k = k 2^{-n} T$, we need to establish that, for almost every $\omega \in \Omega$, the limit
\begin{equation}\label{eq: eta W quad covar}
  \langle \eta, W(\omega) \rangle_t := \lim_{n \to \infty} \sum_{k=0}^{2^n - 1} \eta_{t^n_k \wedge t,t^n_{k+1} \wedge t} \otimes W_{t^n_k \wedge t,t^n_{k+1} \wedge t}(\omega)
\end{equation}
exists and holds uniformly for $t \in [0,T]$.

\begin{lemma}\label{lemma: the existence of quadratic covariation}
  Let $\alpha \in (0,1]$, let $\eta \colon [0,T] \to \R$ be an $\alpha$-H\"older continuous deterministic path, and let $W$ be a one-dimensional Brownian motion. Then, for almost every $\omega \in \Omega$, the quadratic covariation of $\eta$ and $W(\omega)$ along the dyadic partitions, in the sense of \eqref{eq: eta W quad covar}, exists, and satisfies $\langle \eta, W(\omega) \rangle_t = 0$ for all $t \in [0,T]$.
\end{lemma}

\begin{proof}
We consider the discrete-time martingale given by $t \mapsto \sum_{k \hspace{1pt} : \hspace{1pt} t^n_{k+1} \leq t} \eta_{t^n_k,t^n_{k+1}} W_{t^n_k,t^n_{k+1}}$ for $t \in \cP_D^n$, for some fixed $n \in \N$. By the Burkholder--Davis--Gundy inequality, we have that
  \begin{align*}
    \E \bigg[\bigg\|\sum_{k \hspace{1pt} : \hspace{1pt} t^n_{k+1} \leq \cdot} \eta_{t^n_k,t^n_{k+1}} W_{t^n_k,t^n_{k+1}}\bigg\|_\infty^2\bigg] &\lesssim \E \bigg[\sum_{k=0}^{2^n - 1} (\eta_{t^n_k,t^n_{k+1}} W_{t^n_k,t^n_{k+1}})^2\bigg] = \sum_{k=0}^{2^n - 1} (\eta_{t^n_k,t^n_{k+1}})^2 (t^n_{k+1} - t^n_k)\\
    &\lesssim \sum_{k=0}^{2^n - 1} (t^n_{k+1} - t^n_k)^{1 + 2\alpha} \lesssim (2^{-n} T)^{2\alpha} \sum_{k=0}^{2^n - 1} (t^n_{k+1} - t^n_k) \lesssim 2^{-2n\alpha}.
  \end{align*}
  For any $\epsilon \in (0,1)$, we then have, by Markov's inequality, that
  \begin{equation*}
    \P \bigg(\bigg\|\sum_{k \hspace{1pt} : \hspace{1pt} t^n_{k+1} \leq \cdot} \eta_{t^n_k,t^n_{k+1}} W_{t^n_k,t^n_{k+1}}\bigg\|_\infty \geq 2^{-n \alpha (1-\epsilon)}\bigg) \lesssim 2^{-2n \alpha \epsilon},
  \end{equation*}
  and the Borel--Cantelli lemma then implies that
  \begin{equation*}
    \bigg\|\sum_{k \hspace{1pt} : \hspace{1pt} t^n_{k+1} \leq \cdot} \eta_{t^n_k,t^n_{k+1}} W_{t^n_k,t^n_{k+1}}\bigg\|_\infty \lesssim 2^{-n \alpha (1-\epsilon)},
  \end{equation*}
  where the implicit multiplicative constant is a random variable which does not depend on $n$.

  For a given $t \in [0,T]$ and $n \in \N$, let $k_0$ be such that $t \in [t^n_{k_0},t^n_{k_0+1}]$. Since $\eta$ is $\alpha$-H\"older continuous, and the sample paths of $W$ are almost surely $\beta$-H\"older continuous for any $\beta \in (0,\frac{1}{2})$, we have that
  \begin{equation*}
    |\eta_{t^n_{k_0},t} W_{t^n_{k_0},t}| \lesssim (t - t^n_{k_0})^{\alpha + \beta} \lesssim 2^{-n (\alpha + \beta)}.
  \end{equation*}
  We thus have the bound
  \begin{align*}
    \bigg|\sum_{k=0}^{2^n - 1} \eta_{t^n_k \wedge t,t^n_{k+1} \wedge t} W_{t^n_k \wedge t,t^n_{k+1} \wedge t}\bigg| &\leq \bigg|\sum_{k \hspace{1pt} : \hspace{1pt} t^n_{k+1} \leq t} \eta_{t^n_k,t^n_{k+1}} W_{t^n_k,t^n_{k+1}}\bigg| + |\eta_{t^n_{k_0},t} W_{t^n_{k_0},t}|\\
    &\lesssim 2^{-n \alpha (1-\epsilon)} + 2^{-n (\alpha + \beta)},
  \end{align*}
  where the implicit multiplicative constant is a random variable which does not depend on $t$ or $n$. It follows that, almost surely,
  \begin{equation*}
    \sum_{k=0}^{2^n - 1} \eta_{t^n_k \wedge t,t^n_{k+1} \wedge t} W_{t^n_k \wedge t,t^n_{k+1} \wedge t} \, \longrightarrow \, 0 \qquad \text{as} \quad n \, \longrightarrow \, \infty,
  \end{equation*}
  uniformly for $t \in [0,T]$.
\end{proof}

It is shown in \cite[Theorem~1]{Diehl2015}, with integrals defined in the Stratonovich sense, that an analogous object to the process $\boldsymbol{\Lambda}$ described in \eqref{eq: Lambda^2 representation} below provides a geometric rough path lift of $(\eta,W)$. In the next theorem we establish that $\boldsymbol{\Lambda}$ is the It\^o rough path lift of $(\eta,W)$, and moreover that it may be obtained as the canonical lift via Property \textup{(RIE)}, thus making our convergence analysis of the Euler scheme applicable to the rough SDE \eqref{eq: rough SDE}.

\begin{theorem}\label{theorem: joint lift RIE}
  Let $p \in (2,3)$. Let $\eta$ be a $\frac{1}{p}$-H\"older continuous $\R^d$-valued path which satisfies Property \textup{(RIE)} relative to $p$ and the sequence of dyadic partitions $(\cP_D^n)_{n \in \N}$, and write $\boldsymbol{\eta} = (\boldsymbol{\eta}^{1},\boldsymbol{\eta}^{2})$ for the canonical rough path lift of $\eta$, so that $\boldsymbol{\eta}^{1} = \eta$, and $\boldsymbol{\eta}^{2}_{s,t} = \int_s^t \eta_{s,u} \otimes \d \eta_u$, defined as in \eqref{eq:RIE rough path}, for every $(s,t) \in \Delta_T$. Let $W$ be an $\R^e$-valued Brownian motion, and write $\bW = (W,\W)$ for the It\^o rough path lift of $W$, so that $\W_{s,t} = \int_s^t W_{s,u} \otimes \d W_u$, defined as an It\^o integral, for every $(s,t) \in \Delta_T$.

  For any $p' \in (p,3)$ and almost every $\omega \in \Omega$, the $\R^{d+e}$-valued path $(\eta,W(\omega))$ satisfies Property \textup{(RIE)} relative to $p'$ and $(\cP_D^n)_{n \in \N}$.

  Moreover, for almost every $\omega \in \Omega$, the canonical rough path lift $\boldsymbol{\Lambda}(\omega) = (\boldsymbol{\Lambda}^{1}(\omega),\boldsymbol{\Lambda}^{2}(\omega)) \in \R^{d+e} \oplus \R^{(d+e) \times (d+e)}$ of $(\eta,W(\omega))$ (constructed via Property \textup{(RIE)} as in \eqref{eq:RIE rough path}) is given by $\boldsymbol{\Lambda}^{1}(\omega) = (\eta,W(\omega))$, and
  \begin{equation}\label{eq: Lambda^2 representation}
    \boldsymbol{\Lambda}^2_{s,t} = \bigg(\begin{array}{cc}
    \boldsymbol{\eta}^2_{s,t} & \int_s^t \eta_{s,u} \otimes \dd W_u\\
    W_{s,t} \otimes \eta_{s,t} - (\int_s^t \eta_{s,u} \otimes \dd W_u)^\top & \W_{s,t}
    \end{array}\bigg)
  \end{equation}
  for every $(s,t) \in \Delta_T$, where $\int_s^t \eta_{s,u} \otimes \d W_u$ is defined as an It\^o integral, and $(\cdot)^\top$ denotes matrix transposition.
\end{theorem}

\begin{proof}
  Let $p' \in (p,3)$. It follows from the Kolmogorov criterion for rough paths (see \cite[Theorem~3.1]{Friz2020}) that, for almost every $\omega \in \Omega$,
  \begin{equation}\label{eq: int eta dW Holder bound}
    \bigg|\bigg(\int_s^t \eta_{s,u} \otimes \d W_u\bigg)(\omega)\bigg| \lesssim |t - s|^{\frac{2}{p'}} \qquad \text{for all} \quad (s,t) \in \Delta_T,
  \end{equation}
  and moreover that $\boldsymbol{\Lambda}(\omega) = (\boldsymbol{\Lambda}^1(\omega),\boldsymbol{\Lambda}^2(\omega))$ is a $\frac{1}{p'}$-H\"older continuous rough path. We will show that $(\eta,W(\omega))$ satisfies Property \textup{(RIE)}, and that the associated canonical rough path is indeed given by $\boldsymbol{\Lambda}(\omega)$.

  \emph{Step 1.}
  As usual, we let $\eta^n$ and $W^n$ denote the piecewise constant approximations of $\eta$ and $W$ respectively, along $\cP_D^n$. By assumption, $\eta$ satisfies Property \textup{(RIE)} relative to $p$ and $(\cP_D^n)_{n \in \N}$. By Proposition~\ref{prop: BM satisfies RIE} (or Proposition~\ref{prop: Ito diffusion has RIE}), for almost every $\omega \in \Omega$, the sample path $W(\omega)$ also satisfies Property \textup{(RIE)} relative to $p$ and $(\cP_D^n)_{n \in \N}$.

  It follows from the first condition in Property \textup{(RIE)} for $\eta$ and $W(\omega)$ that, for almost every $\omega \in \Omega$,
  \begin{equation*}
    (\eta^n,W^n(\omega)) \, \longrightarrow \, (\eta,W(\omega)) \qquad \text{uniformly as} \quad n \, \longrightarrow \, \infty,
  \end{equation*}
  so that this condition also holds for the pair $(\eta,W(\omega))$. Moreover, it follows from the second condition in Property \textup{(RIE)} that $\int_0^\cdot \eta^n_u \otimes \d \eta_u$ converges uniformly to $\int_0^\cdot \eta_u \otimes \d \eta_u$, and, for almost every $\omega \in \Omega$, that $(\int_0^\cdot W^n_u \otimes \d W_u)(\omega)$ converges uniformly to $(\int_0^\cdot W_u \otimes \d W_u)(\omega)$.

  By the Burkholder--Davis--Gundy inequality, and the observation that $\|\eta^n - \eta\|_\infty \lesssim 2^{-\frac{n}{p}}$, we have that
  \begin{equation*}
    \E \bigg[\bigg\|\int_0^\cdot \eta^n_u \otimes \d W_u - \int_0^\cdot \eta_u \otimes \d W_u\bigg\|_\infty^2\bigg] \lesssim \E \bigg[\int_0^T |\eta^n_u - \eta_u|^2 \dd u\bigg] \lesssim 2^{-\frac{2n}{p}}.
  \end{equation*}
  For any $\epsilon \in (1 - \frac{2}{p},1)$, it then follows from Markov's inequality that
  \begin{equation*}
    \P \bigg(\bigg\|\int_0^\cdot \eta^n_u \otimes \d W_u - \int_0^\cdot \eta_u \otimes \d W_u\bigg\|_\infty \geq 2^{-\frac{n}{2}(1-\epsilon)}\bigg) \lesssim 2^{n(1 - \frac{2}{p} - \epsilon)}.
  \end{equation*}
  The Borel--Cantelli lemma then implies that, for almost every $\omega \in \Omega$,
  \begin{equation}\label{eq: Borel-Cantelli bound joint lift}
    \bigg\|\bigg(\int_0^\cdot \eta^n_u \otimes \d W_u - \int_0^\cdot \eta_u \otimes \d W_u\bigg)(\omega)\bigg\|_\infty \lesssim 2^{-\frac{n}{2}(1-\epsilon)}
  \end{equation}
  for all $n \in \N$, and in particular that $(\int_0^\cdot \eta^n_u \otimes \d W_u)(\omega)$ converges uniformly to $(\int_0^\cdot \eta_u \otimes \d W_u)(\omega)$ as $n \to \infty$.

  Let us write $\cP_D^n = \{0 = t^n_0 < t^n_1 < \cdots < t^n_{2^n} = T\}$ for $n \in \N$, where $t^n_k = k 2^{-n} T$. It is straightforward to verify that, for any $t \in [0,T]$,
  \begin{equation*}
    W_t \otimes \eta_t = \int_0^t W^n_u \otimes \d \eta_u + \bigg(\int_0^t \eta^n_u \otimes \d W_u\bigg)^{\hspace{-3pt}\top} + \langle W, \eta \rangle^n_t,
  \end{equation*}
  where, by Lemma~\ref{lemma: the existence of quadratic covariation}, the discrete quadratic variation $\langle W, \eta \rangle^n_t := \sum_{k=0}^{2^n - 1} W_{t^n_k \wedge t,t^n_{k+1} \wedge t} \otimes \eta_{t^n_k \wedge t,t^n_{k+1} \wedge t}$ almost surely converges uniformly to $\langle W, \eta \rangle_t = 0$ as $n \to \infty$. We then see that, for almost every $\omega \in \Omega$,
  \begin{equation*}
    \int_0^t W^n_u(\omega) \otimes \d \eta_u \, \longrightarrow \, W_t(\omega) \otimes \eta_t - \bigg(\int_0^t \eta_u \otimes \d W_u\bigg)^{\hspace{-3pt}\top}(\omega)
  \end{equation*}
  as $n \to \infty$, uniformly in $t \in [0,T]$. We have thus established that, for almost every $\omega \in \Omega$, the path $(\eta,W(\omega))$ also satisfies the second condition of Property \textup{(RIE)}, and moreover that the resulting canonical rough path is indeed given by \eqref{eq: Lambda^2 representation}.

  \emph{Step 2.} It remains to show that $(\eta,W(\omega))$ satisfies the third condition of Property \textup{(RIE)} relative to $p'$ and $(\cP_D^n)_{n \in \N}$.

  Since $\eta$ satisfies Property \textup{(RIE)} relative to $p$ and $(\cP_D^n)_{n \in \N}$, there exists a control function $w_\eta$ such that
  \begin{equation}\label{eq: RIE for eta}
    \sup_{(s,t) \in \Delta_T} \frac{|\eta_{s,t}|^p}{w_\eta(s,t)} + \sup_{n \in \N} \, \sup_{0 \leq k < \ell \leq 2^n} \frac{|\int_{t^n_k}^{t^n_\ell} \eta^n_u \otimes \d \eta_u - \eta_{t^n_k} \otimes \eta_{t^n_k,t^n_\ell}|^{\frac{p}{2}}}{w_\eta(t^n_k,t^n_\ell)} \leq 1,
  \end{equation}
  which implies that the same inequality also holds with $p$ replaced by $p'$ (possibly with a different control function, but without loss of generality we may assume that $w_\eta$ remains valid for $p'$). Similarly, since for almost every $\omega \in \Omega$ the sample path $W(\omega)$ satisfies Property \textup{(RIE)} relative to $p$ (and therefore also to $p'$) and $(\cP_D^n)_{n \in \N}$, there exists a control function $c$ such that
  \begin{equation}\label{eq: RIE for W}
    \sup_{(s,t) \in \Delta_T} \frac{|W_{s,t}(\omega)|^{p'}}{c(s,t)} + \sup_{n \in \N} \, \sup_{0 \leq k < \ell \leq 2^n} \frac{|(\int_{t^n_k}^{t^n_\ell} W^n_u \otimes \d W_u - W_{t^n_k} \otimes W_{t^n_k,t^n_\ell})(\omega)|^{\frac{p'}{2}}}{c(t^n_k,t^n_\ell)} \leq 1.
  \end{equation}

  \emph{Step 3.}
  Let $\beta \in (0,\frac{1}{2})$. Since $\eta$ is $\frac{1}{p}$-H\"older continuous, and the sample paths of $W$ are almost surely $\beta$-H\"older continuous, we have that
  \begin{equation*}
    |\eta_{t^n_{i-1}} \otimes W_{t^n_{i-1},t^n_i} + \eta_{t^n_i} \otimes W_{t^n_i,t^n_{i+1}} - \eta_{t^n_{i-1}} \otimes W_{t^n_{i-1},t^n_{i+1}}| = |\eta_{t^n_{i-1},t^n_i} \otimes W_{t^n_i,t^n_{i+1}}| \lesssim |t^n_{i+1} - t^n_{i-1}|^{\frac{1}{p} + \beta}
  \end{equation*}
  for any $i = 1, \ldots, N_n-1$, where the implicit multiplicative constant is a random variable, and we can follow the proof of \cite[Lemma~3.2]{Liu2018} to deduce that, for almost any fixed $\omega \in \Omega$, for any $k < \ell$, and writing $N = \ell - k = 2^n |t^n_\ell - t^n_k| T^{-1}$,
  \begin{equation*}
    \bigg|\bigg(\int_{t^n_k}^{t^n_\ell} \eta^n_u \otimes \d W_u\bigg)(\omega) - \eta_{t^n_k} \otimes W_{t^n_k,t^n_\ell}(\omega)\bigg| \lesssim N^{1 - \frac{2}{\rho}} |t^n_\ell - t^n_k|^{\frac{2}{\rho}} \lesssim 2^{n(1 - \frac{2}{\rho})} |t^n_\ell - t^n_k|,
  \end{equation*}
  where $\frac{2}{\rho} = \frac{1}{p} + \beta$.

  Let $\epsilon \in (1 - \frac{2}{p},1)$. If $2^{-n} \geq |t^n_\ell - t^n_k|^{\frac{4}{p(1-\epsilon)}}$, then
  \begin{align*}
    \bigg|\bigg(\int_{t^n_k}^{t^n_\ell} \eta^n_u \otimes \d W_u\bigg)(\omega) - \eta_{t^n_k} \otimes W_{t^n_k,t^n_\ell}(\omega)\bigg| \lesssim |t^n_\ell - t^n_k|^{1 - \frac{4}{p(1-\epsilon)} (1 - \frac{2}{\rho})}.
  \end{align*}
  By choosing $\epsilon$ close to $1 - \frac{2}{p}$, we can make the above exponent $1 - \frac{4}{p(1-\epsilon)} (1 - \frac{2}{\rho})$ arbitrarily close to $\frac{4}{\rho} - 1 = \frac{2}{p} + 2\beta - 1$. By then choosing $\beta$ close to $\frac{1}{2}$, we can make this value arbitrarily close to $\frac{2}{p}$ from below. In particular, by making suitable choices of $\epsilon$ and $\beta$, we can ensure that $1 - \frac{4}{p(1-\epsilon)} (1 - \frac{2}{\rho}) = \frac{2}{p'}$, and we obtain
  \begin{equation}\label{eq: first estimate for joint lift}
    \bigg|\bigg(\int_{t^n_k}^{t^n_\ell} \eta^n_u \otimes \d W_u\bigg)(\omega) - \eta_{t^n_k} \otimes W_{t^n_k,t^n_\ell}(\omega)\bigg| \lesssim |t^n_\ell - t^n_k|^{\frac{2}{p'}}.
  \end{equation}

  We will now aim to obtain the same estimate in the case that $2^{-n} < |t^n_\ell - t^n_k|^{\frac{4}{p(1-\epsilon)}}$, with $\epsilon$ chosen as above. Recalling \eqref{eq: int eta dW Holder bound} and \eqref{eq: Borel-Cantelli bound joint lift}, we have that
  \begin{align*}
    \bigg|&\bigg(\int_{t^n_k}^{t^n_\ell} \eta^n_u \otimes \d W_u\bigg)(\omega) - \eta_{t^n_k} \otimes W_{t^n_k,t^n_\ell}(\omega)\bigg|\\
    &= \bigg|\bigg(\int_{t^n_k}^{t^n_\ell} \eta^n_u \otimes \d W_u\bigg)(\omega) - \bigg(\int_{t^n_k}^{t^n_\ell} \eta_u \otimes \d W_u\bigg)(\omega) + \bigg(\int_{t^n_k}^{t^n_\ell} \eta_u \otimes \d W_u\bigg)(\omega) - \eta_{t^n_k} \otimes W_{t^n_k,t^n_\ell}(\omega)\bigg|\\
    &\leq 2 \bigg\|\bigg(\int_0^\cdot \eta^n_u \otimes \d W - \int_0^\cdot \eta_u \otimes \d W\bigg)(\omega)\bigg\|_\infty + \bigg|\bigg(\int_{t^n_k}^{t^n_\ell} \eta_{t^n_k,u} \otimes \d W_u\bigg)(\omega)\bigg|\\
    &\lesssim 2^{-\frac{n}{2}(1 - \epsilon)} + |t^n_\ell - t^n_k|^{\frac{2}{p'}}\\
    &\lesssim |t^n_\ell - t^n_k|^{\frac{2}{p'}}.
  \end{align*}
  Combining this with \eqref{eq: first estimate for joint lift}, we conclude that
  \begin{equation}\label{eq: int eta dW RIE bound}
    \sup_{n \in \N} \, \sup_{0 \leq k < \ell \leq 2^n} \frac{|(\int_{t^n_k}^{t^n_\ell} \eta^n_u \otimes \d W_u)(\omega) - \eta_{t^n_k} \otimes W_{t^n_k,t^n_\ell}(\omega)|^{\frac{p'}{2}}}{C(\omega)|t^n_\ell - t^n_k|} \leq 1,
  \end{equation}
  for a suitable random variable $C$.

  \emph{Step 4.}
  For any $n \in \N$ and $0 \leq k < \ell \leq 2^n$, it is straightforward to verify that
  \begin{equation*}
    |\eta_{t^n_k,t^n_\ell}|^2 = 2 \int_{t^n_k}^{t^n_\ell} \eta^n_{t^n_k,u} \cdot \d \eta_u + \sum_{i=k}^{\ell-1} |\eta_{t^n_i,t^n_{i+1}}|^2, 
  \end{equation*}
  where $\cdot$ denotes the Euclidean inner product. It follows from \eqref{eq: RIE for eta} that $|\eta_{t^n_k,t^n_\ell}|^2 \lesssim w_\eta(t^n_k,t^n_\ell)^{\frac{2}{p'}}$, and that
  \begin{equation*}
    \sup_{n \in \N} \, \sup_{0 \leq k < \ell \leq 2^n} \frac{|\int_{t^n_k}^{t^n_\ell} \eta^n_{t^n_k,u} \cdot \d \eta_u|^{\frac{p'}{2}}}{w_\eta(t^n_k,t^n_\ell)} \lesssim 1,
  \end{equation*}
  from which we then have that
  \begin{equation*}
    \sup_{n \in \N} \, \sup_{0 \leq k < \ell \leq 2^n} \frac{|\sum_{i=k}^{\ell-1} |\eta_{t^n_i,t^n_{i+1}}|^2|^{\frac{p'}{2}}}{w_\eta(t^n_k,t^n_\ell)} \lesssim 1.
  \end{equation*}
  The same argument holds for the sample paths of $W$, and since
  \begin{equation*}
    \bigg|\sum_{i=k}^{\ell-1} W_{t^n_i,t^n_{i+1}} \otimes \eta_{t^n_i,t^n_{i+1}}\bigg| \lesssim \sum_{i=k}^{\ell-1} |W_{t^n_i,t^n_{i+1}}|^2 + \sum_{i=k}^{\ell-1} |\eta_{t^n_i,t^n_{i+1}}|^2,
  \end{equation*}
  we deduce that
  \begin{equation}\label{eq: sum W otimes eta RIE bound}
    \sup_{n \in \N} \, \sup_{0 \leq k < \ell \leq 2^n} \frac{|\sum_{i=k}^{\ell-1} W_{t^n_i,t^n_{i+1}} \otimes \eta_{t^n_i,t^n_{i+1}}|^{\frac{p'}{2}}}{w_\eta(t^n_k,t^n_\ell) + c(t^n_k,t^n_\ell)} \lesssim 1.
  \end{equation}

  By the H\"older continuity of $\eta$ and $W$, it is clear that $|W_{t^n_k,t^n_\ell} \otimes \eta_{t^n_k,t^n_\ell}| \lesssim |t^n_\ell - t^n_k|^{\frac{2}{p'}}$, so that
  \begin{equation}\label{eq: W otimes eta RIE bound}
    \sup_{n \in \N} \, \sup_{0 \leq k < \ell \leq 2^n} \frac{|W_{t^n_k,t^n_\ell} \otimes \eta_{t^n_k,t^n_\ell}|^{\frac{p'}{2}}}{|t^n_\ell - t^n_k|} \lesssim 1.
  \end{equation}

  For any $n \in \N$ and $0 \leq k < \ell \leq 2^n$, it is straightforward to verify that
  \begin{equation*}
    W_{t^n_k,t^n_\ell} \otimes \eta_{t^n_k,t^n_\ell} = \int_{t^n_k}^{t^n_\ell} W^n_{t^n_k,u} \otimes \d \eta_u + \bigg(\int_{t^n_k}^{t^n_\ell} \eta^n_{t^n_k,u} \otimes \d W_u\bigg)^{\hspace{-3pt}\top} + \sum_{i=k}^{\ell-1} W_{t^n_i,t^n_{i+1}} \otimes \eta_{t^n_i,t^n_{i+1}}.
  \end{equation*}
  Recalling \eqref{eq: int eta dW RIE bound}, \eqref{eq: sum W otimes eta RIE bound} and \eqref{eq: W otimes eta RIE bound}, we thus have that
  \begin{equation*}
    \sup_{n \in \N} \, \sup_{0 \leq k < \ell \leq 2^n} \frac{|\int_{t^n_k}^{t^n_\ell} W^n_{t^n_k,u} \otimes \d \eta_u|^{\frac{p'}{2}}}{\hat{w}(t^n_k,t^n_\ell)} \leq 1
  \end{equation*}
  for a suitable random control function $\hat{w}$. Combining this with \eqref{eq: RIE for eta}, \eqref{eq: RIE for W} and \eqref{eq: int eta dW RIE bound}, we conclude that, for almost every $\omega \in \Omega$, the path $(\eta,W(\omega))$ indeed satisfies the third condition of Property \textup{(RIE)}.
\end{proof}

\begin{remark}
  A joint rough path lift of $(\eta,W)$ is constructed in \cite[Section~2]{Diehl2015} which allows \eqref{eq: rough SDE} to be treated as a rough Stratonovich SDE. Since the construction of the joint lift $\boldsymbol{\Lambda}$ above is based on a piecewise constant approximation, as in Property \textup{(RIE)}, rather than on linear interpolations as considered in \cite{Diehl2015}, Theorem~\ref{theorem: joint lift RIE} provides a joint It\^o-type rough path lift of $(\eta,W)$ and, thus, an It\^o interpretation of the rough SDE \eqref{eq: rough SDE}, consistent with that in \cite{Friz2021}.
\end{remark}

\appendix

\section{Proof of Theorem~\ref{thm: RDE}}\label{appendix proof of RDE thm}

\begin{proof}[Proof of Theorem~\ref{thm: RDE}]
  \emph{Step 1.} Let $L > 0$ such that $\|A\|_r, \|H\|_r, \|\bX\|_p \leq L$, and let $w \colon \Delta_T \to [0,\infty)$ be the right-continuous control function given by
  \begin{equation*}
    w(s,t) = \|A\|_{r,[s,t]}^{r} + \|H\|_{r,[s,t]}^{r} + \|X\|_{p,[s,t]}^p + \|\X\|_{\frac{p}{2},[s,t]}^{\frac{p}{2}}, \qquad \text{for} \quad (s,t) \in \Delta_T.
  \end{equation*}
  For $t \in (0,T]$, we define the map $\cM_t \colon \cV^{q,r}_X([0,t];\R^k) \to \cV^{q,r}_X([0,t];\R^k)$ by
  \begin{equation*}
    \cM_t(Y,Y') = \bigg(y_0 + \int_0^\cdot b(H_s,Y_s) \dd A_s + \int_0^\cdot \sigma(H_s,Y_s) \dd \bX_s,\sigma(H,Y)\bigg),
  \end{equation*}
  and, for $\delta \geq 1$, introduce the subset of controlled paths
  \begin{equation*}
    \cB^{(\delta)}_t = \Big\{(Y,Y') \in \cV^{q,r}_X([0,t];\R^k) : (Y_0,Y'_0) = (y_0,\sigma(H_0,y_0)), \, \|Y,Y'\|^{(\delta)}_{X,q,r} \leq 1\Big\},
  \end{equation*}
  where
  \begin{equation*}
    \|Y,Y'\|^{(\delta)}_{X,q,r} := \|Y'\|_{q,[0,t]} + \delta \|R^Y\|_{r,[0,t]}.
  \end{equation*}

  Applying standard estimates for Young and rough integrals (e.g., \cite[Proposition~2.4 and Lemma~3.6]{Friz2018}), for any $(Y,Y') \in \cB^{(\delta)}_t$, we deduce that
  \begin{equation*}
    \|\cM_t(Y,Y')\|^{(\delta)}_{X,q,r} \leq C_1 \bigg(\frac{1}{\delta} + \delta (\|A\|_{r,[0,t]} + \|H\|_{r,[0,t]} + \|\bX\|_{p,[0,t]})\bigg),
  \end{equation*}
  for a constant $C_1 \geq \frac{1}{2}$ which depends only on $p, q, r, \|b\|_{C_b^2}, \|\sigma\|_{C_b^3}$, and $L$. Let $\delta = \delta_1 := 2 C_1$, so that
  \begin{equation*}
    \|\cM_t(Y,Y')\|^{(\delta_1)}_{X,q,r} \leq \frac{1}{2} + 2 C_1^2 (2w(0,t)^{\frac{1}{r}} + w(0,t)^{\frac{1}{p}} + w(0,t)^{\frac{2}{p}}).
  \end{equation*}
  By the right-continuity of $w$, we can then take $t = t_1$ sufficiently small such that
  \begin{equation*}
    \|\cM_{t_1}(Y,Y')\|^{(\delta_1)}_{X,q,r} \leq 1,
  \end{equation*}
  and we have that $\cB^{(\delta_1)}_{t_1}$ is invariant under $\cM_{t_1}$.

  \emph{Step 2.} Let $(Y,Y'), (\tY,\tY') \in \cB^{(\delta)}_t$, for some (new) $\delta \geq 1$ and $t \in (0,t_1]$. Applying standard estimates for Young and rough integrals (e.g., \cite[Proposition~2.4, Lemma~3.1 and Lemma~3.7]{Friz2018}), we deduce that
  \begin{align*}
    &\|\cM_t(Y,Y') - \cM_t(\tY,\tY')\|^{(\delta)}_{X,q,r}\\
    &\quad\leq C_2 \Big(\|R^Y - R^{\tY}\|_{r,[0,t]} + \delta (\|Y' - \tY'\|_{q,[0,t]} + \|R^Y - R^{\tY}\|_{r,[0,t]}) (\|A\|_{r,[0,t]} + \|\bX\|_{p,[0,t]})\Big),
  \end{align*}
  where $C_2 > \frac{1}{2}$ depends only on $p, q, r, \|b\|_{C_b^2}, \|\sigma\|_{C_b^3}$ and $L$. Let $\delta = \delta_2 := 2 C_2 > 1$, so that
  \begin{align*}
	&\|\cM_t(Y,Y') - \cM_t(\tY,\tY')\|^{(\delta_2)}_{X,q,r} \\
    &\quad\leq \frac{\delta_2}{2} \|R^Y - R^{\tY}\|_{r,[0,t]}\\
	&\quad \quad+ 2 C_2^2 (\|Y' - \tY'\|_{q,[0,t]} + \|R^Y - R^{\tY}\|_{r,[0,t]}) (w(0,t)^{\frac{1}{r}} + w(0,t)^{\frac{1}{p}} + w(0,t)^{\frac{2}{p}}).
  \end{align*}
  Again by the right-continuity of $w$, we then take $t = t_2 \leq t_1$ sufficiently small such that
  \begin{align*}
	\|\cM_{t_2}(Y,Y') - \cM_{t_2}(\tY,\tY')\|^{(\delta_2)}_{X,q,r} 
	&\leq \frac{1}{2} \|Y' - \tY'\|_{q,[0,t_2]} + \frac{\delta_2 + 1}{2} \|R^Y - R^{\tY}\|_{r,[0,t_2]}\\
	&\leq \frac{\delta_2 + 1}{2 \delta_2} \|(Y,Y') - (\tY,\tY')\|^{(\delta_2)}_{X,q,r},
  \end{align*}
  from which it follows that $\cM_{t_2}$ is a contraction on the Banach space $(\cB^{(\delta_1)}_{t_2}, \|\cdot\|^{(\delta_2)}_{X,q,r})$. The fixed point of this map is the unique solution of the RDE~\eqref{eq:general RDE} over the time interval $[0,t_2]$.

  \emph{Step 3.} Now let $\tA \in D^{q_1}$, $\tH \in D^{q_2}$, $\tbX = (\tX,\tbbX) \in \cD^p$ and $\ty_0 \in \R^n$, such that $\|\tA\|_r, \|\tH\|_r, \|\tbX\|_p \leq L$. By considering instead the control function $w$ given by
  \begin{align*}
	  w(s,t) &= \|A\|_{r,[s,t]}^r + \|H\|_{r,[s,t]}^r + \|X\|_{p,[s,t]}^p + \|\X\|_{\frac{p}{2},[s,t]}^{\frac{p}{2}}\\
	&\quad + \|\tA\|_{r,[s,t]}^r + \|\tH\|_{r,[s,t]}^r + \|\tX\|_{p,[s,t]}^p + \|\tbbX\|_{\frac{p}{2},[s,t]}^{\frac{p}{2}}, \qquad \text{for} \quad (s,t) \in \Delta_T,
  \end{align*}
  it follows from the above that there exist unique solutions $(Y,Y') \in \cV^{q,r}_X([0,t_2];\R^k)$ and $(\tY,\tY') \in \cV^{q,r}_{\tX}([0,t_2];\R^k)$ of the RDE~\eqref{eq:general RDE}, with data $(A,H,\bX,y_0)$ and $(\tA,\tH,\tbX,\ty_0)$ respectively, over a sufficiently small time interval $[0,t_2]$. Standard estimates for Young and rough integrals (e.g., \cite[Proposition~2.4, Lemma~3.1 and Lemma~3.7]{Friz2018}) imply, after some calculation, that for any $\delta \geq 1$ and $t \in (0,t_2]$,
  \begin{align*}
	&\|Y' - \tY'\|_{q,[0,t]} + \delta \|R^Y - R^{\tY}\|_{r,[0,t]}\\
	&\quad\leq C_3 \Big(|y_0 - \ty_0| + |H_0 - \tH_0| + \|H - \tH\|_{r,[0,t]} + \|R^Y - R^{\tY}\|_{r,[0,t]}\\
	&\qquad + \delta (\|A - \tA\|_{r,[0,t]} + \|\bX;\tbX\|_{p,[0,t]})\\
	&\qquad + \delta (|y_0 - \ty_0| + |H_0 - \tH_0| + \|H - \tH\|_{r,[0,t]} + \|Y' - \tY'\|_{q,[0,t]} + \|R^Y - R^{\tY}\|_{r,[0,t]})\\
	&\qquad \quad \times (\|A\|_{r,[0,t]} + \|\bX\|_{p,[0,t]})\Big),
  \end{align*}
  where $C_3 > 0$ depends only on $p, q, r, \|b\|_{C_b^2}, \|\sigma\|_{C_b^3}$ and $L$. Let $\delta = \delta_3 := C_3 + 1$, so that
  \begin{align*}
	&\|Y' - \tY'\|_{q,[0,t]} + \|R^Y - R^{\tY}\|_{r,[0,t]}\\
	&\quad\leq C_3 \Big(|y_0 - \ty_0| + |H_0 - \tH_0| + \|H - \tH\|_{r,[0,t]} + \delta_3 (\|A - \tA\|_{r,[0,t]} + \|\bX;\tbX\|_{p,[0,t]})\\
	&\qquad + \delta_3 (|y_0 - \ty_0| + |H_0 - \tH_0| + \|H - \tH\|_{r,[0,t]} + \|Y' - \tY'\|_{q,[0,t]} + \|R^Y - R^{\tY}\|_{r,[0,t]})\\
	&\qquad \quad \times (w(0,t)^{\frac{1}{r}} + w(0,t)^{\frac{1}{p}} + w(0,t)^{\frac{2}{p}})\Big).
  \end{align*}
  By taking $t = t_3 \leq t_2$ sufficiently small, we deduce that
  \begin{equation}\label{eq:local estimate for RDE}
    \begin{split}
	&\|Y - \tY\|_{p,[0,t_3]} + \|Y' - \tY'\|_{q,[0,t_3]} + \|R^Y - R^{\tY}\|_{r,[0,t_3]}\\
	&\quad\leq C_4 \Big(|y_0 - \ty_0| + |H_0 - \tH_0| + \|H - \tH\|_{r,[0,t_3]} + \|A - \tA\|_{r,[0,t_3]} + \|\bX;\tbX\|_{p,[0,t_3]}\Big),
    \end{split}
  \end{equation} 
  for a new constant $C_4$, still depending only on $p, q, r, \|b\|_{C_b^2}, \|\sigma\|_{C_b^3}$ and $L$.

  \emph{Step 4.} We infer from the above that there exists a constant $\epsilon > 0$, which depends only on $p, q, r, \|b\|_{C_b^2}, \|\sigma\|_{C_b^3}$ and $L$, such that, given initial values $Y_s, \tY_s \in \R^k$, the local solutions $(Y,Y')$ and $(\tY,\tY')$ established above exist on any interval $[s,t]$ such that $w(s,t) \leq \epsilon$. Moreover, these local solutions satisfy an estimate on this interval of the form in \eqref{eq:local estimate for RDE}.

  By \cite[Lemma~1.5]{Friz2018}, there exists a partition $\cP = \{0 = t_0 < t_1 < \cdots < t_N = T\}$, such that $w(t_i,t_{i+1}-) < \epsilon$ for every $i = 0, 1, \ldots, N-1$. We can then define the solutions $(Y,Y')$ and $(\tY,\tY')$ on each of the half-open intervals $[t_i,t_{i+1})$. Given the solutions on $[t_i,t_{i+1})$, the values $Y_{t_{i+1}}$ and $\tY_{t_{i+1}}$ at the right end-point of the interval are uniquely determined by the jumps of $A, \tA, \bX$ and $\tbX$ at time $t_{i+1}$. We thus deduce the existence of unique solutions $(Y,Y')$ and $(\tY,\tY')$ of the RDE on the entire interval $[0,T]$.

  Since $w$ is superadditive, we have that
  \begin{equation*}
    w(t_0,t_1-) + w(t_1-,t_1) + w(t_1,t_2-) + \cdots + w(t_{N-1},t_{N}-) + w(t_{N}-,t_N) \leq w(0,T).
  \end{equation*}
  It is then straightforward to see that the partition $\cP$ may be chosen such that the number of partition points in $\cP$ may be bounded by a constant depending only on $\epsilon$ and $w(0,T)$. Thus, we may combine the local estimates in \eqref{eq:local estimate for RDE} on each of the subintervals, together with simple estimates on the jumps at the end-points of these subintervals, to obtain the global estimate in \eqref{eq:estimate for RDE}.
\end{proof}

\section{The convergence of piecewise constant approximations}

In the following, we adopt the notation
\begin{equation*}
  \liminf_{n \to \infty} \cP^n := \bigcup_{m \in \N} \, \bigcap_{n \geq m} \cP^n
\end{equation*}
for the times $t \in [0,T]$ which, as $n \to \infty$, eventually belong to all subsequent partitions in the sequence $(\cP^n)_{n \in \N}$. The following proposition generalizes the result of \cite[Proposition~2.14]{Allan2023b} so that the sequence of partitions is no longer assumed to be nested.

\begin{proposition}\label{prop: uniform convergence}
  Let $\mathcal{P}^n = \{0 = t^n_0 < t^n_1 < \cdots < t^n_{N_n} = T\}$, $n \in \N$, be a sequence of partitions with vanishing mesh size, so that $|\mathcal{P}^n| \to 0$ as $n \to \infty$. Let $F \colon [0,T] \to \R^d$ be a c{\`a}dl{\`a}g path, and let
  \begin{equation*}
    F^n_t = F_T \1_{\{T\}}(t) + \sum_{k=0}^{N_n - 1} F_{t^n_k} \1_{[t^n_k,t^n_{k+1})}(t), \qquad t \in [0,T],
  \end{equation*}
  be the piecewise constant approximation of $F$ along $\cP^n$. Let
  \begin{equation*}
    J_F := \{t \in (0,T] : F_{t-} \neq F_t\}
  \end{equation*}
  be the set of jump times of $F$. The following are equivalent:
  \begin{enumerate}
    \item[(i)] $J_F \subseteq \liminf_{n \to \infty} \mathcal{P}^n$,
    \item[(ii)] the sequence $(F^n)_{n \in \N}$ converges pointwise to $F$,
    \item[(iii)] the sequence $(F^n)_{n \in \N}$ converges uniformly to $F$.
  \end{enumerate}
\end{proposition}

\begin{proof}
  We first show that conditions (i) and (ii) are equivalent. To this end, suppose that $J_F \subseteq \liminf_{n \to \infty} \mathcal{P}^n$ and let $t \in (0,T]$. If $t \in J_F$, then there exists $m \geq 1$ such that $t \in \mathcal{P}^n$ for all $n \geq m$. In this case we then have that $F^n_t = F_t$ for all $n \geq m$. If $t \notin J_F$, then $F$ is continuous at time $t$, and, since the mesh size $|\mathcal{P}^n| \to 0$, it follows that $F^n_t \to F_t$ as $n \to \infty$.

  Now suppose instead that there exists a $t \in J_F$ such that $t \notin \liminf_{n \to \infty} \mathcal{P}^n$. Then there exists a subsequence $(n_j)_{j \in \N}$ such that $F^{n_j}_t \to F_{t-}$ as $j \to \infty$. Since $F_{t- }\neq F_t$, it follows that $F^n_t \nrightarrow F_t$. This establishes the equivalence of (i) and (ii).

Since (iii) clearly implies (ii), it only remains to show that (ii) implies (iii). By \cite[Theorem~3.3]{Frankova2019}, it is enough to show that the family of paths $\{F^n : n \in \N\}$ is equiregulated in the sense of \cite[Definition~3.1]{Frankova2019}. That is, we need to show that, for every $t \in (0,T]$ and $\epsilon > 0$, there exists a $u \in [0,t)$ such that $|F^n_s - F^n_{t-}| < \epsilon$ for every $s \in (u,t)$ and every $n \in \N$, and moreover that for every $t \in [0,T)$ and $\epsilon > 0$, there exists a $u \in (t,T]$ such that $|F^n_s - F^n_t| < \epsilon$ for every $s \in (t,u)$ and every $n \in \N$.

  \textit{Step~1.} Let $t \in (0,T]$ and $\epsilon > 0$. Since the left limit $F_{t-}$ exists, there exists $\delta > 0$ with $t - \delta > 0$, such that
  \begin{equation*}
    |F_s - F_{t-}| < \frac{\epsilon}{2} \qquad \text{for all} \quad s \in (t - \delta,t).
  \end{equation*}

  Since $|\mathcal{P}^n| \to 0$ as $n \to \infty$, there exists an $m \in \N$ such that, for every $n \geq m$, there exists a partition point $t^n_k \in \cP^n$ such that $t - \delta < t^n_k < t - \frac{\delta}{2}$.

  Let
  \begin{equation*}
    u := \max \bigg(\Big(t - \frac{\delta}{2},t\Big) \cap \bigcup_{n < m} \cP^n\bigg),
  \end{equation*}
  where here we define $\max (\emptyset) := t - \frac{\delta}{2}$.

  Take any $s \in (u,t)$ and any $n \in \N$. Let $i = \max \{k : t^n_k \leq s\}$ and $j = \max \{k : t^n_k < t\}$, so that $F^n_s = F_{t^n_i}$ and $F^n_{t-} = F_{t^n_j}$.

  If $n \geq m$, then there exists a point $t^n_k \in \cP^n$ such that $t - \delta < t^n_k < t - \frac{\delta}{2} \leq u < s$, and it follows that $t^n_i, t^n_j \in (t - \delta,t)$. If instead $n < m$, and if there exists a partition point $t^n_k \in (t - \frac{\delta}{2},t)$, then $t - \frac{\delta}{2} < t^n_k \leq u < s$, and it again follows that $t^n_i, t^n_j \in (t - \delta,t)$. In either case, we then have that
  \begin{equation*}
    |F^n_s - F^n_{t-}| = |F_{t^n_i} - F_{t^n_j}| \leq |F_{t^n_i} - F_{t-}| + |F_{t^n_j} - F_{t-}| < \frac{\epsilon}{2} + \frac{\epsilon}{2} = \epsilon.
  \end{equation*}
  The remaining case is when $n < m$ but $(t - \frac{\delta}{2},t) \cap \cP^n = \emptyset$. In this case the path $F^n$ is constant on the interval $[t - \frac{\delta}{2},t)$ and, since $s \in (t - \frac{\delta}{2},t)$, we have that $F^n_s = F^n_{t-}$.

  In each case, we have that $|F^n_s - F^n_{t-}| < \epsilon$ for all $s \in (u,t)$ and all $n \in \N$.

  \textit{Step~2.}
  Let $t \in (J_F \cup \{0\}) \setminus \{T\}$ and $\epsilon > 0$. Since $F$ is right-continuous, there exists a $\delta > 0$ with $t + \delta < T$, such that 
  \begin{equation*}
    |F_s - F_t| < \epsilon \qquad \text{for all} \quad s \in [t,t + \delta).
  \end{equation*}
  Since condition~(ii) implies condition~(i), we know that $t \in \liminf_{n \to \infty} \mathcal{P}^n$, so that there exists an $m \in \N$ such that $t \in \cap_{n \geq m} \cP^n$. Let
  \begin{equation*}
    u := \min \bigg((t,t + \delta) \cap \bigcup_{n < m} \cP^n\bigg),
  \end{equation*}
  where here we define $\min (\emptyset) := t + \delta$.

  Take any $s \in (t,u)$, and any $n \in \N$. Let $i = \max \{k : t^n_k \leq s\}$, so that $F^n_s = F_{t^n_i}$.

  If $n \geq m$, then $t \in \cP^n$, so $F^n_t = F_t$ and, moreover, $t \leq t^n_i \leq s < u \leq t + \delta$, so that in particular $t^n_i \in [t,t + \delta)$, and hence
  \begin{equation*}
    |F^n_s - F^n_t| = |F_{t^n_i} - F_t| < \epsilon.
  \end{equation*}
  If $n < m$, then there does not exist any partition point $t^n_k \in (t,u) \cap \cP^n$. It follows that the path $F^n$ is constant on the interval $[t,u)$, so that in particular $F^n_s = F^n_t$.

  In each case, we have that $|F^n_s - F^n_t| < \epsilon$ for all $s \in (t,u)$ and all $n \in \N$.

  \textit{Step~3.} Let $t \in (0,T) \setminus J_F$ and $\epsilon > 0$. Since $F$ is continuous at time $t$, there exists a $\delta > 0$ with $0 < t - \delta$ and $t + \delta < T$, such that
  \begin{equation*}
    |F_s - F_t| < \frac{\epsilon}{2} \qquad \text{for all} \quad s \in (t - \delta,t + \delta).
  \end{equation*}
  Since $|\mathcal{P}^n| \to 0$ as $n \to \infty$, there exists an $m \in \N$ such that, for every $n \geq m$, there exists a partition point $t^n_k \in \cP^n$ such that $t - \delta < t^n_k < t$. Let
  \begin{equation*}
    u := \min \bigg((t,t + \delta) \cap \bigcup_{n < m} \cP^n\bigg),
  \end{equation*}
  where here we define $\min (\emptyset) := t + \delta$.

  Take any $s \in (t,u)$ and any $n \in \N$. Let $i = \max \{k : t^n_k \leq s\}$ and $j = \max \{k : t^n_k \leq t\}$, so that $F^n_s = F_{t^n_i}$ and $F^n_t = F_{t^n_j}$.

  If $n \geq m$, then there exists a point $t^n_k \in \cP^n$ such that $t^n_k \in (t - \delta,t)$, and it follows that $t^n_i, t^n_j \in (t - \delta,t + \delta)$, so that
  \begin{equation*}
    |F^n_s - F^n_t| = |F_{t^n_i} - F_{t^n_j}| \leq |F_{t^n_i} - F_t| + |F_{t^n_j} - F_t| < \frac{\epsilon}{2} + \frac{\epsilon}{2} = \epsilon.
  \end{equation*}
  If $n < m$, then there does not exist any partition point $t^n_k \in (t,u) \cap \cP^n$. It follows that the path $F^n$ is constant on the interval $[t,u)$, so that in particular $F^n_s = F^n_t$.

  In each case, we have that $|F^n_s - F^n_t| < \epsilon$ for all $s \in (t,u)$ and all $n \in \N$. It follows that the family of paths $\{F^n : n \in \N\}$ is indeed equiregulated.
\end{proof}

\begin{theorem}\label{thm:rough int under RIE}
  Let $p \in (2,3)$, $q \in [p,\infty)$ and $r \in [\frac{p}{2},2)$ such that $\frac{1}{p} + \frac{1}{r} > 1$ and $\frac{1}{p} + \frac{1}{q} = \frac{1}{r}$, and let $\cP^n = \{0 = t^n_0 < t^n_1 < \cdots < t^n_{N_n} = T\}$, $n \in \N$, be a sequence of partitions with vanishing mesh size. Suppose that $X$ satisfies Property \textup{(RIE)} relative to $p$ and $(\cP^n)_{n \in \N}$, and let $\bX$ be the canonical rough path lift of $X$, as constructed in \eqref{eq:RIE rough path}. Let $(F,F') \in \cV_X^{q,r}$ be a controlled path with respect to $X$, and suppose that $J_F \subseteq \liminf_{n \to \infty} \cP^n$, where $J_F$ is the set of jump times of $F$. Then the rough integral of $(F,F')$ against $\bX$ is given by
  \begin{equation*}
    \int_0^t F_u \dd \bX_u = \lim_{n \to \infty} \sum_{k=0}^{N_n-1} F_{t^n_k} X_{t^n_k \wedge t,t^n_{k+1} \wedge t},
  \end{equation*}
  where the convergence is uniform in $t \in [0,T]$.
\end{theorem}

The previous theorem generalizes the result of \cite[Theorem~2.15]{Allan2023b} so that the sequence of partitions is no longer assumed to be nested. The proof of Theorem~\ref{thm:rough int under RIE} follows the proof of \cite[Theorem~2.15]{Allan2023b} almost verbatim. The only difference is that, rather than using \cite[Proposition~2.14]{Allan2023b} to establish the uniform convergence of $F^n$ to $F$, we can instead use Proposition~\ref{prop: uniform convergence} (which does not require the sequence of partitions to be nested).

\bibliography{quellen}{}
\bibliographystyle{amsalpha}

\end{document}